\documentclass[12pt]{amsart}
\usepackage{amssymb}
\address{Department of Mathematics, Northeastern University, 360 Huntington Avenue,
Boston MA 02115, USA}
\email{i.loseu@neu.edu}
\newcommand{\OCat}{\mathcal{O}}
\newcommand{\Cat}{\mathcal{C}}
\newcommand{\gr}{\operatorname{gr}}
\newcommand{\K}{\mathbb{K}}
\newcommand{\bDelta}{\overline{\Delta}}
\newcommand{\bnabla}{\overline{\nabla}}
\newcommand{\Ext}{\operatorname{Ext}}

\newcommand{\Hom}{\operatorname{Hom}}
\newcommand{\Loc}{\operatorname{Loc}}

\newcommand{\A}{\mathcal{A}}
\newcommand{\Ca}{\mathsf{C}}
\newcommand{\Z}{\mathbb{Z}}
\newcommand{\C}{\mathbb{C}}
\newcommand{\R}{\mathbb{R}}
\newcommand{\CW}{\mathfrak{CW}}
\newcommand{\WC}{\mathfrak{WC}}
\newcommand{\p}{\mathfrak{p}}
\newcommand{\param}{\mathfrak{p}}
\newcommand{\Str}{\mathcal{O}}
\newcommand{\Coh}{\operatorname{Coh}}
\newcommand{\B}{\mathcal{B}}

\newcommand{\Tcal}{\mathcal{T}}
\newcommand{\g}{\mathfrak{g}}
\newcommand{\Supp}{\operatorname{Supp}}
\newcommand{\quo}{/\!/}
\newcommand{\VA}{\operatorname{V}}
\newtheorem{Thm}{Theorem}[section]
\newtheorem{Prop}[Thm]{Proposition}
\newtheorem{Cor}[Thm]{Corollary}
\newtheorem{Lem}[Thm]{Lemma}
\theoremstyle{definition}

\newtheorem{Rem}[Thm]{Remark}

\numberwithin{equation}{section}
\title{On categories $\mathcal{O}$
for quantized symplectic resolutions}
\author{Ivan Losev}
\begin{document}
\begin{abstract}
In this paper we study categories $\mathcal{O}$ over quantizations of symplectic resolutions
admitting Hamiltonian tori actions with finitely many fixed points. In this generality,
these categories were introduced by Braden, Licata, Proudfoot and Webster. We establish
a family of standardly stratified structures (in the sense of the author and Webster)
on these categories $\mathcal{O}$. We use these structures to study shuffling functors of Braden, Licata,
Proudfoot and Webster (called  cross-walling functors in this paper). Most importantly,
we prove that all cross-walling functors are derived equivalences that define
an action of the Deligne groupoid of a suitable real hyperplane arrangement.
\end{abstract}
\maketitle
\tableofcontents
\section{Introduction}
This paper continues the study of categories $\mathcal{O}$ over quantizations of symplectic resolutions
started in \cite{BLPW}.

Our base field is $\C$. Let $X$ be a symplectic algebraic variety with form $\omega$. We assume that $X$
is equipped with a $\C^\times$-action that has the following two properties:
\begin{itemize}
\item The algebra $\C[X]$ is finitely generated. Furthermore, the weights of
$\C^\times$ in $\C[X]$ are non-negative, and the zero weight component
consists of scalars.
\item The torus $\C^\times$ rescales the symplectic form with a positive weight, i.e., there is
an integer $d>0$ such that $t.\omega= t^d\omega$ for all $t\in \C^\times$.
\end{itemize}
We say that $X$ is a conical symplectic resolution (of $X_0:=\operatorname{Spec}(\C[X])$)
if a natural morphism $X\rightarrow X_0$ is projective and  birational.

A classical example is as follows. Let $G$ be a semisimple algebraic group, $T\subset B\subset G$
be a maximal torus and a Borel subgroup. Set $X:=T^*(G/B)$ and equip it with the $\C^\times$-action
by fiberwise dilations. The corresponding variety $X_0$ is the nilpotent cone in $\g^*$, and the morphism
$X\rightarrow X_0$ is the Springer resolution.

To $\lambda\in H^2_{DR}(X)$, we can assign a {\it quantization}, $\A_\lambda^\theta$ of $\mathcal{O}_X$
that is a sheaf of filtered algebras with $\gr\A_\lambda^\theta=\mathcal{O}_X$ and some additional
conditions. So far $\theta$ is just an element of notation. For example, in the case $X=T^*(G/B)$,
the sheaf $\mathcal{A}_\lambda^\theta$ is essentially the sheaf of $\lambda-\rho$-twisted differential operators
on $G/B$.

Now suppose that on $X$ we have a Hamiltonian action of a torus $T$ with finitely many fixed points.
The action lifts to a Hamiltonian action on $\A_\lambda^\theta$. A one-parameter subgroup
$\nu:\C^\times\rightarrow T$ is said to be {\it generic} if the number of fixed points
for $\nu(\C^\times)$ in $X$ is finite. The set of generic one-parameter subgroups in
the complement in $\operatorname{Hom}(\C^\times,T)$ to finitely many hyperplanes (to be called
{\it walls}). The connected components to the walls in $\operatorname{Hom}(\C^\times,T)\otimes_{\Z}\mathbb{R}$
will be called {\it chambers}.

Now fix a generic one-parameter subgroup $\nu:\C^\times\rightarrow T$. Consider its attracting locus $Y$
in $X$, i.e., the set of all $x\in X$ such that the limit $\lim_{t\rightarrow 0} t.x$ exists.
Then $Y$ is the lagrangian subvariety in $X$ whose irreducible components are affine spaces
labelled by $X^T$. We define the category $\mathcal{O}_\nu(\A_\lambda^\theta)$ as the category
of all coherent $\A_\lambda^\theta$-modules supported on $Y$ that admit a weakly $\nu(\C^\times)$-equivariant
structure. This definition is equivalent to the definition appearing in \cite[Section 3.3]{BLPW}.
In the case of $X=T^*(G/B)$, we recover a version of the classical Bernstein-Gelfand-Gelfand
category $\mathcal{O}$.

One of the main results of \cite{BLPW} is that the category $\mathcal{O}_\nu(\A_\lambda^\theta)$
is highest weight with simple objects labelled by $X^T$. For an order that is a part of a highest
weight structure, one can take the usual (contraction) order $\leqslant_{\nu}$ on the fixed points produced from
$\nu$. Being highest weight with respect to a given order means certain upper triangularity properties,
just as in the BGG case.

Our first result about the categories $\mathcal{O}_\nu(\A_\lambda^\theta)$ establishes more structure
on them. Namely, take $\nu_0$ lying in the closure of the chamber containing $\nu$. Then $\nu_0$
defines a {\it pre-order} $\preceq_{\nu_0}$ on $X^T$,  by contraction of the irreducible components of $X^{\nu_0(\C^\times)}$. This pre-order is
refined by $\leqslant_{\nu}$. We show, Theorem \ref{Thm:cat_O_ss},
that the category $\mathcal{O}_\nu(\A_\lambda^\theta)$ becomes {\it standardly stratified}
in the sense of \cite{LW} with respect to $\preceq_{\nu_0}$. Roughly speaking, this means some
additional upper-triangularity properties.

The main result of this paper concerns certain derived functors introduced in \cite[Section 8.2]{BLPW}
that relate categories $\mathcal{O}_\nu(\A_\lambda^\theta)$ and $\mathcal{O}_{\nu'}(\A_\lambda^\theta)$
with \underline{different} generic $\nu,\nu'$. The functors were called shuffling in \cite{BLPW},
but in this paper we call them {\it cross-walling}, because they are supposed to be Koszul dual
to functors called wall-crossing in \cite{BL} (and twisting in \cite{BPW}). We show that
the cross-walling functor $\CW_{\nu\rightarrow \nu'}:D^b(\OCat_\nu(\A_\lambda^\theta))
\rightarrow D^b(\OCat_{\nu'}(\A_\lambda^\theta))$ is an equivalence of triangulated categories
proving a conjecture of Braden, Licata, Proudfoot and Webster.
Moreover, we show that the cross-walling functors define an action of the Deligne
groupoid of the real hyperplane arrangement given by the walls in $\operatorname{Hom}(\C^\times,T)\otimes_{\Z}\mathbb{R}$  on the categories $D^b(\OCat_?(\A_\lambda^\theta))$, see Section \ref{SS_main_CW}.
Recall that by the Deligne groupoid of a real hyperplane arrangement  in $\C^n$ one means the full
subgroupoid in the fundamental groupoid of the hyperplane complement $\mathfrak{X}$ whose objects are points
in $\mathfrak{X}\cap \mathbb{R}^n$ one per chamber.
This establishes another conjecture of Braden, Licata, Proudfoot and Webster
(there is a stronger version taking a Weyl group into account, see \cite[Proposition 8.14]{BLPW}).
We also show that both the long wall-crossing and the long cross-walling functors
realize the Ringel duality, which implies \cite[Conjecture 8.24]{BLPW}.

In order to prove these claims we examine an interplay between the cross-walling functors
and the standardly stratified structures. This interplay is of independent interest, it will be
used in a subsequent paper to provide combinatorial recipes to determine the supports
of the global sections of the  irreducible objects in $\OCat_\nu(\A_\lambda^\theta)$ in the case when $X$
is a Gieseker moduli space.

The paper is organized as follows. In Section \ref{S_sympl_quant} we will recall
some generalities on symplectic resolutions and their quantizations  following \cite{BPW}. In Section
\ref{S_SSC} we recall generalities of standardly stratified categories mostly
following  \cite{LW}. Section \ref{S_cat_O} deals with categories $\mathcal{O}$
of symplectic resolutions. Here we follow \cite{BLPW,B_ineq}. These three sections
basically contain no new results.  In Section \ref{S_parab} we introduce a ``parabolic
induction'', one of the main tools to study the categories $\mathcal{O}$.
In particular, we apply the parabolic induction in Section \ref{S_SS_O} to introduce standardly
stratified structures on those categories. In Section \ref{S_CW}
we study cross-walling functors proving all results about them mentioned before.
In the appendix we discuss some analogies between our work and that of
Maulik and Okounkov, \cite{MO}.

{\bf Acknowledgements}. I would like to thank Roman Bezrukavnikov and Andrei
Okounkov for numerous stimulating discussions. I am also grateful to Justin Hilburn and 
Nick Proudfoot for their comments on a preliminary version of this text.  My work
was partially supported by the NSF under Grant  DMS-1161584.

\section{Symplectic resolutions and their quantizations}\label{S_sympl_quant}
\subsection{Conical symplectic resolutions and their deformations}
Let $X_0$ be a normal Poisson affine variety equipped with an action of $\C^\times$
such that the grading on $\C[X_0]$ is positive (meaning that the graded component $\C[X_0]_i$
is zero when $i<0$, and $\C[X_0]_0=\C$) and there is a positive integer $d$
such that $\{\C[X_0]_i,\C[X_0]_j\}\subset \C[X_0]_{i+j-d}$. By a symplectic resolution
of singularities of $X_0$ we mean a pair $(X,\rho)$ of
\begin{itemize}
\item a symplectic algebraic variety $X$ (with form $\omega$)
\item a morphism $\rho$ of Poisson varieties that is a projective resolution of singularities.
\end{itemize}
Below we assume that $(X,\rho)$ is a symplectic resolution of singularities. Besides, we will
assume that $(X,\rho)$ is {\it conical} meaning that the $\C^\times$-action lifts to $X$
in such a way that $\rho$ is equivariant. This $\C^\times$-action will be called {\it contracting}
later on.

Note that  $\rho^*:\C[X_0]\rightarrow \C[X]$ is an isomorphism because $X_0$ is normal.
By the Grauert-Riemenschneider theorem, we have
$H^i(X,\mathcal{O}_X)=0$ for $i>0$. By results of Kaledin, \cite[Theorem 2.3]{Kaledin}, $X_0$ has finitely
many symplectic leaves.

\begin{Rem}\label{Rem:sympl_leaves_finite}
Note that if $X_0'$ is an affine Poisson variety and its normalization
$X_0$ has finitely many leaves, then so does $X_0'$. It follows that
any affine Poisson algebraic variety $X_0'$ admitting a symplectic resolution
$X$ has finitely many leaves.
\end{Rem}

We will be interested in deformations $\hat{X}/\param$, where $\param$ is a finite dimensional vector space,
and $\hat{X}$ is a symplectic scheme over $\param$ with symplectic form $\hat{\omega}\in
\Omega^2(\hat{X}/\param)$ that specializes to $\omega$ and also with a $\C^\times$-action
on $\hat{X}$ having the following properties:
\begin{itemize}
\item the morphism $\hat{X}\rightarrow \p$ is $\C^\times$-equivariant,
\item the restriction of the action to $X$ coincides with the contracting action,
\item $t.\hat{\omega}:=t^{d}\hat{\omega}$.
\end{itemize}
It turns out that there is a universal such deformation $\tilde{X}$ over $\tilde{\param}:=H^2(X,\C)$
(any other deformation is obtained via the pull-back with respect to
a linear map $\param\rightarrow \tilde{\param}$). Moreover, the deformation
$\tilde{X}$ is trivial in the category of $C^{\infty}$-manifolds. This result is
due to Namikawa, \cite[Lemmas 12,22, Proposition 13]{Namikawa08}.

Now we are going to review some structural features of conical symplectic resolutions mostly due to
Kaledin and Namikawa.

First of all, let us point out that $X$ has no odd cohomology, \cite[Proposition 2.5]{BPW}.

For $\lambda\in \tilde{\param}$, let us write $X_{\lambda}$ for the corresponding fiber of $\tilde{X}
\rightarrow \tilde{\param}$. It turns out that, for $\lambda$ Zariski generic, $X_\lambda$
is affine and independent of the choice of a conical symplectic resolution $X$.
This shows that the groups $H^i(X,\Z)$ are independent of the choice of $X$.
Furthermore, if $X,X'$ are two conical symplectic resolutions of $X$, then there are
open subvarieties $X^0\subset X, X'^0\subset X'$ with complements of codimension bigger than $1$
that are isomorphic, see, e.g., \cite[Proposition 2.19]{BPW}. This allows to identify the Picard groups $\operatorname{Pic}(X)=\operatorname{Pic}(X')$.
Moreover, the Chern class map defines an isomorphism $\C\otimes_{\Z}\operatorname{Pic}(X)\xrightarrow{\sim} H^2(X,\C)$.
The isomorphisms $\C\otimes_{\Z}\operatorname{Pic}(X)\xrightarrow{\sim} H^2(X,\C)$,
$\C\otimes_{\Z}\operatorname{Pic}(X')\xrightarrow{\sim} H^2(X',\C)$ intertwine the identifications
$\operatorname{Pic}(X)\cong \operatorname{Pic}(X'), H^2(X,\C)\cong H^2(X',\C)$.
Let $\tilde{\param}_{\Z}$ be the image of $\operatorname{Pic}(X)$ in $H^2(X,\C)$.

Set $\tilde{\param}_{\R}:=\R\otimes_{\Z}\tilde{\param}_{\Z}$. There is a finite group $W$ acting on
$\tilde{\param}_{\R}$ as a reflection group, such that the movable cone $C$ of $X$ (that does not
depend on the choice of a resolution) is a fundamental chamber for $W$. We can partition $C$
into the union of cones $C_1,\ldots,C_m$ (with disjoint interiors)
such that the possible conical symplectic resolutions of $X_0$ are in one-to-one correspondence with
$C_1,\ldots,C_m$ in such a way that the cone corresponding
to a resolution $X'$ is its  ample cone. Let $H_1,\ldots,H_k$ be the hyperplanes spanned by
the codimension $1$ faces of the cones $C_1,\ldots,C_m$. Set
\begin{equation}\label{eq:singular} \mathcal{H}_{\C}:=\bigcup_{1\leqslant i\leqslant k,w\in W}w(\C\otimes_{\R}H^i).\end{equation} Then $X_\lambda$ is affine if and only if $\lambda\not\in \mathcal{H}_{\C}$.
The results in this paragraph are due to Namikawa, \cite{Namikawa13}.

For $\theta\not\in \mathcal{H}_{\C}$, we will write $X^\theta$ for the resolution corresponding to the ample cone
containing $W\theta\cap C$. Further, if $w\theta\in C$, we will choose a different identification
of $H^2(X^\theta,\C)$ with $\tilde{\param}$, one twisted by $w$ (so that the ample cone actually
contains $\theta$).

\subsection{Quantizations}
We will study quantizations of $X,X_0,\tilde{X}$. By a quantization of $X_0$, we mean a filtered
algebra $\A$ together with an isomorphism $\gr\A\cong \C[X_0]$ of graded Poisson algebras.
By a quantization of $X=X^\theta$, we mean a sheaf $\A^\theta$ of filtered algebras in the conical
topology on $X$ (in this topology, ``open'' means Zariski open and $\C^\times$-stable) that is complete
and separated with respect to the filtration together with an isomorphism $\gr\A^\theta\cong \Str_{X^\theta}$
(of sheaves of graded Poisson algebras).
A quantization $\tilde{\A}^\theta$ of the universal deformation $\tilde{X}^\theta$ is
\begin{itemize}
\item a filtered
sheaf of $\C[\tilde{\param}]$-algebras (with a complete and separated filtration),
where the induced filtration on $\C[\tilde{\param}]$ coincides with a natural one, where $\deg\tilde{\param}=d$,
\item together with an isomorphism $\gr\tilde{\A}^\theta\xrightarrow{\sim} \Str_{\tilde{X}^\theta}$ of graded Poisson $\C[\tilde{\param}]$-algebras.
\end{itemize}
Note that we can specialize a quantization  $\tilde{\A}^\theta$ to a point in $\tilde{\param}$
getting a quantization of $X$.

A result of Bezrukavnikov and Kaledin, \cite{BK}, (with ramifications given in \cite[Section 2]{quant}) shows that quantizations $\A^\theta$ are parameterized (up to an isomorphism) by the points in $\tilde{\param}$. More
precisely, there is a {\it canonical} quantization $\tilde{\A}^\theta$ of $\tilde{X}^\theta$ such that the quantization of $X^\theta$ corresponding to $\lambda\in \tilde{\param}$ is the specialization of $\tilde{\A}^\theta$ to $\lambda$. The quantization $\tilde{\A}^\theta$ has the following important property: there is an anti-automorphism, $\varsigma$
(to be called the parity anti-automorphism) that preserves the filtration, is the identity on the associated
graded, preserves $\tilde{\param}^*\subset \Gamma(\tilde{\A}^\theta)$ and induces $-1$ on $\tilde{\param}^*$.
It follows that $\A^\theta_{-\lambda}$ is isomorphic to $(\A_\lambda^\theta)^{opp}$.

We set $\tilde{\A}:=\Gamma(\tilde{\A}^\theta), \A_\lambda=\Gamma(\A_\lambda^\theta)$ so that
$\A_\lambda$ is the specialization of $\tilde{\A}$ to $\lambda$. It follows
from \cite[Section 3.3]{BPW} that the algebras $\tilde{\A}, \A_\lambda$ are independent from the choice of $\theta$.
From $H^i(X^\theta,\Str_{X^\theta})=0$, we deduce that the higher cohomology of both $\tilde{\A}^\theta$
and $\A^\theta_\lambda$ vanish and also that $\tilde{\A}$ is a quantization of $\C[\tilde{X}]$
and $\A_\lambda$ is a quantization of $\C[X]=\C[X_0]$. Also we note that $\A_\lambda$
is the specialization of $\tilde{\A}$ at $\lambda$ and that $\A_{-\lambda}\cong \A_\lambda^{opp}$.
Also we have $\A_\lambda\cong \A_{w\lambda}$ for all $\lambda\in \tilde{\param},w\in W$,
see \cite[Section 3.3]{BPW}.

\subsection{Localization and translation equivalences}\label{SSS_trans_equi}
Consider the category $\operatorname{Coh}(\A_\lambda^\theta)$  of coherent $\A_\lambda^\theta$-modules
and $\A_\lambda\operatorname{-mod}$ of finitely generated $\A_\lambda$-modules. Recall that an
$\A_\lambda^\theta$-module $M$ is called {\it coherent} if it can be equipped with a complete
and separated $\A_\lambda^\theta$-module filtration such that $\gr M$ is coherent
(such a filtration is called {\it good}).

We have functors between these categories, the global section functor $\Gamma_\lambda:\Coh(\A_\lambda^\theta)
\rightarrow \A_\lambda\operatorname{-mod}$ and its left adjoint, the localization
functor $\Loc_\lambda:\A_\lambda\operatorname{-mod}\rightarrow \Coh(\A_\lambda^\theta),
N\mapsto \A_\lambda^\theta\otimes_{\A_\lambda}N$. We can also consider their derived
functors $R\Gamma_\lambda: D^b(\Coh(\A_\lambda^\theta))\rightarrow D^b(\A_\lambda\operatorname{-mod})$
and $L\Loc_\lambda: D^-(\A_\lambda\operatorname{-mod})\rightarrow D^-(\Coh(\A_\lambda^\theta))$.

We will need a partial answer to the question of when $\Gamma_\lambda,\Loc_\lambda$ are mutually quasi-inverse
equivalences (in this case we will say that abelian localization holds for $(\lambda,\theta)$), see \cite[Corollary 5.12]{BPW}.

\begin{Lem}\label{Lem:local}
Let $\lambda\in \tilde{\param}$ and $\chi\in \tilde{\param}_{\Z}$ be ample for $X^\theta$.
Then there is $n_0\in \Z$ such that $\Gamma_{\lambda+n\chi}: \Coh(\A_{\lambda+n\chi}^\theta)
\rightleftarrows \A_{\lambda+n\chi}\operatorname{-mod}:\Loc_{\lambda+n\chi}$ are mutually quasi-inverse
equivalences for all $n\geqslant n_0$.
\end{Lem}

Let us proceed to translation equivalences. When $\chi\in \tilde{\param}_{\Z}$, we have an equivalence $\operatorname{Coh}(\A_\lambda^\theta)\cong \operatorname{Coh}(\A_{\lambda+\chi}^{\theta})$. This equivalence is given by tensoring with the $\A_{\lambda+\chi}^\theta$-$\A_\lambda^\theta$-bimodule to be denoted by $\A_{\lambda,\chi}^\theta$. This bimodule is a unique quantization of the line bundle $\mathcal{O}(\chi)$ corresponding to $\chi$. Similarly, we can define the $\tilde{\A}^\theta$-bimodule $\tilde{\A}^{\theta}_{\tilde{\param},\chi}$ quantizing the extension to $\tilde{X}$ of the line bundle corresponding to $\chi$, \cite[Section 5.1]{BPW}. Clearly, $\A^\theta_{\lambda,\chi}$ is the specialization
of $\tilde{\A}^\theta_{\tilde{\param},\chi}$.

We write $\A^{(\theta)}_{\lambda,\chi}, \tilde{\A}^{(\theta)}_{\tilde{\param},\chi}$ for the global
sections of $\A^\theta_{\lambda,\chi}, \tilde{\A}^\theta_{\tilde{\param},\chi}$, respectively.
For an affine subspace $\param\subset \tilde{\param}$, we set $\A^{(\theta)}_{\param,\chi}:=\Gamma(\A^\theta_{\tilde{\param},\chi}
\otimes_{\C[\tilde{\param}]}\C[\param])$.

\begin{Lem}\label{Lem:transl_spec}
Suppose that $H^1(X,\mathcal{O}(\chi))=0$. Then
$\A^{(\theta)}_{\lambda,\chi}:=\tilde{\A}^{(\theta)}_{\param,\chi}\otimes\C_\lambda$.
\end{Lem}

The proof repeats that of \cite[Lemma 5.4(1)]{BL}.

Let us provide a sufficient condition on $z$ for the algebras $\A_{\lambda+z\chi}$ to be simple, to have
finite homological dimension and also for the localization theorem to hold.

\begin{Lem}\label{Lem:gen_simplicity}
Let $\chi$ be ample. Then the following is true.
\begin{enumerate}
\item There are finitely many elements $z_1,\ldots,z_k\in \C$ such that $\Loc_{\lambda+z\chi}$
is an equivalence provided $z_i-z\not\in \Z_{\geqslant 0}$ for every $i=1,\ldots,k$.
\item For  $z\not\in\bigcup_{i=1}^k (z_i+\Z)$, the algebra $\A_{\lambda+z\chi}$ is simple.
\item For a Zariski generic $z\in\C$, the algebra $\A_{\lambda+z\chi}$
has finite homological dimension. This dimension does not exceed $\dim X$.
\end{enumerate}
\end{Lem}
\begin{proof}
Let us prove (1).  
The localization holds for all parameters of the form $\lambda+(z+n)\chi$ with $n\in \Z_{\geqslant 0}$
if and only if the bimodules $\A_{\lambda+(z+n)\chi,\chi}^{(\theta)}=\A_{\ell,\chi}^{(\theta)}|_{\lambda+(z+n)\chi},
\A_{\ell,-\chi}^{(\theta)}|_{\lambda+(z+n+1)\chi}$ define mutually inverse Morita equivalences,
see \cite[Section 5.3]{BPW}. The set of
$z\in \C$ such that $\A^{(\theta)}_{\lambda+z\chi,\chi}$ and $\A^{(\theta)}_{\ell,-\chi}|_{\lambda+(z+1)\chi}$ are
mutually inverse Morita equivalences is Zariski open. The reason is that this set is precisely the locus $\{\lambda+z\chi, z\in \C\}$,
where the natural homomorphisms (given by multiplication) $$\A_{\ell+\chi,-\chi}^{(\theta)}\otimes_{\A_{\ell+\chi}}\A^{(\theta)}_{\ell,\chi}
\rightarrow \A_{\ell}, \A^{(\theta)}_{\ell,\chi}\otimes_{\A_{\ell}}\A^{(\theta)}_{\ell+\chi,-\chi}\rightarrow
\A_{\ell+\chi}$$
are isomorphisms. Together with Proposition \ref{Prop:supp_gr} below, this proves (1).

Let us prove (2):  $\A_{\lambda+z\chi}$ is simple for a Weil generic $z$. Equivalently,
we need to show that the $\A_{\lambda+z\chi}\otimes \A_{-\lambda-z\chi}$-module $\A_{\lambda+z\chi}$
is simple. By (1), abelian localization holds for the parameter $(\lambda+z\chi, -\lambda+(-z)\chi)$
and the variety $X^\theta\times X^\theta$ when $z$ is Weil generic. The module $\A_{\lambda+z\chi}$
localizes to $\A^\theta_{\lambda+z\chi}$. The latter is simple because it is supported
on the diagonal in $X^\theta\times X^\theta$ and has rank 1. This finishes the proof of
simplicity of $\A_{\lambda+z\chi}$ for a Weil generic $z$.

Let us proceed to (3). The algebra $\A_{\lambda+z\chi}$ has  homological dimension not exceeding $\dim X$
provided localization holds for $(\lambda+z\chi,\theta)$ or for $(\lambda+z\chi,-\theta)$. Now our
claim easily follows from (1) applied to both these situations.
\end{proof}

We will also need the following general lemma whose proof is completely analogous to that of
\cite[Lemma 5.7]{BL}.

\begin{Lem}\label{Lem:gen_freeness}
Let $\ell$ denote a one-dimensional vector space.  Let
$\mathfrak{A}_\ell$ be a filtered $\C[\ell]$-algebra such that $\ell^*$ has degree $d$ and $\gr\mathfrak{A}_\ell$
is commutative and  finitely generated. Let $M_\ell$ be a finitely generated $\mathfrak{A}_\ell$-module.
Then there is an open subset $\ell^0\subset \ell$ such that $\C[\ell^0]\otimes_{\C[\ell]}M_\ell$ is a
free $\C[\ell^0]$-module.
\end{Lem}

\subsection{Harish-Chandra bimodules}
Let us introduce the notion of a Harish-Chandra, shortly, HC bimodule.
Let $\A,\A'$ two $\Z_{\geqslant 0}$-filtered algebras. Suppose we have a fixed isomorphism
$\gr\A\cong \gr \A'$ of graded algebras. Further, suppose that
$A:=\gr\A=\gr\A'$ is finitely generated. Let $\B$ be an $\A'$-$\A$-bimodule.
We say that $\B$ is HC if it is equipped with a bimodule filtration
(called {\it good}) such that $\gr\B$ is a finitely generated
$A$-module (meaning that the left and the right actions of $A$
coincide). Note that a HC bimodule is automatically finitely generated as a left
and as a right module.

An example is as follows: let $\A:=\A_\lambda, \A':=\A_{\lambda+\chi}$
and $\B:=\A_{\lambda,\chi}^{(\theta)}$. Then $\B$ is HC, see \cite[Section 6.3]{BPW}.

Below we will need a result about the supports of HC bimodules in the parameter
spaces. Namely, suppose that $\param$ is a finite dimensional vector space and
let $\A$ be a filtered $\C[\param]$-algebra such that $\param$ has degree $d>0$
and $[\A_{\leqslant i},\A_{\leqslant j}]\subset \A_{\leqslant i+j-d}$.
Further, suppose that $\gr \A/(\param)$ is finitely generated and the corresponding
Poisson scheme $X_0$ has finitely many symplectic leaves.

Now let $\B$ be a Harish-Chandra $\A$-bimodule. Define its right $\param$-support
$\Supp^r_{\param}(\B)$ as the set of all $p\in \param$ such that $\B\otimes_{\C[\param]}\C_p\neq 0$.
Similarly, we can define the  $\param$-support of $\gr\B$ (where the associated graded
is taken with respect to a good filtration) to be denoted by $\Supp_{\param}(\gr\B)$.
The following proposition generalizes \cite[Proposition 5.15]{BL}.
\begin{Prop}\label{Prop:supp_gr}
The subset $\Supp^r_{\param}(\B)\subset \param$ is Zariski closed.
The associated cone of $\Supp^r_{\param}(\B)$ coincides with
$\Supp_{\param}(\gr\B)$.
\end{Prop}
\begin{proof}
The proof generalizes that of \cite[Proposition 5.15]{BL}. Namely,
to $\B$ we can assign its associated variety $\VA(\B)\subset X_0$
defined as the support of the finitely generated $\C[X_0]$-module
$\gr\B/\param \gr \B$.  It is easy to see that $\gr\B/\param \gr\B$
is a Poisson module.

Let us order the symplectic leaves of $X_0$: $\mathcal{L}_1,\ldots,\mathcal{L}_k$
in such a way that $\mathcal{L}_i\subset \overline{\mathcal{L}}_j$
implies $i\leqslant j$. We prove our claim by induction on $k$
such that $\VA(\B)\subset \bigcup_{i\leqslant k}\mathcal{L}_i$. The base,
$k=1$, follows from the observation that $\mathcal{L}_1$ is a single
point, in which case $\B$ is finitely generated as a module over $\C[\param]$.

Now suppose that we know our claim for all $\B'$ with $\VA(\B')\subset
\bigcup_{i\leqslant k-1}\mathcal{L}_i$. Pick $\B$ with $\VA(\B)\subset
\bigcup_{i\leqslant k}\mathcal{L}_i$.

Equip $\B$ with a good filtration and consider the Rees bimodule $\B_\hbar$.
Pick $x\in \mathcal{L}_k$ and consider the completion $\B^{\wedge_x}_\hbar$
at $x$. It decomposes as $\B^{\wedge_x}_\hbar=\mathbb{A}_\hbar^{\wedge_0}\widehat{\otimes}_{\C[\hbar]}
\underline{\B}_\hbar$, where $\mathbb{A}_\hbar$ is the homogenized Weyl algebra
of the tangent space $T_x\mathcal{L}_i$ that is embedded into $\A^{\wedge_x}_\hbar$
(compare, for example, to \cite[Section 5.4]{BL})
and $\underline{\B}_\hbar$ is the centralizer of $\mathbb{A}^{\wedge_0}_\hbar$ in
$\B_\hbar^{\wedge_x}$. By the choice of $x$ and $\mathcal{L}_i$, we see that
$\underline{\B}_\hbar$ is finitely generated over $\C[[\param,\hbar]]$.
It follows that the right $\C[[\param,\hbar]]$-support of $\underline{\B}_\hbar$ (that is the same as
the support of $\B_\hbar^{\wedge_x}$) is the closed subscheme of the formal
neighborhood of $0$ in $\param\oplus \C$ whose intersection with $\param$
coincides with the $\C[[\param]]$-support of $(\gr\B)^{\wedge_x}$. We note that $\B_\hbar$
comes with the Euler derivation $\hbar\partial_{\hbar}$. This derivation extends
to $\B_\hbar^{\wedge_x}$. We deduce that the support of $\B_{\hbar}^{\wedge_x}$
is the formal neighborhood at $0$ of a closed subvariety $Z\subset \param\oplus \C$
that is stable with respect to the $\C^\times$-action on $\param\oplus \C,
t.(p,\alpha)=(t^dp,t\alpha)$. Let $Z_1$ be the intersection of $Z$ with $\param\times \{1\}$.
Let us check that $Z_1\subset \Supp^r_{\param}(\B)$. Indeed, let $z\in Z_1$
and $D$ be the formal neighborhood of $0$ in $\C z$. Then $(\B_\hbar^{\wedge_x})\otimes_{\C[[\param,\hbar]]}\C[D]$
coincides with $(\B_{z_1})_\hbar^{\wedge_x}$. This implies $Z_1\subset \Supp^r_{\param}(\B)$.

Moreover, consider $Z_0=Z\cap(\param\times\{0\})$. It coincides with the asymptotic cone of $Z_1$,
and the formal neighborhood of $0$ in $Z_0$ is the support of $(\gr\B)^{\wedge_x}$.

Now let $I$ be the annihilator of $\B_\hbar^{\wedge_x}$ in $\C[\param,\hbar]$.
The corresponding subvariety in $\param\oplus \C$ is $Z$. Let $I_1,I_0$ be the specializations
of $I$ to $\hbar=1,\hbar=0$, respectively. Note that $(\B I_1)_\hbar^{\wedge_x}=0$.
It follows that $\VA(\B I_1)\subset \bigcup_{i\leqslant k-1}\mathcal{L}_i$.
So $\Supp^r_{\param}(\B I_1)$ is a closed subvariety in $\param$ and its asymptotic cone
is $\Supp_{\param}(\gr(\B I_1))$. Similarly to the proof of \cite[Proposition 5.15]{BL},
we see that $\Supp^r_{\param}(\B)=\Supp^r_{\param}(\B I_1)\cup Z_1$ and
$\Supp_{\param}(\gr\B)=\Supp_{\param}(\gr(\B I_1))\cup Z_0$. Our claim follows.
\end{proof}

\section{Standardly stratified categories}\label{S_SSC}
Here we recall the notion of a standardly stratified category following \cite[Section 2]{LW}
that is a refinement of the first definition of a standardly stratified category given in
\cite{CPS}.

\subsection{Definition}
Let $\K$ be a field. Let $\Cat$ be a $\K$-linear abelian category equivalent to the category of
finite  dimensional modules over a split unital associative finite dimensional $\K$-algebra. We will write
$\Tcal$ for an indexing set for the simple objects of  $\Cat$. Let us write $L(\tau)$
for the simple object indexed by $\tau\in \Tcal$ and $P(\tau)$ for the projective cover of
$L(\tau)$.

The additional structure of a standardly stratified category on $\mathcal{C}$ is a partial  pre-order
$\leqslant$ on $\Tcal$ that should satisfy  axioms (SS1),(SS2) to be explained below.
Let us write $\Xi$ for the set of equivalence classes of $\leqslant$,
this is a poset (with partial order again denoted by $\leqslant$) that comes with a natural surjection $\varrho:\Tcal\twoheadrightarrow \Xi$. The pre-order $\leqslant$ defines a filtration on $\Cat$ by Serre subcategories indexed by $\Xi$. Namely, to $\xi\in \Xi$ we assign the subcategory $\Cat_{\leqslant \xi}$ that is the Serre span
of the simples $L(\tau)$ with $\varrho(\tau)\leqslant \xi$. Define $\Cat_{<\xi}$ analogously and let
$\Cat_\xi$ denote the quotient $\Cat_{\leqslant \xi}/\Cat_{<\xi}$. Let $\pi_\xi$ denote the quotient
functor $\Cat_{\leqslant \xi}\twoheadrightarrow \Cat_{\xi}$. Let us write $\Delta_\xi:\Cat_\xi\rightarrow \Cat_{\leqslant \xi}$ for the left adjoint functor of $\pi_\xi$. Also we write $\gr\Cat$ for $\bigoplus_{\xi}\Cat_\xi, \Delta$
for $\bigoplus_\xi \Delta_\xi:\gr\Cat\rightarrow \Cat$. Finally, for $\tau\in \varrho^{-1}(\xi)$
we write $L_\xi(\tau)$ for $\pi_\xi(L(\tau))$, $P_\xi(\tau)$ for the projective cover of
$L_\xi(\tau)$ in $\Cat_\xi$ and $\Delta(\tau)$ for $\Delta_\xi(P_\xi(\tau))$.
The object $\Delta(\tau)$ is called {\it standard}. The object $\bDelta(\tau):=\Delta_\xi(L_\xi(\tau))$
is called {\it proper standard}.

The axioms to be satisfied by $(\Cat,\leqslant)$ in order to give a standardly stratified structure are as follows.
\begin{itemize}
\item[(SS1)] The functor $\Delta:\gr\Cat\rightarrow \Cat$ is exact.
\item[(SS2)] The kernel of  $P(\tau)\twoheadrightarrow \Delta(\tau)$
has a filtration with successive quotients $\Delta(\tau')$, where $\tau'>\tau$.
\end{itemize}


Note that (SS1) allows to identify $K_0(\gr\Cat)$ and $K_0(\Cat)$ by means of $\Delta$. If (SS2)
holds as well, then we also have the identification of $K_0(\gr\Cat\operatorname{-proj})$
and $K_0(\Cat\operatorname{-proj})$.

If all quotient categories $\Cat_\xi$ are equivalent to $\operatorname{Vect}$, then
we recover the classical definition of  a highest weight category. On
the opposite end, if we take the trivial pre-order on $\Tcal$, then there is no additional structure.

\subsection{(Proper) standardly filtered objects}
We say that an object in $\Cat$ is {\it standardly filtered} if it admits a filtration whose successive quotients are
standard. The notion of a proper standardly filtered object is introduced in a similar fashion. The categories
of the standardly filtered objects and of the proper standardly filtered objects will be denoted by $\Cat^{\Delta}$
and $\Cat^{\bDelta}$. Note that (SS1) implies that $\Cat^{\Delta}\subset \Cat^{\bDelta}$.

\begin{Lem}\label{Lem:prop_stand_stratif}
Suppose (SS1) holds.  Let $M$ be an object in $\Cat^{\bDelta}$ such that all proper standard  filtration subquotients are of the form $\bDelta(\tau)$ with $\varrho(\tau)=\xi$. Then $M=\Delta_\xi(\pi_\xi(M))$.
\end{Lem}
\begin{proof}
From adjointness, we have a homomorphism $\Delta_\xi(\pi_\xi(M))\rightarrow M$. After applying
$\pi_\xi$ this homomorphism becomes an isomorphism. But all simples in the head of $M$
are not killed by $\pi_\xi$. So $\Delta_\xi(\pi_\xi(M))\twoheadrightarrow M$. Since
the classes of $M$ and $\Delta_\xi(\pi_\xi(M))$ in $K_0$ coincide, we conclude that this epimorphism is iso.
\end{proof}

Also note that in a standardly stratified category the following hold:
\begin{align}\label{Ext:vanish}
&\Ext^i_{\Cat}(\Delta_\xi(M),\Delta_{\xi'}(N))\neq 0\Rightarrow \xi'\leqslant \xi.\\\label{Ext:equal}
&\Ext^i_{\Cat}(\Delta_\xi(M),\Delta_{\xi}(N))=\Ext^i_{\Cat_\xi}(M,N).
\end{align}

\subsection{Subcategories and quotients}
Suppose that $\Cat$ is  standardly stratified.

Let $\Xi_0$ be a poset ideal in $\Xi$. Let $\Cat_{\Xi_0}$ denote the Serre span
of the simples $L(\tau)$ with $\varrho(\tau)\in \Xi_0$. Then $\Cat_{\Xi_0}$
is a standardly stratified category with pre-order on $\Tcal_0:=\varrho^{-1}(\Xi_0)$
restricted from $\Xi$. Note that, for $\tau\in \Tcal_0$, we have
$\bDelta_{\Xi_0}(\tau)=\bDelta(\tau),
\Delta_{\Xi_0}(\tau)=\Delta(\tau)$,
where the subscript $\Xi_0$ refers to the objects computed in $\Cat_{\Xi_0}$.

The embedding $\iota_{\Xi_0}:\Cat_{\Xi_0}^{\bDelta}\hookrightarrow \Cat^{\bDelta}$
admits a left adjoint functor $\iota^!_{\Xi_0}$: to an object $M\in \Cat^{\bDelta}$,
this functor assigns the maximal quotient lying in $\Cat_{\Xi_0}^{\bDelta}$.

The following lemma will be of importance in the sequel.
\begin{Lem}\label{Lem:derived_inclusion}
The natural functor $D^-(\Cat_{\Xi_0})\rightarrow D^-(\Cat)$ is a fully faithful embedding.
\end{Lem}
\begin{proof}
Let $P_\bullet$ be a projective resolution (in $\Cat$) of $M\in \Cat_{\Xi_0}$. Then
$\iota^!_{\Xi_0}P_\bullet$ is easily seen to be exact.  So it is a projective resolution of $M$ in $\Cat_{\Xi_0}$.
It follows that  for $M,N\in \Cat_{\Xi_0}$, we have $$\operatorname{Ext}^i_{\Cat}(M,N)=
\operatorname{Ext}^i_{\Cat_{\Xi_0}}(M,N).$$ This completes the proof.
\end{proof}

Now let $\Cat^{\Xi_0}$ be the quotient category $\Cat/\Cat_{\Xi_0}$. Let $\pi_{\Xi_0}$
denote the quotient functor $\Cat\rightarrow \Cat^{\Xi_0}$ and let $\pi_{\Xi_0}^!$
be its left adjoint. The category $\Cat^{\Xi_0}$ is standardly stratified with
pre-order on $\Tcal^0:=\Tcal\setminus \Tcal_0$ restricted from $\Tcal$.
For $\xi\in \Xi^0:=\Xi\setminus \Xi_0$ we have $\Delta^0_\xi=\pi_{\Xi_0}\circ \Delta_\xi$.
Let us also point out that $\pi_{\Xi_0}^!$ defines a full embedding $(\Cat^{\Xi_0})^{\bDelta}
\hookrightarrow \Cat^{\bDelta}$ whose image coincides with the full subcategory
$\Cat^{\bDelta,\Tcal^0}$ consisting of all objects that admit a filtration with successive
quotients $\Delta(\tau)$ with $\tau\in \Tcal^0$. This embedding sends $\Delta^{\Xi_0}(\tau)$
to $\Delta(\tau)$, $P^{\Xi_0}(\tau)$ to $P(\tau)$.

\subsection{Opposite category}
Let $\Cat$ be a standardly stratified category.
It turns out that the opposite category $\Cat^{opp}$ is also standardly stratified with the same pre-order
$\leqslant$, see \cite[1.2]{LW}. The standard and proper standard objects for $\Cat^{opp}$ are denoted
by $\nabla(\tau)$ and $\bnabla(\tau)$. When viewed as objects of $\Cat$, they are called
{\it costandard} and {\it proper costandard}. The right adjoint functor to $\pi_\xi$ will be
denoted by $\nabla_\xi$ and we write $\nabla$ for $\bigoplus_\xi \nabla_\xi$. So we have
$\bnabla(\tau)=\nabla(L_\xi(\tau))$ and $\nabla(\tau)=\nabla(I_\xi(\tau))$, where
$I_\xi(\tau)$ is the injective envelope of $L_\xi(\tau)$ in $\Cat_\xi$.

Let us write $\Cat^\nabla, \Cat^{\bnabla}$ for the subcategories of costandardly filtered
objects. We have the following standard lemma.

\begin{Lem}[Lemma 2.4 in \cite{LW}]\label{Lem:exts}
The following is true.
\begin{enumerate}
\item $\dim \Ext^i(\Delta(\tau), \bnabla(\tau'))=\dim
  \Ext^i(\bDelta(\tau), \nabla(\tau'))=
  \delta_{i,0}\delta_{\tau,\tau'}$. 
 \item For $N\in \Cat$, we have $N\in \Cat^\nabla$ (resp., $N\in
  \Cat^{\bnabla}$) if and only if $\Ext^1(\bDelta(\tau),N)=0$
  (resp., $\Ext^1(\Delta(\tau), N)=0$) for all $\tau$. Similar characterizations
  hold for $\Cat^{\Delta},\Cat^{\bDelta}$.
\end{enumerate}
\end{Lem}

Let us also note the following fact.

\begin{Lem}\label{Lem:w_st_stratif}
Suppose (SS1) holds as well as the following weaker version of (SS2): $P(\tau)$ admits an epimorphism
onto $\Delta(\tau)$ and the kernel admits a filtration with quotient of the form $\Delta_\xi(M_\xi)$
for $\xi>\varrho(\tau)$ and $M_\xi\in \Cat_\xi$ (unlike in (SS2) we do not require $M_\xi$ to be projective).
Then $\Cat$ is standardly stratified if and only if the right adjoint $\nabla_\xi$ of $\pi_\xi$ is exact
for all $\xi$.
\end{Lem}
\begin{proof}
It remains to show that if $\nabla_\xi$ is exact, then (SS2) holds. Set $\Xi_0:=\{\xi'\in \Xi| \xi'\leqslant \xi\}$.
The object $\iota^!_{\Xi_0}(P(\tau))$ is projective in $\Cat_{\leqslant \xi}$ if $\varrho(\tau)\leqslant \xi$
and is zero otherwise. Since $\pi_\xi$ has an exact right adjoint, the object
$\pi_\xi(\iota^!_{\Xi_0}(P(\tau)))$ is projective. By Lemma \ref{Lem:prop_stand_stratif},
the object $\Delta_\xi(M)$ in (SS2$'$) is the sum of $\Delta(\tau)$'s. This completes the proof.
\end{proof}

\subsection{Compatible standardly stratified structures}\label{SSS_SS_comp}
Now suppose that $\Cat$ is a highest weight category with set of simples $\Tcal$
and partial order $\leqslant$. Consider a pre-order $\preceq$ on $\Tcal$
refined  by $\leqslant$, meaning that $\tau\leqslant \tau'$ implies
$\tau\preceq \tau'$. We say that $\preceq$ is standardly stratified
if $(\Cat,\preceq)$ is a standardly stratified category. The corresponding standardly
stratified structure will be called {\it compatible} with the initial highest weight
structure. Note that the corresponding associated graded category $\gr_{\preceq}\Cat$
carries a natural highest weight structure.

The following lemma gives a criterium for $\preceq$ to be standardly stratified.
Namely, for a standardly stratified object $M$ and an equivalence  class $\xi$
for $\preceq$ define a standardly stratified object $M(\xi)$ as
$\pi_\xi\circ \iota_{\Tcal_0}^!(M)$, where $\Tcal_0:=\{\tau|\varrho(\tau)\leqslant \xi\}$
is a poset ideal in $\Tcal$ with respect to $\leqslant$, and $\iota_{\Tcal_0}:
\Cat_{\Tcal_0}\hookrightarrow \Cat$ is an inclusion. The object $M(\xi)$ is standardly
stratified in $\Cat_\xi$. Similarly, for a costandardly filtered object $N$, we can
define $N(\xi)$ in a similar way (but we need to use the right
adjoint $\iota_{\Tcal_0}^*$ instead of the left adjoint $\iota_{\Tcal_0}^!$).

\begin{Lem}\label{Lem:ss_order}
The following are equivalent:
\begin{enumerate}
\item The partial order $\preceq$ is standardly stratified.
\item For any projective $P\in \Cat$ and any injective $I\in \Cat$, the objects
$P(\xi),I(\xi)\in \Cat_\xi$ are projective and injective, respectively.
\end{enumerate}
\end{Lem}
\begin{proof}
If (1) holds, then (2) is just (SS2) for $\Cat$ and $\Cat^{opp}$. Conversely,
as we have seen in the proof of Lemma \ref{Lem:w_st_stratif}, the claim that
$I(\xi)$ are injective for all injectives $I$ and all $\xi$ implies (SS1).
Now the claim about $P(\xi)$'s implies (SS2).
\end{proof}

Conversely, let $\Cat$  be a standardly stratified category with respect to a pre-order $\preceq$.
Suppose that each $\Cat_\xi$ is highest weight with respect to an order $\leqslant_\xi$ for all
$\xi\in \Xi$. Then $\Cat$ is highest weight with respect to the order $\leqslant$ defined as follows:
$\tau\leqslant \tau'$ if and only if either $\tau\prec \tau'$ or $\tau\sim \tau'$ and $\tau\leqslant_\xi\tau'$
for $\xi=\varrho(\tau)$.

\section{Categories $\mathcal{O}$}\label{S_cat_O}
\subsection{Cartan subquotient: algebra level}\label{SSS_Cart_subquot_alg}
Here we define a suitable subquotient of an algebra equipped with a Hamiltonian $\C^\times$-action.

Let $\A=\bigcup_{i\in \Z} \A_{\leqslant i}$ be a filtered associative algebra (with a complete and separated
filtration) such that $[\A_{\leqslant i},\A_{\leqslant j}]\subset \A_{\leqslant i+j-d}$ for some $d\in \Z_{>0}$.
The algebra $\gr\A$ is commutative and we require that it is finitely generated.

Suppose that $\A$ is equipped with a pro-rational Hamiltonian $\C^\times$-action $\nu$ that preserves the filtration.
Let $h\in \A$ be the image of $1$ under the comoment map and let $\A^i$ denote the $i$th graded component
so that $\A^i:=\{a| [h,a]=ia\}$.  We set $\A^{\geqslant 0}:=\bigoplus_{i\geqslant 0}\A^i,
\A^{> 0}:=\bigoplus_{i>0}\A^i, \Ca_\nu(\A):=\A^{\geqslant 0}/(\A^{\geqslant 0}\cap \A\A^{>0})$. Note that $\A^{\geqslant 0}\cap \A\A^{>0}$ is a two-sided ideal in $\A^{\geqslant 0}$. We have a natural isomorphism
\begin{equation}\label{eq:Cartan_iso}\Ca_{\nu}(\A)\cong \A^0/\bigoplus_{i>0}\A^{-i}\A^i.\end{equation}

The algebra $\Ca_\nu(\A)$ inherits a filtration
from $\A$. This filtration is complete and separated as any two-sided ideal in $\A^0$ is closed
with respect to the topology induced by the filtration, compare to \cite[Lemma 2.4.4]{HC}. The algebra $\gr\Ca_\nu(\A)$ is commutative and, being a quotient of $(\gr \A)^{\nu(\C^\times)}$, finitely generated.

Let us note that the definitions of $\A^{>0},\A^{\geqslant 0}, \Ca_\nu(\A)$ make sense even if
the action $\nu$ is not Hamiltonian.

\subsection{Categories $\mathcal{O}$: algebra level}\label{SSS_cat_O_gen}
Let $\A$ be as in Section \ref{SSS_Cart_subquot_alg} and assume that the filtration mentioned there is by $\Z_{\geqslant 0}$.
By the category $\OCat_\nu(\A)$ we mean the full subcategory of the category $\A\operatorname{-mod}$
of the finitely generated $\A$-modules consisting of all modules such that $\A^{>0}$ acts locally
nilpotently. We have the {\it Verma module} functor $\Delta_\nu:\Ca_\nu(\A)\operatorname{-mod}\rightarrow\OCat_\nu(\A)$
given by $\Delta_\nu(N):=\A\otimes_{\A^{\geqslant 0}}N=(\A/\A\A^{>0})\otimes_{\Ca_{\nu}(\A)}N$. Note that if $h$ acts on $N$ with eigenvalue $\alpha$
(so that $N$ is a single generalized $h$-eigen-space), then $h$ acts locally finitely on $\Delta_\nu(N)$
with eigenvalues in $\alpha+\Z_{\leqslant 0}$. The generalized eigenspace with eigenvalue $\alpha$
coincides with $(\A/\A\A^{>0})^{\nu(\C^\times)}\otimes_{\A^0}N=N$ (note that $\Ca_\nu(\A)$
is identified with $(\A/\A\A^{>0})^{\nu(\C^\times)}$). The module $\Delta_\nu(N)$ has the maximal submodule
that does not intersect $N$, the quotient is denoted by $L_\nu(N)$. Note that $L_\nu(N)$ is simple provided
$N$ is simple. Conversely, any simple object $L$ in $\OCat_\nu(\A)$ is of the form $L_\nu(N)$ for a unique
simple $N\in \Ca_\nu(\A)\operatorname{-mod}$, this simple $N$ coincides with the annihilator
$L^{\A^{>0}}$ of $\A^{>0}$ in $L$.

Assume now that $\dim\Ca_\nu(\A)<\infty$. Let $N_1,\ldots,N_r$ be the full list of the simple
$\Ca_\nu(\A)$-modules. Let $\alpha_1,\ldots,\alpha_r$ be the eigenvalues of $h$ on these modules
(note that $h$ maps into the center of $\Ca_\nu(\A)$). We define a partial order $\leqslant$ on
the set $N_1,\ldots,N_r$ by setting $N_i\leqslant N_j$ if $\alpha_j-\alpha_i\in \Z_{>0}$ or $i=j$.
Using the order it is easy to prove the following.

\begin{Lem}\label{Lem:fin_length}
Under the assumption $\dim\Ca_\nu(\A)<\infty$, the following is true:
\begin{enumerate}
\item The modules $\Delta_\nu(N)$ have finite length.
\item The simple $L_\nu(N_i)$ appears in the composition series of $\Delta_\nu(N_i)$ with
multiplicity $1$ and as a quotient. Further, if $L_\nu(N_j)$ appears in the composition
series of $\Delta_\nu(N_i)$, then $N_j\leqslant N_i$.
\item All modules in $\OCat_\nu(\A)$ have finite length and all generalized eigen-spaces
for $h$ are finite dimensional.
\item The category $\OCat_{\nu}(\A)$ consists of all $\A$-modules with locally nilpotent action
of $\A^{>0}$, locally finite action of $h$ and finite dimensional generalized eigen-spaces for $h$.
\end{enumerate}
\end{Lem}

Also an argument similar to that in \cite[Section 2.4]{GGOR} implies that $\OCat_\nu(\A)$ has enough projectives.

Now let us describe the opposite category. The opposite algebra $\A^{opp}$ is also graded.
Consider the category $\OCat_{-\nu}(\A^{opp})$. We have a natural identification
$\Ca_{-\nu}(\A^{opp})\cong \Ca_{\nu}(\A)^{opp}$.  For a module $M\in \OCat_\nu(\A)$, we define its restricted dual
$M^{(*)}$ as $\bigoplus_\alpha M_\alpha^*$, where $M_\alpha$ stands for the generalized
eigen-space for $h$ with eigenvalue $\alpha$. Note that $M^{(*)}\in \OCat_{-\nu}(\A^{opp})$
by (4) of Lemma \ref{Lem:fin_length}.
Similarly, for $M'\in \OCat_{-\nu}(\A^{opp})$, we can form $M'^{(*)}$ and this will
be an object of $\OCat_\nu(\A)$. The functors of taking the restricted dual define
mutually inverse contravariant equivalences between $\OCat_\nu(\A)$ and
$\OCat_{-\nu}(\A^{opp})$.

Using this equivalences one can introduce dual Verma modules $\nabla_\nu(N):=\Delta_{-\nu}^{opp}(N^*)^{(*)}\in \OCat_\nu(\A)$. Equivalently, $\nabla_\nu(N)=\Hom_{\Ca_{\nu}(\A)}(\A/\A^{<0}\A,N)$ because the the space
of $\Hom$'s from any $M\in \OCat_\nu(\A)$ to the left and the right hand side of the previous equality
coincide.

\subsection{Case of quantizations of symplectic resolutions}\label{SS_quant_resol_O}
Now let $X$ be a conical symplectic resolution of an affine variety $X_0$. Let $\A_\lambda$
stand for the algebra of  global sections of the filtered quantization $\A_\lambda^\theta$ of $X$
corresponding to $\lambda\in \tilde{\param}:=H^2(X,\C)$.
Assume that we have a Hamiltonian action $\nu$ of $\C^\times$ on $X$ with finitely many fixed points
that commutes with the contracting $\C^\times$-action.
This action lifts to a Hamiltonian action on a quantization and hence gives rise to a Hamiltonian
action on $\A_\lambda$ again denoted by $\nu$.

The following results were obtained in \cite[Section 5]{BLPW}.

\begin{Prop}\label{Prop:hw_alg}
Fix $\lambda^\circ\in \tilde{\param}$ and pick  $\chi\in \tilde{\param}$ such that $X_\chi$ is affine.
Then, for a Zariski generic $\lambda\in \ell:=\{\lambda^\circ+z\chi, z\in \C\}$, we have the following
\begin{enumerate}
\item  The natural homomorphism $\Ca_\nu(\A_\lambda)\rightarrow \Gamma(\Ca_{\nu}(\A_\lambda^\theta))= \C[X^{\nu(\C^\times)}]$ is an isomorphism.
\item The category $\OCat_\nu(\A_\lambda)$ is highest weight, where the standard objects
are $\Delta_\nu(N_i)$, costandard objects are $\nabla_\nu(N_i)$, $N_i\in \operatorname{Irr}(\Ca_\nu(\A_{\lambda+z\chi}))$.
The order is given as explained before Lemma \ref{Lem:fin_length}.
\item A natural functor $D^b(\OCat_\nu(\A_\lambda))\rightarrow D^b(\A_\lambda\operatorname{-mod})$
is a full embedding.
\end{enumerate}
\end{Prop}

The image of the embedding in (3) is the category $D^b_{\OCat_\nu}(\A_\lambda\operatorname{-mod})$
of all complexes whose homology lie in $\OCat_\nu(\A_\lambda)$.

%
%

We also have the category $\mathcal{O}$ on the level of sheaves.
Following \cite[Section 3.3]{BLPW},
consider the category $\OCat_\nu(\A_\lambda^\theta)$. It consists of all coherent $\A_\lambda^\theta$-modules
that come with a good filtration stable under $h$ and that are supported on the contracting locus $Y$
for $\nu$ in $X$, i.e., $Y:=\{x\in X| \lim_{t\rightarrow 0} \nu(t).x\text{ exists}\}$. Note that $Y$
is a lagrangian subvariety.

The translation equivalences preserve categories $\mathcal{O}_\nu$ and so do $\Gamma$
and $\Loc$, see \cite[Lemma 3.17, Corollary 3.19]{BLPW}.

\begin{Lem}\label{Lem:O_equiv}   A coherent $\A_\lambda^\theta$-module supported on
$Y$ lies in $\OCat_\nu(\A_\lambda^\theta)$ if and only if it can be made weakly $\nu(\C^\times)$-equivariant.
\end{Lem}
\begin{proof}
We may assume that  abelian localization holds for $(\lambda,\theta)$. Let $Y_0$
be the $\nu$-contracting locus in $X_0$ so that $Y=\rho^{-1}(Y_0)$. For $M\in \Coh(\A_\lambda^\theta)$,
we have $\Supp(M)\subset Y$ if and only if $\Supp(\Gamma(M))\subset Y_0$. So it is enough to check
that a finitely generated module $M_0$ supported on $Y_0$ lies in $\OCat_\nu(\A_\lambda)$ if and only if it admits a
weakly $\nu(\C^\times)$-equivariant structure.

Suppose $M_0\in \OCat_\nu(\A_\lambda)$. Then $h$ acts locally finitely on $M_0$. Pick any
map $\varphi:\C\rightarrow \Z$ such that $\varphi(z+1)=\varphi(z)+1$. Then we define a
grading on $M_0$ by requiring that the generalized $z$-eigenspace for $h$ has degree $\varphi(z)$.
This gives a required weakly $\nu(\C^\times)$-equivariant structure.

Conversely, assume that $M_0$ has a weakly $\nu(\C^\times)$-equivariant structure. We can pick a
$\nu(\C^\times)$-equivariant good filtration on $M_0$. Since $\Supp(M_0)\subset Y_0$,
we see that the weights of $\nu(\C^\times)$ in $M_0$ are bounded from above. It follows
that $\A^{>0}$ acts locally nilpotently. So $M\in \OCat_\nu(\A_\lambda)$.
\end{proof}

Again, we assume that $\nu$ is generic.  Pick an ample class $\chi\in \operatorname{Pic}(X)$. Using  equivalences
$\OCat_\nu(\A_\lambda^\theta)\xrightarrow{\sim}\OCat_\nu(\A_{\lambda+n\chi}^\theta)
\xrightarrow{\sim} \OCat_\nu(\A_{\lambda+n\chi})$ for $n\gg 0$, we can carry the highest weight structure from
$\OCat_\nu(\A_{\lambda+n\chi})$  to $\OCat_\nu(\A_\lambda^\theta)$. However, we can use
a different partial order getting the same collection of standard objects.
Namely, for $p\in X^{\nu(\C^\times)}$, let $Y_p$ denote the contracting
component of $p$, $Y_p:=\{x\in X| \lim_{t\rightarrow 0}\nu(t)x=p\}$. Define the partial order
$\preceq_\nu$ on $X^{\nu(\C^\times)}$ by taking the transitive closure of the relation
$p'\preceq p$ given by $p'\preceq p$ if and only if $p'\in \overline{Y_p}$.
Set $Y_{\preceq p}:=\bigcup_{p'\preceq p}Y_{p'}$,
this is a closed lagrangian subvariety in $X$. Further, set
$$\OCat_\nu(\A^\theta_\lambda)_{\preceq p}:=\{M\in \OCat_\nu(\A^\theta_\lambda)| \VA(M)\subset Y_{\preceq p}\}.$$
Here, as usual, we write $\VA(M)$ for the support of $\gr M$ in $X$,
where the associated is taken with respect to a good filtration. It is defined
as the support of the coherent sheaf $\gr M$ for some choice of a good filtration on $M$.

It turns out that $\OCat_\nu(\A^\theta_\lambda)$ is a highest weight category with respect to this order.
The standard objects are localizations of Verma modules $\operatorname{Loc}_{\lambda+n\chi}(\Delta_{\nu}(N))\in
\OCat_{\nu}(\A_{\lambda+n\chi}^\theta)\cong \OCat_\nu(\A_\lambda^\theta)$, see \cite[Section 5.3]{BLPW}.


Below we will need to understand the structure of the category $\OCat_\nu(\A_\lambda^\theta)$
in the case when $X$ splits into the product $X^1\times X^2$ of two conical symplectic resolutions.
Both $\lambda$ and $\theta$ are naturally pairs $(\lambda^1,\lambda^2), (\theta^1,\theta^2)$.


\begin{Lem}\label{Lem:product}
The functor $\OCat_{\nu}(\A^{1\theta^1}_{\lambda^1})\boxtimes \OCat_\nu(\A^{2\theta^2}_{\lambda^2})
\rightarrow \OCat_\nu(\A_\lambda^\theta)$ defined by $M\boxtimes N\rightarrow M\otimes N$ is an equivalence
of highest weight categories.
\end{Lem}
\begin{proof}
Note that we have
\begin{equation}\label{eq:Ext_prod}\dim\operatorname{Ext}^i_{\A_\lambda^\theta}(P^1(p_1)\boxtimes P^2(p_2), L^1(p_1')\boxtimes
L^2(p_2'))=\delta_{i,0}\delta_{p_1,p_1'}\delta_{p_2,p_2'},\end{equation}
where $p_i,p_i'\in X^{i\nu(\C^\times)},i=1,2,$ and $P^i, L^i$ stand for the projective and simple objects
in the categories $\OCat_{\nu}(\A^{i\theta^i}_{\lambda^i})$.
Indeed, (\ref{eq:Ext_prod}) is clear for the categories on the level of algebras and on the level of sheaves follows
from the localization argument. The objects $L^1(p_1')\boxtimes L^2(p_2')$ are the simples in
$\OCat_\nu(\A^{\theta})$. It follows that the objects $P^1(p_1)\boxtimes P^2(p_2)$ are the indecomposable projectives.
\end{proof}

\subsection{Holonomic modules}\label{SS_holon}
In this section, we recall some results from \cite{B_ineq}. Let $X_0$ be a Poisson variety with finitely many
symplectic leaves. Following \cite{B_ineq}, we say that a subvariety $Y_0\subset X_0$ is {\it isotropic}
if its intersection with every leaf is isotropic in the leaf.  For example, when $X_0$ admits a symplectic
resolution of singularities $(X,\rho)$, then $Y_0\subset X_0$ is isotropic if and only if
$\rho^{-1}(Y_0)$ is isotropic, see \cite[Lemma 5.1]{B_ineq}.

%

Now let $X,X_0,\nu$ be as in Section \ref{SS_quant_resol_O}. Let $Y_0\subset X_0$ be the contracting
locus of $\nu$. Then $\rho^{-1}(Y_0)=Y$.
We conclude that $Y_0\subset X_0$ is isotropic.

Let us proceed to holonomic $\A_\lambda$-modules. We say that a finitely generated $\A_\lambda$-module
is {\it holonomic} if its support in $X_0$ is holonomic. In particular, any module in $\OCat_{\nu}(\A_\lambda)$
is supported on $Y_0$ and so is holonomic.

For holonomic modules, we have the following result, see \cite[Theorems 1.2,1.3]{B_ineq}.

\begin{Prop}\label{Prop:B_eq}
Let $N$ be a holonomic $\A_\lambda$-module and $\mathcal{I}$ be its annihilator.
Then $\operatorname{GK-}\dim \A_\lambda/\mathcal{I}=2\operatorname{GK-}\dim N$,
where $\operatorname{GK-}\dim$ stands for the Gelfand-Kirillov dimension.
\end{Prop}

We will need a consequence of this proposition and Lemma \ref{Lem:gen_simplicity}.

\begin{Cor}\label{Cor:full_GK_O}
Let $\lambda,\chi$ be as in Lemma \ref{Lem:gen_simplicity}. Then, for a Weil generic $z$
and any nonzero holonomic $\A_{\lambda+z\chi}$-module $M$, we have $\operatorname{GK-}\dim M=\frac{1}{2}\dim X$.
\end{Cor}

Let us provide one more example of holonomic modules. Let $\B$ be a HC $\A_{\lambda}$-bimodule.
Then it is a holonomic $\A_{\lambda}\otimes \A_{\lambda}^{opp}$-module (its support is the diagonal
in $X_0\times X_0$ that is easily seen to be isotropic).
%

\section{Parabolic induction}\label{S_parab}
\subsection{Hamiltonian actions and fixed point components}\label{SS_Ham_fixed}
Suppose that we have a Hamiltonian action $\nu$ of $\C^\times$ on $X=X^\theta$
that commutes with the contracting action of $\C^\times$.
By the universality of the deformations, this action extends to both
$\tilde{X}^\theta$ and $\tilde{\A}^\theta$. Since $H^1_{DR}(X)=0$, we see
that the actions on $\tilde{X}^\theta,\tilde{\A}^\theta$ are Hamiltonian as well.
For the action on $\tilde{\A}^\theta$, this means, by definition, that there is a
$\nu(\C^\times)$-invariant element $h\in \tilde{\A}$ such that the derivation
$[h,\cdot]$ of $\tilde{\A}^\theta$ coincides with the derivation induced by the
$\C^\times$-action.


\begin{Prop}\label{Prop:sympl_res_torus_fixed}
The fixed point subvariety $X^{\nu(\C^\times)}$ is a conical symplectic resolution
itself.
\end{Prop}
\begin{proof}
The proof is in two steps.


{\it Step 1}. Let us take an irreducible component $Z$ of $X^{\nu(\C^\times)}$. Assume for the time being that
$\dim \rho(Z)=\dim Z$. Apply the Stein decomposition to the projective morphism $Z\rightarrow \rho(Z)$: it factors
through a birational morphism $Z\rightarrow Z_0:=\operatorname{Spec}(\C[Z])$ and a finite dominant morphism
$Z_0\rightarrow \rho(Z)$. The morphism $Z\rightarrow Z_0$ is a symplectic resolution of singularities.
It is conical because $\rho(Z)$ is a closed $\C^\times$-stable subvariety of $X_0$. So it is enough to
check that indeed $\dim \rho(Z)=\dim Z$.

{\it Step 2}. Pick  a Zariski generic one-dimensional subspace  $\ell\subset \tilde{\param}$. Consider the irreducible
component $Z_{\ell}$ of $X_{\ell}^{\nu(\C^\times)}$ containing $Z$, a one-dimensional smooth deformation of $Z$ over
$\ell$. So $\tilde{\rho}(Z_{\ell})$ is an irreducible scheme  over $\ell$. The fiber over zero is $\rho(Z)$.
But, over $\ell\setminus \{0\}$, the map $Z_{\ell}\rightarrow \tilde{\rho}(Z_{\ell})$ is an iso. It follows
that $\dim \rho(Z)=\dim \rho(Z_{\ell})-1=\dim Z_\ell-1=\dim Z$. This completes the proof.
\end{proof}

We also note that the irreducible components of $X^{\nu(\C^\times)}_\chi$ are in a natural bijection
with those of $X^{\nu(\C^\times)}$.

\subsection{Cartan subquotient: sheaf level}\label{SSS_Cart_subquot_sheaf}
We start with a symplectic variety $X$ equipped with a $\C^\times$-action that rescales the symplectic form
and also with a commuting Hamiltonian action $\nu$. Of course, it still makes sense to speak about quantizations
of $X$ that are Hamiltonian for $\nu$.
We want to construct a quantization $\Ca_\nu(\A^{\theta})$ of $X^{\nu(\C^\times)}$ starting from a Hamiltonian quantization
$\A^\theta$ of $X$.

The variety $X$ can be covered by $(\C^{\times})^{2}$-stable open affine subvarieties.
Pick such a subvariety $X'$ with $(X')^{\nu(\C^\times)}\neq \varnothing$.
Define $\Ca_\nu(\A^\theta)(X')$ as $\Ca_\nu(\A^\theta(X'))$.
We remark that the open subsets of the form $(X')^{\nu(\C^\times)}$ form a base of the Zariski topology on $X^{\nu(\C^\times)}$.

The following proposition defines the sheaf $\Ca_\nu(\A^\theta)$.
%

\begin{Prop}\label{Prop:A0}
The following holds.
\begin{enumerate}
\item Suppose that the contracting $\nu$-locus in $X'$ is a complete intersection defined by homogeneous (for $\nu(\C^\times)$) equations of positive weight. Then the algebra $\Ca_\nu(\A^\theta(X'))$ is a  quantization of $\C[X'^{\nu(\C^\times)}]$.
\item There is a unique sheaf $\Ca_\nu(\A^\theta)$ in the conical topology on $X^{\nu(\C^\times)}$
whose sections on $X'^{\nu(\C^\times)}$ with $X'$ as above coincide with $\Ca_\nu(\A^\theta(X'))$.
This sheaf is a quantization of $X^{\nu(\C^\times)}$.
\item If $X'$ is a $(\C^\times)^2$-stable affine subvariety, then $\Ca_\nu(\A^\theta)(X')=\Ca_\nu(\A^\theta(X'))$.
\end{enumerate}
\end{Prop}
\begin{proof}
Let us prove (1). To simplify the notation, we write $\A$ for $\A^\theta(X')$.
The algebra $\A$ is Noetherian because it is complete and separated with respect to
a filtration whose associated graded is Noetherian.   Let us show that
$\operatorname{gr}(\A\A^{>0})=\C[X']\C[X']^{>0}$, this will complete the proof of (1).

In the proof it is more convenient to deal with $\hbar$-adically completed homogenized quantizations. Namely,
let $\A_\hbar$ stand for the $\hbar$-adic completion of $R_\hbar(\A)$. The claim that
$\operatorname{gr}\A\A^{>0}=\C[X']\C[X']^{>0}$ is equivalent to the condition that
$\A_\hbar\A_{\hbar,>0}$ is $\hbar$-saturated meaning that $\hbar a\in \A_{\hbar}\A_{\hbar,>0}$
implies that $a\in \A_{\hbar}\A_{\hbar,>0}$.

Recall that we assume that there are $\nu$-homogeneous elements $f_1,\ldots,f_k\in \C[X']^{>0}$ that form a regular sequence
generating the ideal $\C[X']\C[X']^{>0}$. We can lift those elements to homogeneous $\tilde{f}_1,\ldots, \tilde{f}_k\in \A_{\hbar,>0}$.
We claim that these elements still generate $\A_\hbar\A_{\hbar,>0}$. Indeed, it is enough to check that $\A_{\hbar,>0}\subset \operatorname{Span}_{\A_\hbar}(\tilde{f}_1,\ldots,\tilde{f}_k)$. For a homogeneous element $f\in \A_{\hbar,>0}\setminus \hbar\A_{\hbar}$
we can find homogeneous elements $g_1,\ldots,g_k$
such that $f-\sum_{i=1}^k g_i\tilde{f}_i$ still has the same $\nu(\C^\times)$-weight and is divisible by
$\hbar$. Divide by $\hbar$ and  repeat the
argument. Since the $\hbar$-adic topology is complete and separated, we see that $f\in \operatorname{Span}_{\A_\hbar}(\tilde{f}_1,\ldots,\tilde{f}_k)$. So it is enough to check that $\operatorname{Span}_{\A_\hbar}(\tilde{f}_1,\ldots,\tilde{f}_k)$ is $\hbar$-saturated.

Pick elements $\tilde{h}_1,\ldots,\tilde{h}_k$ such that $\sum_{j=1}^k \tilde{h}_j\tilde{f}_j$ is divisible by $\hbar$.
Let $h_j\in \C[X']$ be congruent to $\tilde{h}_j$ modulo $\hbar$
so that $\sum_{j=1}^k h_j f_j=0$. Since $f_1,\ldots,f_k$ form a regular sequence, we see that
there are  elements $h_{ij}\in \C[X']$ such that $h_{j'j}=-h_{jj'}$ and $h_j=\sum_{\ell=1}^k h_{j\ell}f_\ell$.
Lift the elements $h_{jj'}$ to $\tilde{h}_{jj'}\in \A_\hbar$ with $\tilde{h}_{jj'}=-\tilde{h}_{j'j}$.
So we have $\tilde{h}_j=\sum_{\ell=1}^k \tilde{h}_{j\ell}\tilde{f}_\ell+\hbar\tilde{h}'_j$ for
some $\tilde{h}'_j\in \A_\hbar$.
It follows that $\sum_{j=1}^k \tilde{h}_j \tilde{f}_j= \hbar\sum_{j=1}^k \tilde{h}'_j \tilde{f}_j+
\sum_{j,\ell=1}^k \tilde{h}_{j\ell}\tilde{f}_\ell\tilde{f}_j$. But $\sum_{j,\ell=1}^k \tilde{h}_{j\ell}\tilde{f}_\ell\tilde{f}_j=
\sum_{j<\ell} \tilde{h}_{j\ell}[\tilde{f}_\ell,\tilde{f}_j]$. The bracket is divisible by $\hbar$.
But $\frac{1}{\hbar}[\tilde{f}_\ell, \tilde{f}_j]$ is still in $\A_{\hbar,>0}$ and so in $\operatorname{Span}_{\A_\hbar}(\tilde{f}_1,\ldots,\tilde{f}_k)$.
This finishes the proof of (1).

Let us proceed to the proof of (2).  Let us show that we can choose a covering of $X^{\nu(\C^\times)}$ by
$X'^{\nu(\C^\times)}$, where $X'$ is as in (1). This is easily reduced to
the affine case. Here the existence of such a covering is deduced from the Luna slice theorem applied to a fixed point for  $\nu$.
In more detail, for a fixed point $x$, we can choose an open affine neighborhood $U$ of $x$ in $X\quo \nu(\C^\times)$
with an \'{e}tale morphism $U\rightarrow T_xX\quo \nu(\C^\times)$ such that $\pi^{-1}(U)\cong U\times_{T_xX\quo \nu(\C^\times)}T_xX$,
where $\pi$ stands for the quotient morphism for the action $\nu$. The subset $\pi^{-1}(U)$ then obviously satisfies
the requirements in (1).

It is easy to see that  the algebras $\Ca_\nu(\A^\theta(X'))$ form a presheaf with respect to the covering
$X'^{\nu(\C^\times)}$ (obviously, if $X',X''$ satisfy our assumptions, then their intersection does).
Since the subsets $X'^{\nu(\C^\times)}$ form a base of topology on $X^{\nu(\C^\times)}$, it is enough
to show that the algebras $\Ca_\nu(\A^\theta(X'))$ form a sheaf with respect to the covering. This is easily deduced from the two straightforward claims:
\begin{itemize}
\item $\Ca_\nu(\A^\theta(X'))$ is complete and separated with respect to the filtration (here we use an easy claim
that, being finitely generated, the ideal $\A^\theta(X')\A^{\theta}(X')^{>0}$ is closed).
\item The algebras $\operatorname{gr}\Ca_\nu(\A^\theta(X'))=\C[X'^{\nu(\C^\times)}]$ do form a sheaf --
the structure sheaf $\mathcal{O}_{X^{\nu(\C^\times)}}$.
\end{itemize}
For example, let us show that elements $a_i\in \Ca_{\nu}(\A^\theta(X'_i))$ such that
the restrictions of $a_i,a_j$ to $\Ca_\nu(\A^\theta(X'_i\cap X'_j))$ agree glue
to an element of $\Ca_\nu(\A^\theta(X'))$ (here $X=\bigcup X'_i$ is a covering).
Let $d$ be the common filtration degree of the elements $a_i$. We can find
a homogeneous element $f\in \C[X'^{\nu(\C^\times)}]$ of degree $d$ such that
the degree $d$ part of $a_i$ is $f|_{X'_i}$. Lift $f$ to an element
$a\in \Ca_{\nu}(\A^\theta(X'))$ and replace $a_i$ with $a_i-a$. This reduces
the degree by $1$. Now we can use the descending induction on the degree
(that makes sense because the filtrations involved are complete and separated).

The proof of (2) is now complete.


To prove (3) it is enough to assume that $X$ is affine. Let $\pi$ denote the categorical quotient map
$X\rightarrow X\quo \nu(\C^\times)$. It is easy to see that, for  every open $(\C^\times)^2$-stable affine subvariety $X'$
that intersects $X^{\nu(\C^\times)}$ non-trivially, and any point $x\in X'^{\nu(\C^\times)}$,
there is some $\C^\times$-stable open affine subvariety $Z\subset X\quo \nu(\C^\times)$ with $x\in \pi^{-1}(Z)\subset X'$. So we can assume, in addition, that all covering
affine subsets $X^i$ are of the form $\pi^{-1}(?)$. Moreover, we can assume that they are all principal
(and so are given by non-vanishing of $\nu(\C^\times)$-invariant and $\C^\times$-semiinvariant elements of $\A^\theta(X)$).
Then all algebras $\Ca_\nu(\A^\theta(X'))$ are obtained from $\Ca_\nu(\A^\theta(X))$ by microlocalization.
By standard properties of  microlocalization, the algebra $\Ca_{\nu}(\A^\theta)(X)$ is the only
algebra with these properties. (3) follows.
\end{proof}

\subsection{Comparison between algebra and sheaf levels}\label{SS_Ca_sheaf_vs_alg}
Now let us assume that $X$ is a conical symplectic resolution of $X_0$.
We write $\A^\theta_\lambda$
for the quantization of $X$ corresponding to $\lambda$ and $\A_\lambda$
for its algebra of global sections. By the construction, for any $\lambda\in H^2(X)$,
there are a natural homomorphism $\Ca_\nu(\A_{\lambda})\rightarrow \Gamma(\Ca_\nu(\A_\lambda^\theta))$
as well as a natural homomorphism $\Ca_{\nu}(\C[X_0])\rightarrow \Gamma(\Ca_\nu(\mathcal{O}_X))=
\C[X^{\nu(\C^\times)}]$. Our goal in this section is to prove the following result.

\begin{Prop}\label{Prop:A0_descr}
For a Zariski generic $\lambda\in \tilde{\param}$,
the homomorphism $\Ca_\nu(\A_{\lambda})\rightarrow \Gamma(\Ca_\nu(\A^\theta_\lambda))$ is an isomorphism.
\end{Prop}

In the proof we will need a lemma describing some properties of the natural
morphism $X^{\nu(\C^\times)}\rightarrow \operatorname{Spec}(\Ca_\nu(\C[X_0]))$.

\begin{Lem}\label{Lem:A0_classic}
The morphism $X^{\nu(\C^\times)}\rightarrow \operatorname{Spec}(\Ca_\nu(\C[X_0]))$
is surjective and projective.
Further, the scheme $\operatorname{Spec}(\Ca_\nu(\C[X_0]))$ has finitely many symplectic leaves.
\end{Lem}
\begin{proof}
Let $Y_0$ denote the contracting locus for the action of $\nu(\C^\times)$ on $X_0$, we view $Y_0$ as a subscheme of $X_0$ with the defining ideal $\C[X_0]\C[X_0]^{>0}$ so that $\operatorname{Spec}(\Ca_\nu(\C[X_0]))=Y_0\quo \nu(\C^\times)$.
The embedding $X_0^{\nu(\C^\times)}\hookrightarrow Y_0$ gives rise to
the morphism $X_0^{\nu(\C^\times)}\rightarrow \operatorname{Spec}(\Ca_\nu(\C[X_0]))$.
The ideal of $X_0^{\nu(\C^\times)}$ in $Y_0$ contains all elements of
positive degree with respect to $\nu$. Therefore the morphism
$X_0^{\nu(\C^\times)}\rightarrow \operatorname{Spec}(\Ca_\nu(\C[X_0]))$ is also
an embedding. It is a bijection on the level of points,
so $X_0^{\nu(\C^\times)}=\operatorname{Spec}(\Ca_\nu(\C[X_0]))_{red}$. On the
other hand, the morphism $X^{\nu(\C^\times)}\rightarrow \operatorname{Spec}(\Ca_\nu(\C[X_0]))$
we study is the composition of the surjective projective morphism $\rho:X^{\nu(\C^\times)}\rightarrow X_0^{\nu(\C^\times)}$ and the embedding $X_0^{\nu(\C^\times)}\hookrightarrow \operatorname{Spec}(\Ca_\nu(\C[X_0]))$.
It follows that $X^{\nu(\C^\times)}\rightarrow \operatorname{Spec}(\Ca_\nu(\C[X_0]))$
is surjective and projective.

We see that $\operatorname{Spec}(\Ca_\nu(\C[X_0]))_{red}$ admits a symplectic resolution
$X^{\nu(\C^\times)}$ (see Proposition \ref{Prop:sympl_res_torus_fixed})
so it has finitely many leaves.
\end{proof}

\begin{proof}[Proof of Proposition \ref{Prop:A0_descr}]
Recall that $\tilde{X}$ denotes the universal deformation of $X$ over $\tilde{\param}$ and $\tilde{X}_0$ be its affinization. Consider
the natural homomorphism $\Ca_\nu(\C[\tilde{X}_0])\rightarrow \C[\tilde{X}^{\nu(\C^\times)}]$. It is an isomorphism outside of $\mathcal{H}_\C$ (defined by (\ref{eq:singular}))
since $\tilde{X}\rightarrow \tilde{X}_0$ is an isomorphism precisely outside that locus.
Now consider the canonical quantization $\tilde{\A}^\theta$ of $\tilde{X}$. Similarly to the previous
section, $\Ca_\nu(\tilde{\A}^{\theta})$ is a quantization of $\tilde{X}^{\nu(\C^\times)}$. The equality
$H^i(X^{\nu(\C^\times)},\mathcal{O})=0, i>0,$ (that is a consequence of Proposition \ref{Prop:sympl_res_torus_fixed}) implies that  $H^i(\tilde{X}^{\nu(\C^\times)},\mathcal{O})=0, i>0$. It follows that $\gr \Gamma(\Ca_\nu(\tilde{\A}^\theta))=\C[\tilde{X}^{\nu(\C^\times)}]$.
Also there is a natural epimorphism $\Ca_\nu(\C[\tilde{X}_0])\twoheadrightarrow \gr\Ca_\nu(\tilde{\A})$ and a natural homomorphism $\gr\Ca_\nu(\tilde{\A})\rightarrow \gr \Gamma(\Ca_\nu(\tilde{\A}^{\theta}))$. The resulting homomorphism $\gr\Ca_\nu(\tilde{\A})\rightarrow \gr\Gamma(\Ca_\nu(\tilde{\A}^{\theta}))$ is, on one hand, the associated graded of the homomorphism $\Ca_\nu(\tilde{\A})\rightarrow\Gamma(\Ca_\nu(\tilde{\A}^{\theta}))$ and on the other hand, an isomorphism over  $\tilde{\param}\setminus\mathcal{H}_\C$. We deduce that the supports of the associated graded modules of the kernel and the cokernel of
$\Ca_\nu(\tilde{\A})\rightarrow\Gamma(\Ca_\nu(\tilde{\A}^{\theta}))$ are supported on $\mathcal{H}_{\C}$ as $\C[\tilde{\param}]$-modules.

Now let us check that this kernel and the cokernel are HC bimodules over $\Ca_\nu(\tilde{\A})$.
For the kernel it is obvious. For the cokernel, we need to check that $\Gamma(\Ca_\nu(\tilde{\A}^\theta))$
is a HC $\Ca_\nu(\tilde{\A})$-bimodule.
This will follow if we check that $\C[X^{\nu(\C^\times)}]=[\gr \Gamma(\Ca_\nu(\tilde{\A}^\theta))]/(\tilde{\param})$ is finitely
generated over $[\gr \Ca_\nu(\tilde{\A})]/(\tilde{\param})$. The latter algebra is a quotient
of $\Ca_{\nu}(\C[X_0])$. Now the claim about the cokernel follows from Lemma
\ref{Lem:A0_classic}.

Also thanks to Lemma \ref{Lem:A0_classic}, $\operatorname{Spec}(\Ca_\nu(\C[X_0]))$ has finitely many symplectic
leaves. It follows that $\Ca_\nu(\tilde{\A})$ satisfies the assumptions of
Proposition \ref{Prop:supp_gr}. Applying this proposition to the kernel and the cokernel
above, we see that their supports in $\param$ are closed subvarieties whose asymptotic
cones lie in $\mathcal{H}_{\C}$. So, for a Zariski generic parameter $\lambda$,
the homomorphism $\Ca_{\nu}(\tilde{\A})_\lambda\rightarrow \Gamma(\Ca_\nu(\tilde{\A}^\theta))_\lambda$
is an isomorphism.

We note that $\Gamma(\Ca_\nu(\tilde{\A}^{\theta}))$ is flat over $\param$
and the specialization at $\lambda$ coincides with $\Gamma(\Ca_\nu(\A_{\lambda}^{\theta}))$, this is because
$X^{\nu(\C^\times)}$ is a symplectic resolution, Proposition \ref{Prop:sympl_res_torus_fixed}.
Further,  $\Ca_\nu(\tilde{\A})_\lambda=\Ca_\nu(\A_\lambda)$.
This implies the claim of the proposition.
\end{proof}

\subsection{Correspondence between parameters}\label{SS_Ca_param}
Our next goal is to understand how to recover the periods of the direct summands $\Ca_\alpha(\A^{\theta})$ from that of $\A^\theta$. We will assume that $H^i(X^{\alpha(\C^\times)},\mathcal{O})=0$ for $i>0$,
but we will not require that of $X$. The period map still makes sense  for $X$, see \cite[Section 4]{BK}.
Consider the decomposition $X^{\alpha(\C^\times)}=\bigsqcup_i X^0_i$ into connected components. Let $Y_i$ denote the
contracting locus of $X^0_i$ and let $\A_{i}^{\theta 0}$ be the restriction of $\Ca_\alpha(\A^\theta)$ to $X^0_i$.
To determine the period of $\A_i^{\theta 0}$, we will quantize $Y_i$ and then use results from  \cite{BGKP}
on quantizations of line bundles on lagrangian subvarieties.

First of all, let us consider the case when $X$ is affine and so is quantized by a single algebra, $\A$. We will quantize the contracting locus $Y$ by a single $\A$-$C_\nu(\A)$-bimodule, this bimodule is $\A/\A\A^{>0}$.

\begin{Lem}\label{Lem:repel_quant_affine}
Under the above assumptions,  $\gr\A/\A\A^{>0}=\C[Y]$.
\end{Lem}
\begin{proof}
This was established in the proof of Proposition \ref{Prop:A0}. More precisely, the case when $Y$ is a complete intersection given by $\alpha(\C^\times)$-semiinvariant elements of positive weight follows from the proof of assertion (1), while the general case follows similarly to the proof of (3).
\end{proof}

Now let us consider the non-affine case. Let us cover $X\setminus \bigcup_{k\neq i}X^0_k$ with $(\C^\times)^2$-stable open affine subsets $X^j$. We may assume that $X^j\cap Y_i=\varnothing$ or $X^j\cap Y_i=\pi_i^{-1}(X^j\cap X_i^0)$,
where $\pi_i:Y_i\rightarrow X_i^0$ is the projection. For this we first choose some covering by $(\C^\times)^2$-stable
open affine subsets. Then we delete $Y_i\setminus \pi_i^{-1}(X^j\cap X_i^0)$ from each $X^j$, we still have a covering. We cover the remainder of each $X^j$ by subsets that are preimages of open affine subsets on $X^j\quo \nu(\C^\times)$ that do not
intersect $X^0_j\quo \nu(\C^\times)$. It is easy to see that this covering has required properties. Let us replace $X$ with the union of $X^j$ that intersect $Y_i$.

After this replacement, we can quantize $Y_i$ by a $\A^\theta$-$\Ca_\nu(\A^\theta)$-bimodule. We have natural $\A^\theta(X^j)$-$\Ca_\nu(\A^\theta)(X^j\cap X_i^0)$-bimodule structures on $\A^\theta(X^j)/\A^\theta(X^j)\A^\theta(X^j)^{>0}$
and glue the bimodules corresponding to different $j$ together
along the intersections $X^i\cap X^j$ (we have homomorphisms $\A^\theta(X^i)\rightarrow \A^\theta(X^i\cap X^j)$
that give rise to $\A^\theta(X^i)/\A^\theta(X^i)\A^\theta(X^i)^{>0}\rightarrow \A^\theta(X^i\cap X^j)/\A^\theta(X^i\cap X^j)\A^\theta(X^i\cap X^j)^{>0}$ and to $\Ca_\nu(\A^\theta(X^i))\rightarrow \Ca_\nu(\A^\theta(X^i\cap X^j))$). Similarly to the proof of (2) in Proposition \ref{Prop:A0}, we get a sheaf of $\A^\theta$-$\Ca_\nu(\A^{\theta})$-bimodules on $Y_i$ that we denote by $\A^\theta/\A^\theta \A^{\theta,>0}$. The following is a direct consequence of  the construction.

\begin{Lem}\label{Lem:repel_gener}
The associated graded of $\A^\theta/\A^{\theta}\A^{\theta,>0}$  coincides with the
$\mathcal{O}_X$-$\mathcal{O}_{X^0_i}$-bimodule $\mathcal{O}_{Y_i}$.
\end{Lem}

Now we want to interpret $Y$ as a lagrangian subvariety in $X\times X^0$
(we still use $X$ as in the paragraph preceding Lemma \ref{Lem:repel_gener}, and so can write $Y$ instead of $Y_i$ and $X^0$ instead of $X^0_i$).  Namely, let $\iota$ denote the inclusion $Y\hookrightarrow X$ and $\pi$ be the projection $Y\rightarrow X^0$. We embed $Y$ into $X\times X^0$ via $(\iota,\pi)$. We equip $X\times X^0$
with the symplectic form $(\omega,-\omega^0)$, where $\omega^0$ is the restriction of $\omega$ to $X^0$.
With respect to this symplectic form $Y$ is a lagrangian subvariety. Further, $\A^\theta\widehat{\otimes}\Ca_\nu(\A^\theta)^{opp}$
is a quantization of $X\times X^0$ with period $(\lambda,-\lambda^0)$, where $\lambda,\lambda^0$ are periods of
$\A^\theta, \Ca_\nu(\A^\theta)$.

\begin{Prop}\label{Prop:shift}
The period  $\lambda^0$ coincides with  $\iota^{0*}(\lambda+c_1(K_Y)/2)\in H^2(X^0)=H^2(Y)$,
where $K_Y$ denotes the canonical class of $Y$ and $\iota^0$ is the inclusion $X^0\hookrightarrow X$.
\end{Prop}
\begin{proof}
The period of $\A^\theta\widehat{\otimes}\Ca_\nu(\A^\theta)^{opp}$ coincides with $p_1(\lambda)-p_2(\lambda^0)$,
where $p^*_1:X\times X^0\rightarrow X, p^*_2:X\times X^0\rightarrow X^0$ are the projections.
So the pull-back of the period to $Y$ is $\iota^*(\lambda)-\pi^*(\lambda^0)$. The structure sheaf of $Y$
admits a quantization to a $\A^\theta\widehat{\otimes}\Ca_{\nu}(\A^\theta)^{opp}$-bimodule,
By \cite[(1.1.3),Theorem 1.1.4]{BGKP}, we have $\iota^*(\lambda)-\pi^*(\lambda^0)=-\frac{1}{2}c_1(K_Y)$.
Restricting this equality to $X^0$, we get the equality required in the proposition.
%
\end{proof}

\subsection{Parabolic induction functor}\label{SS_parab_induc}
Let $X$ be a conical symplectic resolution of $X_0$. We assume that $X$ comes with a Hamiltonian action of a torus
$T$. Let $\mathfrak{C}$ stand for $\operatorname{Hom}(\C^\times,T)$. We introduce
a  pre-order $\prec^\lambda$ on $\mathfrak{C}$ as follows: $\nu\prec^\lambda \nu'$ if
\begin{itemize}
\item
$\A_\lambda \left(\A_{\lambda}^{>0,\nu}+ (\A_\lambda^{\nu(\C^\times)})^{>0,\nu'}\right)=\A_\lambda\A_{\lambda}^{>0,\nu'}$,
\item and the natural action of $\nu(\C^\times)$ on $\Ca_{\nu'}(\A_\lambda)$ is trivial.
\end{itemize}
It extends
naturally to $\mathfrak{C}_{\mathbb{Q}}:=\mathbb{Q}\otimes_{\Z}\mathfrak{C}$.

The following lemma explains why this ordering is important.

\begin{Lem}\label{Lem:parab_ind}
Suppose $\nu\prec\nu'$. Then $\Ca_{\nu'}(\Ca_\nu(\A_\lambda))=\Ca_{\nu'}(\A_\lambda)$. Further, let
$\Delta_{\nu'}: \Ca_{\nu'}(\A_\lambda)\operatorname{-mod}\rightarrow \A_\lambda\operatorname{-mod}, \Delta_{\nu}:
\Ca_\nu(\A_\lambda)\operatorname{-mod}\rightarrow \A_\lambda\operatorname{-mod}, \underline{\Delta}:\Ca_{\nu'}(\A)\operatorname{-mod}
\rightarrow \Ca_\nu(\A_\lambda)\operatorname{-mod}$ be the Verma module functors. We have $\Delta_{\nu'}=\Delta_\nu\circ\underline{\Delta}$.
\end{Lem}
\begin{proof}
Let us write $\A$ for $\A_\lambda$.

Let us prove that $\Ca_{\nu'}(\Ca_\nu(\A))=\Ca_{\nu'}(\A)$.
(i) produces an epimorphism $\A/\A\A^{>0,\nu}\twoheadrightarrow \A/\A\A^{>0,\nu'}$.
By (ii), the latter gives rise to an algebra epimorphism $\Ca_{\nu}(\A)\twoheadrightarrow \Ca_{\nu'}(\A)$
that induces $\Ca_{\nu'}(\Ca_\nu(\A))\twoheadrightarrow \Ca_{\nu'}(\A)$.

To show that $\Delta_{\nu}\circ \underline{\Delta}\cong \Delta_{\nu'}$ we note that
the right hand side is $\A/I_r\otimes_{\Ca_{\nu'}(\A)}\bullet$, while the
left hand side is $\A/I_\ell\otimes_{\Ca_{\nu'}(\A)}\bullet$, where we write
$I_\ell,I_r$ for the ideals of $\A$ in the left and the right hand side of the
equality in (i).
\end{proof}

The lemma shows that the Verma module functor can be studied in stages. This is what we mean by the parabolic induction.

Here is an example  of two one-parameter subgroups that are comparable with respect to $\prec^\lambda$.
Pick one-parameter subgroups $\nu,\kappa:\C^\times\rightarrow T$.

\begin{Lem}\label{Lem:prec_spec}
For $m\gg 0$, we have $\nu\prec^\lambda \nu':=m\nu+\kappa$ for all $\lambda$.
\end{Lem}
\begin{proof}
Let us write $\A$ for $\A_\lambda$.
Let us prove the inclusion $\A^{>0,\nu}\subset \A\A^{>0,\nu'}$ that will
imply $\A\left(\A^{>0,\nu}+(\A^{\nu(\C^\times)})^{>0,\nu'}\right)\subset \A\A^{>0,\nu'}$.

The algebra $\gr(\A^{\geqslant 0,\nu})=\C[X_0]^{\geqslant 0,\nu}$
is finitely generated, as in the proof of \cite[Lemma 3.1.2]{GL}. So we can choose finitely many $T$-semiinvariant generators of the ideal $\C[X_0]^{>0,\nu}$ in $\C[X_0]^{\geqslant 0,\nu}$, say $f_1,\ldots,f_k$.
Let $\tilde{f}_1,\ldots,\tilde{f}_k$ denote their lifts to $T$-semiinvariant elements in $\A$,
these lifts are generators of the ideal $\A^{>0,\nu}$ in $\A^{\geqslant 0,\nu}$.
Let $a_1,\ldots,a_k>0$ be their weights for $\nu$ and $b_1,\ldots,b_k$ be their weights for $\kappa$.
Take $m\in \Z^{>0}$ such that $ma_i+b_i>0$ for all $i$. The elements $\tilde{f}_1,\ldots,\tilde{f}_k$ then
lie in $\A^{>0,\nu'}$ and so $\A\A^{>0,\nu}\subset \A\A^{>0,\nu'}$.

The inclusion $\A^{>0,\nu'}\subset \A\left(\A^{>0,\nu}+(\A^{\nu(\C^\times)})^{>0,\nu'}\right)$
is proved in a similar fashion. So we get the required equality of the left ideals.

Let us check that $\nu(\C^\times)$ acts trivially on $\Ca_{\nu'}(\A)$. We get
$$\Ca_{\nu'}(\A)=\left(\A/\A(\A^{>0,\nu}\oplus (\A^{\nu(\C^\times)})^{>0,\kappa})\right)^{\nu'(\C^\times)}.$$
Note that the expression in brackets is independent of $m$ and the weights of
$\nu$ are all non-negative. From here it is easy to deduce
that the invariants with respect to $\nu'(\C^\times)$ coincides with the invariants
for $\nu(\C^\times)\kappa(\C^\times)$ for $m$ large enough.
\end{proof}

\section{Standardly stratified structures on categories $\mathcal{O}$}\label{S_SS_O}
In this section we introduce a family of standardly stratified structures on the category
$\mathcal{O}_\nu(\A_\lambda^\theta)$ that are compatible
in the sense of  Section \ref{SSS_SS_comp} with the highest weight structure
recalled in Section \ref{SS_quant_resol_O}.

\subsection{Main result}
Suppose that we have a Hamiltonian action of a torus $T$ on $X$ (commuting with the contracting action)
such that the fixed point set $X^T$ is finite. Let $\chi_1,\ldots,\chi_N$ be the characters
of the $T$-action on $$\bigoplus_{p\in X^T}T_pX$$
The kernels $\ker\chi_i, i=1,\ldots,N,$ partition the co-character space $\operatorname{Hom}(\C^\times,T)\otimes_{\Z}\mathbb{R}$ into polyhedral cones to
be called {\it chambers}. If $\nu$ lies in the interior of a chamber (i.e., $\langle\chi_i,\nu\rangle\neq 0$ for all $i$), then $X^{\nu(\C^\times)}=X^T$. Such one-parameter subgroups will be called
{\it generic}. Otherwise, $X^{\nu(\C^\times)}$ is infinite.

Let $\nu$ be a generic one-parameter subgroup of $T$ and $\nu_0$ be a one-parameter subgroup
lying in the closure of the chamber containing $\nu$. In this case we will write $\nu\rightsquigarrow
\nu_0$. Our goal is to produce a standardly stratified structure on $\OCat_\nu(\A_\lambda^\theta)$ corresponding to $\nu_0$. The first step is to define a pre-order on $X^T$ from $\nu_0$.

Let $Z_1,\ldots,Z_k$ be the irreducible (=connected) components of $X^{\nu_0(\C^\times)}$.
We define a partial order $\preceq_{\nu_0}$ on the set $\{Z_1,\ldots,Z_k\}$ as the transitive closure
of the relation $Z_i\preceq_{\nu_0}Z_j$ specified by $Z_i\cap \overline{Y_j}\neq \varnothing$, where $Y_j$
is the contracting locus of $Z_j$, i.e., $$Y_j=\{x\in X| \lim_{t\rightarrow 0} \nu_0(t)x\in Z_j\}.$$
For $p,p'\in X^T$, we set $p\preceq_{\nu_0}p'$ if $p\in Z_i, p'\in Z_j$ with $Z_i\prec_{\nu_0}Z_j$.
Clearly, $\preceq_{\nu_0}$ is a partial pre-order on the set $X^T$ and the associated poset is
$\{Z_1,\ldots,Z_k\}$.

\begin{Lem}\label{Lem:order_compat}
The relation $p\preceq_\nu p'$ implies $p\preceq_{\nu_0}p'$.
\end{Lem}
\begin{proof}
In the proof we can assume that $p\in \overline{Y}_{p'}$. In this case, what we need to check
is that $\nu_0$ contracts $Y_{p'}$ to the irreducible component $Z$ of $X^{\nu_0(\C^\times)}$
containing $p'$. We can replace $X$ with a $T$-stable affine neighborhood of $p'$.
Applying the Luna slice theorem to $T$ and $p'$, we see that $Y_{p'}$ is $T$-equivariantly
isomorphic to a $T$-module, such that all weights of $\nu$ are positive. Since $\nu_0$
lies in the closure of the chamber containing $\nu$, we see that all weights of $\nu_0$
are non-negative. This implies our claim.
\end{proof}

Let us proceed to the main result of this section.

\begin{Thm}\label{Thm:cat_O_ss}
Suppose $\nu_0$ lies in the closure of the chamber containing $\nu$. Then $(\OCat_{\nu}(\A_\lambda^\theta),\preceq_{\nu_0})$
is a standardly stratified category.
\end{Thm}

In the proof we will always assume that $\nu=m\nu_0+\kappa$ for a fixed $\kappa$ and $m\gg 0$.

Let us explain our scheme of proof of Theorem \ref{Thm:cat_O_ss}.
We will start by describing the associated graded category $\gr \mathcal{O}_\nu(\A_\lambda^\theta)$,
we will show that this category is $\OCat_\nu(\Cat_{\nu_0}(\A_\lambda^\theta))$ and that, under the identifications
$\OCat_\nu(\A_\lambda^\theta)\cong \OCat_{\nu}(\A_{\lambda+n\chi})$ and $\OCat_\nu(\Cat_{\nu_0}(\A_\lambda^\theta))
\cong \OCat_{\nu}(\Ca_{\nu_0}(\A_{\lambda+n\chi}))$ (where $\chi$ is ample and $n\gg 0$) the standardization
functor becomes the Verma module functor $\Delta_{\nu_0}$. Using this, we will check the standardization
functor is exact confirming (SS1), this reduces to the claim that the
right $\Ca_{\nu_0}(\A_{\lambda+n\chi})$-module $\A_{\lambda+n\chi}/\A_{\lambda+n\chi}(\A_{\lambda+n\chi})_{>0,\nu_0}$
is projective. Finally, we will use Lemma \ref{Lem:ss_order} to complete the proof.

Below we will write $\lambda^+$ for $\lambda+n\chi$ with $n\gg 0$.

\subsection{Associated graded category}
The goal of this part is to prove the following proposition.

\begin{Prop}\label{Prop:assoc_graded_O}
Let $Z$ be an irreducible component of $X^{\nu_0(\C^\times)}$. The subquotient $\OCat_\nu(\A_\lambda^\theta)_Z$
is equivalent to  $\OCat_\nu(\Ca_{\nu_0}(\A_\lambda^\theta)|_Z)$ as an abelian category so that
the functors $\pi^!_Z$ and $\pi^*_Z$ get identified with $\Delta_{\nu_0}$ and $\nabla_{\nu_0}$,
respectively.
\end{Prop}

In order to prove this proposition, we need to understand the ``highest weights'' of modules in
$\OCat_\nu(\A_{\lambda+z\chi})$. We can identify
$\Ca_{\nu}(\A_{\lambda^+})$ with $\C[(X^\theta)^T]$. Let $\ell$ denote the line $\{\lambda+z\chi| z\in \C\}$.
Let $\Phi_{\ell}:\mathfrak{t}\rightarrow \A_{\ell,\leqslant d}$ be a quantum
comoment map for the action of $T$ on $\A_{\ell}$. Note that $\Phi_\ell$ is defined up to
adding a linear function $\mathfrak{t}\rightarrow \operatorname{Aff}(\ell,\C)$, where
$\operatorname{Aff}(\ell,\C)$ stands for the space of affine functions $\ell\rightarrow \C$.
Then we have a linear function $c^p_{\lambda+z\chi}:\mathfrak{t}\rightarrow \C$ defined as the composition
$\eta_p\circ\varphi\circ\Phi_{\lambda+z\chi}$, where the quantum comoment map $\Phi_{\lambda+z\chi}$
is viewed as a map $\mathfrak{t}\rightarrow \A_{\lambda+z\chi}^{\nu(\C^\times)}$, by
$\varphi$ we denote the natural homomorphism $\A_{\lambda+z\chi}^{\nu(\C^\times)}\rightarrow
\Gamma(\Ca_\nu(\A^\theta_{\lambda+z\chi}))$, and $\eta_p$ stands for the projection
$\Gamma(\Ca_\nu(\A^\theta_{\lambda+z\chi}))\rightarrow \C$ to the one-dimensional
quotient corresponding to $p\in X^T$. By the construction, $\mathfrak{t}$ acts on
the top degree component of $\Delta_\nu(p)$
by the character $c^p_{\lambda+z\chi}$. We are going to investigate
the properties of $c^p_{\lambda+z\chi}$.

\begin{Lem}\label{Lem:highest_weights}
The following is true.
\begin{enumerate}
\item The function $c^p_{\lambda+z\chi}$ is affine in $z$.
\item Let $\mathfrak{t}_0\subset \mathfrak{t}$ be the Lie subalgebra spanned by the face
containing $\nu_0$ with $\nu\rightsquigarrow \nu_0$.
If $p,p'$ lie in the connected component of $X^{\nu_0(\C^\times)}$,
then $c^p_{\lambda+z\chi}, c^{p'}_{\lambda+z\chi}$ coincide on $\mathfrak{t}_0$.
\item For two points $p,p'\in X^T$, the linear part of the difference $c^p_{\lambda+z\chi}-c^{p'}_{\lambda+z\chi}$
equals $z(\alpha_p(\chi)-\alpha_{p'}(\chi))$, where $\alpha_p(\chi)$ (resp., $\alpha_{p'}(\chi)$)
is the character of the action of $T$ on the fiber $\mathcal{O}(\chi)_p$ (resp., on $\mathcal{O}(\chi)_{p'}$).
Note that $\alpha_p(\chi)$ depends on the choice of the $T$-equivariant structure on $\mathcal{O}(\chi)$
but the difference does not.
\end{enumerate}
\end{Lem}
\begin{proof}
We can consider the universal function $c^p:\mathfrak{t}\rightarrow \C[\ell]$,  the composition
of the quantum comoment map $\Phi: \mathfrak{t}\rightarrow  (\A_\ell^{\nu(\C^\times)})_{\leqslant d}$,
the natural homomorphism $\A_\ell^{\nu(\C^\times)}\rightarrow \Ca_{\nu}(\A_\ell)\rightarrow
\Gamma(\Ca_{\nu}(\A_\ell^\theta))$, and the projection $\Gamma(\Ca_{\nu}(\A_\ell^\theta))\twoheadrightarrow
\C[\ell]$ corresponding to $p$. Note that $\operatorname{im}c^p\in \operatorname{Aff}(\ell,\C)$ and that $c^p_{\lambda+z\chi}$ is the specialization of $c^p$ to $\lambda+z\chi$. This implies (1).

Let us proceed to (2).  By Lemma \ref{Lem:prec_spec},
$\A_{\lambda+z\chi}^{\nu(\C^\times)}\rightarrow \Ca_\nu(\A_{\lambda+z\chi})$
factors through $\A_{\lambda+z\chi}^{\nu(\C^\times)}\rightarrow \Ca_{\nu_0}(\A_{\lambda+z\chi})^{\nu(\C^\times)}$. The image of $\Phi(\mathfrak{t}_0)$ in $\Ca_{\nu_0}(\A_{\lambda+z\chi})$
is central and so, for any $\xi\in \mathfrak{t}_0$, the projection of $\Phi(\xi)$ to the direct summand
corresponding to any component of $X^{\nu_0(\C^\times)}$ is constant. This implies (2).

Let us prove (3). The linear part of the map $c^p_{\lambda+z\chi}$ coincides with the composition
of $\mu^*:\mathfrak{t}^*\rightarrow \C[X_{\C\chi}]^{\nu(\C^\times)}$ and the projection
$\C[X_{\C\chi}]^{\nu(\C^\times)}\rightarrow \C[z]$ corresponding to $p$. So this linear part
is $z\mu(p_1)$, where $p_1\in X_\chi^{T}$ is the point corresponding to $p$
(see the end of Section \ref{SS_Ham_fixed}). We now need to show that $\mu(p_1)-\mu(p_1')=
\alpha_p(\chi)-\alpha_{p'}(\chi)$.

Let us choose a $T\times \C^\times$-equivariant structure on $\mathcal{O}(\chi)$. Let $\mathcal{L}$
denote a unique $T\times \C^\times$-equivariant extension of $\mathcal{O}(\chi)$ to $X_{\C\chi}$.
According to \cite{Kaledin_proj}, the complement $\mathcal{L}^\times$ of the zero section in
$\mathcal{L}$ is a symplectic variety with a Hamiltonian action of $\C^\times$, whose moment
map $\mu_{\C^\times}$ is the projection $\mathcal{L}^\times\rightarrow \C\chi$. The $T$-equivariant structure on
$\mathcal{L}$ induces a moment map $\mu_T$ for the action of $T$ on $\mathcal{L}^\times$,
the induced moment map on $X=\mu_{\C^\times}^{-1}(0)/\!/\C^\times$ coincides with the original
one. Let us pick $p\in X^T$, consider the corresponding fixed point line $P\subset X_{\C\chi}^T$
and let $\tilde{P}$ be its preimage in $\mathcal{L}^\times$. The latter is easily seen to be
a symplectic subvariety in $\mathcal{L}^\times$. This symplectic variety is nothing else
but $T^*\C^\times$, where the zero section $\C^\times$ is the preimage of $p$ in $\mathcal{L}^\times$.
The action of $T$ on the base $\C^\times$ is induced by the character $\alpha_p(\chi)$. We conclude
that the restriction of $\mu_T$  to $P\cong \C$ is also given by (the differential of) $\alpha_p(\chi)$.
This  shows the claim in the previous
paragraph and completes the proof of the lemma.
\end{proof}

We need one more lemma, see \cite[Section 3.2.4]{MO}.

\begin{Lem}\label{Lem:ineq_char}
Let $p,p'\in X^T$ be such that $p\prec_{\nu_0}p'$ and let $\chi,\chi'$ be the characters
of $T$ in the fibers $\Str(\chi)_p, \Str(\chi)_{p'}$. Then $\langle\chi,\nu_0\rangle<\langle\chi',\nu_0\rangle$.
\end{Lem}

Now we are ready to prove Proposition \ref{Prop:assoc_graded_O}.

\begin{proof}[Proof of Proposition \ref{Prop:assoc_graded_O}]
Let us produce an exact functor $$\varpi_Z:\OCat_{\nu}(\A_{\lambda^+})_{\preceq Z}
\twoheadrightarrow\OCat_\nu(\Ca_{\nu_0}(\A_{\lambda^+})|_Z).$$  Let $h_0$
be the image of $1$ under the quantum comoment map for $\nu_0$. For
$M\in \OCat_{\nu}(\A_{\lambda^+})_{\preceq Z}$, let $\varpi_Z(M)$
be the generalized eigen-space for $h_0$ with eigenvalue $c^p_{\lambda^+}(h_0)$,
where $p\in Z$ (this depends only on $Z$ by (2) of Lemma \ref{Lem:highest_weights}).
(3) of Lemma \ref{Lem:highest_weights} and the choice of $\lambda^+$ imply
that, for any eigenvalue $\alpha$ of $h_0$ in $M:=\Delta_{\nu}(p')$, where $p'\in Z'$
with $Z'\preceq Z$, we have  $\alpha - c^p_{\lambda^+}(h_0)\in \Z_{\leqslant 0}$.
It follows that the similar conclusion is true for any
$M \in\OCat_{\nu}(\A_{\lambda^+})_{\preceq Z}$. So $\A_{\lambda^+}^{>0,\nu_0}$ annihilates $\varpi_Z(M)$
and therefore $\varpi_Z$ is an exact functor
$\OCat_{\nu}(\A_{\lambda^+})_{\preceq Z}\rightarrow \OCat_{\nu}(\Ca_{\nu_0}(\A_{\lambda^+})|_Z)$.
From  Lemma \ref{Lem:ineq_char} it follows that $c^p_{\lambda^+}(h_0)$ is not an
eigenvalue of $h_0$ on $\Delta_{\nu}(p')$ with $p'\in Z'$ and $Z'\prec Z$.
We conclude that $\varpi_Z$ annihilates  $\OCat_{\nu}(\A_{\lambda^+})_{\prec Z}$.
Also we note that, for $N\in \OCat_\nu(\Ca_{\nu_0}(\A_{\lambda^+})|_Z)$, we have
$\varpi_Z\circ \Delta_{\nu_0}(N)\cong N$. We conclude that the essential image of
$\varpi_Z$ is $\OCat_\nu(\Ca_{\nu_0}(\A_{\lambda^+})|_Z)$ and that
$\varpi_Z$ descends to an equivalence $\OCat_{\nu}(\A_{\lambda^+})_{\preceq Z}/
\OCat_{\nu}(\A_{\lambda^+})_{\prec Z}\xrightarrow{\sim}\OCat_{\nu}(\Ca_{\nu_0}(\A_{\lambda^+})|_Z)$.

It remains to check that $\Delta_{\nu_0,Z}$ is the left adjoint of $\varpi_Z$, while
$\nabla_{\nu_0,Z}$ is the right adjoint of $\varpi_Z$. To show the former, note that $$\Hom_{\A_{\lambda^+}}(\Delta_{\nu_0}(M),N)=
\Hom_{\Ca_{\nu_0}(\A_{\lambda^+})}(M,N^{\A^{>0,\nu_0}_{\lambda^+}})$$
for $M\in \OCat_\nu(\Ca_{\nu_0}(\A_{\lambda^+})|_Z)$ and $N\in\OCat_{\nu}(\A_{\lambda^+})_{\preceq Z}$.
The right hand side coincides with  $\Hom_{\Ca_{\nu_0}(\A_{\lambda^+})}(M,\varpi_Z(N))$
because the eigenvalue of $h_0$ on $M$ is $c^{p}_{\lambda^+}(h_0)$ and
$\varpi_Z(N)$ is annihilated by $\A_{\lambda^+}^{>0,\nu_0}$. This shows that $\Delta_{\nu_0,Z}$
is the left adjoint of $\varpi_Z$. A similar argument (based on an observation
that $\varpi_Z(N)$ embeds into $N/\A^{<0,\nu_0}_{\lambda^+} N$) shows that
$\nabla_{\nu_0,Z}$ is right adjoint to $\varpi_Z$.
\end{proof}

\subsection{Algebras and bimodules}
Here we will introduce certain algebras and bimodules that will be needed in several proofs below.
Recall that $\chi\in \tilde{\param}_{\Z}$ is ample for $X^\theta$ and that $\ell$
stands for $\{\lambda+z\chi| z\in \C\}$.

Set $\B^+_{\lambda+z\chi}:=\A_{\lambda+z\chi}/\A_{\lambda+z\chi}\A_{\lambda+z\chi}^{>0,\nu_0}$.
This is a $\A_{\lambda+z\chi}$-$\Ca_{\nu_0}(\A_{\lambda+z\chi})$-bimodule. Note that,
for a Zariski generic $z$, $\B^+_{\lambda+z\chi}$ splits into the direct sum
$\bigoplus_Z \B^+_{\lambda+z\chi}|_Z$, where $Z$ runs over the connected components of
$X^{\nu_0(\C^\times)}$ and $\B^+_{\lambda+z\chi}|_Z=\B^+_{\lambda+z\chi}\otimes_{\Ca_{\nu_0}(\A_{\lambda+z\chi})}\Ca_{\nu_0}(\A_{\lambda+z\chi})|_Z$.

Let us write $\B^+_{\ell}$ for $\A_{\ell}/\A_{\ell}\A_{\ell}^{>0,\nu_0}$
so that $\B^+_{\lambda+z\chi}$ is the specialization of $\B^+_\ell$ to $\lambda+z\chi\in \ell$.
We will need a  $\C[\ell]$-algebra $\Ca_{\nu_0}(\A_\ell)|_Z$ and an $\A_\ell$-$\Ca_{\nu_0}(\A_\ell)|_Z$ bimodule $\B^+_{\ell}|_Z$ that specialize to $\Ca_{\nu_0}(\A_{\lambda+z\chi})|_Z,\B^+_{\lambda+z\chi}|_Z$
for a Zariski generic $z$. Recall that we have a natural homomorphism $\Ca_{\nu_0}(\A_\ell)\rightarrow \Gamma(\Ca_{\nu_0}(\A_\ell^\theta))$. The kernel and the cokernel of the homomorphism of the associated graded algebras
are supported at $0\in \C\chi$, this follows from the proof of Proposition \ref{Prop:A0_descr}.

Let $I_Z$ denote the kernel of the composition of the homomorphism $\Ca_{\nu_0}(\A_\ell)\rightarrow \Gamma(\Ca_{\nu_0}(\A_\ell^\theta))$ with the projection $\Gamma(\Ca_{\nu_0}(\A^\theta_\ell))\twoheadrightarrow \Gamma(\Ca_{\nu_0}(\A^\theta_\ell)|_Z)$.
We set $\Ca_{\nu_0}(\A_\ell)|_Z=\Ca_{\nu_0}(\A_\ell)/I_Z$ and $\B^+_{\ell}|_Z=\B^+_{\ell}/\B^+_\ell I_Z$.

\begin{Lem}\label{Lem:Z_spec}
For a Zariski generic $z$, the specializations of $\Ca_{\nu_0}(\A_\ell)|_Z,\B^+_\ell|_Z$ to
$\lambda+z\chi$ coincide with $\Ca_{\nu_0}(\A_{\lambda+z\chi})|_Z, \B^+_{\lambda+z\chi}|_Z$.
\end{Lem}
\begin{proof}
It is enough to check the claim about $\Ca_{\nu_0}(\A_\ell)|_Z$. Lemma \ref{Lem:gen_freeness} applied
to $\mathfrak{A}_\ell=\Ca_{\nu_0}(\A_\ell)$ shows that $I_Z,\Ca_{\nu_0}(\A_\ell)/I_Z$
as well as the cokernel of $\Ca_{\nu_0}(\A_\ell)\rightarrow \Gamma(\Ca_{\nu_0}(\A_\ell^\theta))$
are generically free over $\C[\ell]$. From here we deduce that for a Zariski
generic $z$, the specialization of $I_Z$ to $\lambda+z\chi$ coincides with the kernel of
$\Ca_{\nu_0}(\A_{\lambda+z\chi})\twoheadrightarrow \Gamma(\Ca_{\nu_0}(\A^\theta_{\lambda+z\chi}))$.
Now the claim of the present lemma is a consequence of Proposition \ref{Prop:A0_descr}.
\end{proof}

The algebra $\Ca_{\nu_0}(\A_\ell)|_Z$ is filtered, the filtration is induced from $\Ca_{\nu_0}(\A_\ell)$.
The bimodule $\B^+_{\ell}|_Z$ is also filtered with the filtration induced from $\B^+_\ell$.
We will  need some information about the associated graded algebra $\gr \Ca_{\nu_0}(\A_\ell)|_Z$
and the associated graded bimodule $\gr\B^+_\ell|_Z$. The facts we need about the former
are gathered in the following lemma.

\begin{Lem}\label{Lem:CZ_fibers}
The following is true.
\begin{enumerate}
\item
The fiber
of $\gr \Ca_{\nu_0}(\A_\ell)|_Z$ over $\chi$ is $\C[Z_\chi]$.
\item The scheme $\underline{Z}_{\C\chi}:=\operatorname{Spec}(\gr \Ca_{\nu_0}(\A_\ell)|_Z)$ is equidimensional over $\C\chi$.
\item The fiber $\underline{Z}$ of $\underline{Z}_{\C\chi}$ over $0$ has finitely many symplectic
leaves.
\end{enumerate}
\end{Lem}
\begin{proof}
We have a natural homomorphism
\begin{equation}\label{eq:filt_homom_Ca_Z}\Ca_{\nu_0}(\A_\ell)|_Z\rightarrow \Gamma(\Ca_{\nu_0}(\A_\ell^\theta)|_Z)\end{equation}
that is an isomorphism generically, this follows  Lemma \ref{Lem:Z_spec}. The kernel of its associated graded is
supported at $0\in \C\chi$. Moreover, the induced homomorphism of fibers
$\Ca_{\nu_0}(\A_{\lambda})|_Z\rightarrow \Gamma(\Ca_{\nu_0}(\A_\lambda^\theta)|_Z)$ is a filtration
preserving isomorphism for a Zariski generic $\lambda\in \ell$. It follows that the cokernel of
the associated graded of  (\ref{eq:filt_homom_Ca_Z}) is also supported at $0$. (1) follows.

Let us prove (2). We have $\gr \Gamma(\Ca_{\nu_0}(\A_\ell^\theta)|_Z)=\C[Z_{\C\chi}]$.
This algebra is finite over $\gr\Ca_{\nu_0}(\A_\ell)|_Z$. It follows that if we equip
$\Ca_{\nu_0}(\A_\ell)|_Z$ with the filtration restricted from $\Gamma(\Ca_{\nu_0}(\A^\theta_\ell)|_Z)$,
then the reduced scheme of the associated graded does not change. Therefore  the dimension
of the zero fiber of $\operatorname{Spec}(\gr \Ca_{\nu_0}(\A_\ell)|_Z)\rightarrow \C\chi$
equals $\dim Z$. (2) follows.

Let us prove (3). We have a fiberwise resolution of singularities morphism $Z_{\C\chi}\rightarrow \underline{Z}_{\C\chi}$
that is an isomorphism over $\C^\times\chi$. Since the morphism $Z\rightarrow \underline{Z}$ is a
symplectic resolution of singularities, by Remark \ref{Rem:sympl_leaves_finite},
we conclude that  $\underline{Z}$ has finitely many symplectic leaves.
\end{proof}


Let us proceed to studying the $\C[X_{\C\chi}]\otimes\gr \Ca_{\nu_0}(\A_\ell)|_Z$-bimodule
$\gr\B^+_{\ell}|_Z$.

\begin{Lem}\label{Lem:BZ_assoc_graded}
The following is true:
\begin{enumerate}
\item The support of $\gr\B^+_{\ell}|_Z$ in $X_{0,\C\chi}\times \underline{Z}_{\C\chi}$ lies in
$$\hat{Y}_{\C\chi}:=
\{(x,z)\in X_{0,\C\chi}\times \underline{Z}_{\C\chi}| \lim_{t\rightarrow 0}\nu_0(t)x=z\}.$$
\item The fiber of $\hat{Y}_{\C\chi}$  over $\chi$ coincides with the $Y_{\chi,Z}:=\{(x,z)\in X_\chi\times Z_\chi|
\lim_{t\rightarrow 0}\nu_0(t)x=z\}$.
\item The fiber of $\hat{Y}_{\C\chi}$ over $0$ is an isotropic subvariety in $X_0\times \underline{Z}$.
\item The fiber of $\gr\B^+_{\ell}|_Z$ over $\chi$ is $\C[Y_{\chi,Z}]$.
\end{enumerate}
\end{Lem}
\begin{proof}
(1) and (2) are straightforward.

Let us prove (3). Recall, \cite[Section 5.1]{B_ineq}, that a subvariety of
$X_0\times \underline{Z}$ is isotropic if and only if its pre-image in $X\times Z$
is.  The preimage of the zero fiber of $\hat{Y}_{\C\chi}$
in $X\times Z$ consists of all points $(\tilde{x},\tilde{x}')$
such that $\lim_{t\rightarrow 0}\nu_0(t)\tilde{x}$ exists and $\rho(\lim_{t\rightarrow 0}\nu_0(t)\tilde{x})=
\rho(\tilde{x}')$. The claim that this locus is isotropic easily reduces to checking that  $
(\rho\times \rho)^{-1}(X_{diag})\cap X^{\nu_0(\C^\times)}\times Z$ is isotropic. But  $
(\rho\times \rho)^{-1}(X_{diag})\subset X\times X$ is isotropic while $X^{\nu_0(\C^\times)}\times Z$
 symplectic. These two observations  imply our claim.
We conclude that the zero fiber of $\hat{Y}_{\C\chi}$ is isotropic.

Let us prove (4). Let us check that the kernel of the natural epimorphism
$$\C[X_{\C\chi}]\C[X_{\C\chi}]\C[X_{\C\chi}]^{>0,\nu_0}\twoheadrightarrow
\gr\B^+_\ell$$  is supported at $0\in \C\chi$. Consider the $\hbar$-adic completion
$\A_{\ell,\hbar}$ of $R_\hbar(\A)$. Let $\A_{\ell,\hbar}^{reg}$ denote the (completed) localization
of $\A_{\ell,\hbar}$ to $\C^\times \chi$. Then $\A_{\ell,\hbar}^{reg}/
\A_{\ell,\hbar}^{reg}(\A_{\ell,\hbar}^{>0,\nu'})^{reg}$ coincides with the  (micro)localization of $\A_{\ell,\hbar}/\A_{\ell,\hbar}(\A_{\ell,\hbar})^{>0,\nu'}$.
On the other hand, over $\C^\times \chi$, the ideal $\C[X_{\C\chi}]\C[X_{\C\chi}]^{>0,\nu'}$
is a locally complete intersection (given by elements of positive $\nu'$-weight), compare to the proof
of (2) in Proposition \ref{Prop:A0}. As in the proof of (1) of Proposition \ref{Prop:A0}, this implies that
$\A_{\ell,\hbar}^{reg}/\A_{\ell,\hbar}^{reg}(\A_{\ell,\hbar}^{>0,\nu'})^{reg}$ is flat over $\C[[\hbar]]$.
So the $\hbar$-torsion in  $\A_{\ell,\hbar}/\A_{\ell,\hbar}(\A_{\ell,\hbar})_{>0,\nu'}$ is supported
at $0$. This implies the claim in the beginning of the paragraph.

Part (1) of Lemma \ref{Lem:CZ_fibers}
implies that the  fiber over $\chi$ of $\gr I_Z$ is the kernel of $\C[X_\chi^{\nu_0(\C^\times)}]
\rightarrow \C[Z_\chi]$. So $(\gr\B^+_\ell)_\chi/ (\gr\B^+_\ell)_\chi (\gr I_Z)_\chi=\C[Y_{\chi,Z}]$
and we get an epimorphism $\C[Y_{\chi,Z}]\rightarrow (\gr \B^+_{\ell}|_Z)_\chi$
of $\C[X_\chi]\otimes \C[Z_\chi]$-modules. The Gabber involutivity theorem implies that
the support of $(\gr \B^+_{\ell}|_Z)_\chi$ in $X_\chi\times Z_\chi$ is coisotropic
(if this bimodule is nonzero). Since $Y_{\chi,Z}$ is lagrangian, it follows that either
$(\gr \B^+_{\ell}|_Z)_\chi=\C[Y_{\chi,Z}]$ or $(\gr \B^+_{\ell}|_Z)_\chi=0$.
The latter is not the case because
$$[(\gr \B^+_{\ell}|_Z)_\chi]^{\nu_0(\C^\times)}=\gr[\B^+_{\ell}|_Z^{\nu_0(\C^\times)}]_\chi=
[\gr \Ca_{\nu_0}(\A_\ell)|_Z]_\chi=\C[Z_\chi].$$
\end{proof}

To finish this section, let us point out that the conclusions of this section also hold
for the following bimodules
$$ \B^-_{\ell}:=\A_{\ell}/\A_{\ell}\A_{\ell}^{<0,\nu_0}, \quad \B^{+,o}_\ell:=\A_{\ell}/\A_{\ell}^{>0,\nu_0}\A_\ell,
\quad \B^{-,o}_{\ell}:=\A_{\ell}/\A_{\ell}^{<0,\nu_0}\A_\ell.
$$
\subsection{Standardization functor}\label{SS_stand_exact}
According to Proposition \ref{Prop:assoc_graded_O}, under the identification of
$\OCat_\nu(\A_\lambda^\theta)$ with $\OCat_{\nu}(\A_{\lambda^+})$ and
$\gr_{\nu_0} \OCat_\nu(\A_\lambda^\theta)$ with $\OCat_\nu(\Ca_{\nu_0}(\A_{\lambda^+}))$, the standardization
functor $\gr_{\nu_0} \OCat_\nu(\A_\lambda^\theta)\rightarrow
\OCat_\nu(\A_\lambda^\theta)$ becomes $\Delta_{\nu_0}$. Recall that we write
$\lambda^+$ for $\lambda+n\chi$ with $n\gg 0$ and $\chi\in \tilde{\param}_{\Z}$ ample.
The goal of this section is to prove the  following claim.

\begin{Prop}\label{Prop:Cat_O_stand_exact}
The functor $\Delta_{\nu_0}: \OCat_{\nu}(\Ca_{\nu_0}(\A_{\lambda^+}))
\rightarrow \OCat_{\nu}(\A_{\lambda^+})$ is exact.
\end{Prop}

Since $\Delta_{\nu_0}(\bullet)=\B^+_{\lambda^+}
\otimes_{\Ca_{\nu_0}(\A_{\lambda^+})}\bullet$, this proposition is a direct consequence of the following claim.

\begin{Prop}\label{Prop:projectivity}
For a Zariski generic $z$, the right $\Ca_{\nu_0}(\A_{\lambda+z\chi})$-module $\B_{\lambda+z\chi}$
is projective.
\end{Prop}

\begin{proof}
The proof is in several steps.

{\it Step 1}. The claim will follow if we show
that, for a Zariski generic $z$, $\B^+_{\lambda+z\chi}|_Z$ is a projective
$\Ca_{\nu_0}(\A_{\lambda+z\chi})|_Z$-module. This will follows if we check
that
\begin{equation}\label{eq:Ext_vanish}
\operatorname{Ext}^i_{\Ca_{\nu_0}(\A_{\lambda+z\chi})|_Z^{opp}}
(\B^+_{\lambda+z\chi}|_Z, \Ca_{\nu_0}(\A_{\lambda+z\chi})|_Z)=0, \text{ for }i>0.
\end{equation}
Indeed, (\ref{eq:Ext_vanish}) implies that $\Ext^i_{\Ca_{\nu_0}(\A_{\lambda+z\chi})|_Z^{opp}}
(\B^+_{\lambda+z\chi}|_Z, M)=0$ when $i>0$, for any right  $\Ca_{\nu_0}(\A_{\lambda+z\chi})|_Z$-module $M$.

{\it Step 2}.
By Lemma \ref{Lem:gen_freeness}, $\B^+_{\ell}|_Z$ is generically free over $\C[\ell]$. We deduce that, for a Zariski
generic $z$, the left hand side of (\ref{eq:Ext_vanish}) is the specialization to $\lambda+z\chi$ of
\begin{equation}\label{eq:Ext_over_ell}
\operatorname{Ext}^i_{\Ca_{\nu_0}(\A_{\ell})|_Z^{opp}}
(\B^+_{\ell}|_Z, \Ca_{\nu_0}(\A_{\ell})|_Z).
\end{equation}

So (\ref{eq:Ext_vanish}) for a fixed $i$ will follow if we check that the module in (\ref{eq:Ext_over_ell})
is torsion over $\C[\ell]$. The claim for all $i$ is then proved by using the fact that, for a Zariski
generic $z$, the algebra $\Ca_{\nu_0}(\A_{\lambda+z\chi})|_Z=\Gamma(\Ca_{\nu_0}(\A^\theta_{\lambda+z\chi})|_Z)$
has  homological dimension not exceeding $\dim Z$, see (3) of Lemma \ref{Lem:gen_simplicity}.

{\it Step 3}. So we are proving that (\ref{eq:Ext_over_ell}) is torsion over $\C[\ell]$. To this end  we
will need the category $\mathcal{C}$  of all finitely  generated $\Ca_{\nu_0}(\A_{\ell})|_Z$-$\A_\ell$-bimodules
whose associated variety lies in the  subvariety $\hat{Y}_{\C\chi}$ from Lemma \ref{Lem:BZ_assoc_graded}.

Thanks to (3) of Lemma \ref{Lem:BZ_assoc_graded},  for $M\in \Cat$, the specialization $M_{\lambda+z\chi}$ is a holonomic
$\A_{\lambda+z\chi}^{opp}\otimes \Ca_{\nu_0}(\A_{\lambda+z\chi})|_Z$-module,
see Section \ref{SS_holon}. We conclude that, for a Weil generic $z$, the specialization $M_{\lambda+z\chi}$
has GK dimension equal to $\frac{1}{2}(\dim X+\dim Z)$ (provided $M_{\lambda+z\chi}\neq 0$), this follows from
 Corollary \ref{Cor:full_GK_O} (applied to the algebra $\A_{\lambda+z\chi}^{opp}\otimes \Ca_{\nu_0}(\A_{\lambda+z\chi})|_Z$ that is the algebra of global sections of a conical
symplectic resolution).

{\it Step 4}. Now let $M$ be an object in $\Cat$. Suppose that $M$ has a good filtration such that
$\gr M$ is torsion over $\C[\C\chi]$. We claim  that $M$ is torsion over $\C[\ell]$.
Lemma \ref{Lem:gen_freeness} implies that $M$ is generically free over $\C[\ell]$.
The associated graded bimodule $\gr M$ is supported on the zero fiber of
$\widehat{Y}_{\C\chi}\rightarrow \C\chi$. The dimension
of the latter does not exceed $\frac{1}{2}(\dim X+\dim Z)$, see (3) of Lemma \ref{Lem:BZ_assoc_graded}.
We conclude that the GK dimension of
$M$ does not exceed  $\frac{1}{2}(\dim X+\dim Z)$. On the other hand, since $M$
is generically free over $\C[\ell]$, we have $M_{\lambda+z\chi}\neq 0$ for a Weil generic $z$.
So the GK dimension of a Weil generic fiber of $M$ coincides with that of $M$. It follows that we have quotients
of $M$ with arbitrary large GK multiplicity. But the GK multiplicity
of $M$ is not less than that of any quotient of $M$. We arrive at a contradiction.

{\it Step 5}. It remains to show that (\ref{eq:Ext_over_ell}) lies in $\Cat$ and
has a good filtration like in Step 4. Recall that $\gr \B^+_{\ell}|_{Z}$ is a
finitely generated $\C[X_{\C\chi}]\otimes \gr \Ca_{\nu_0}(\A_{\ell})|_Z$-module supported on $\widehat{Y}_{\C\chi}$
by (1) of Lemma \ref{Lem:BZ_assoc_graded}. Also note that (\ref{eq:Ext_over_ell}) has a filtration with
\begin{equation}\label{eq:Ext_incl}\gr\operatorname{Ext}^i_{\Ca_{\nu_0}(\A_{\ell})|_Z^{opp}}
(\B^+_{\ell}|_Z, \Ca_{\nu_0}(\A_{\ell})|_Z)\subset \operatorname{Ext}^i_{\gr\Ca_{\nu_0}(\A_{\ell})|_Z}
(\gr\B^+_{\ell}|_Z, \gr\Ca_{\nu_0}(\A_{\ell})|_Z).\end{equation} This implies that
(\ref{eq:Ext_over_ell}) is in $\mathcal{C}$. Let us show that the right hand side of
(\ref{eq:Ext_incl}) is supported at $0\in \C\chi$ (when $i>0$). Indeed, the specialization of  $\gr\Ca_{\nu_0}(\A_{\ell})|_Z$ to $\chi$ is $\C[Z_\chi]$ (see (1) of Lemma \ref{Lem:CZ_fibers}),
while the specialization of $\gr\B^+_{\ell}|_Z$ is $\C[Y_{\chi,Z}]$ (by (4) of Lemma \ref{Lem:BZ_assoc_graded}). It
follows that the specialization of the right hand side of (\ref{eq:Ext_incl})
to $\chi$ coincides with $$\Ext^i_{\C[Z_{\chi}]}(\C[Y_{\chi,Z}],\C[Z_\chi]).$$
But $\C[Y_{\chi,Z}]$ is a projective $\C[Z_\chi]$-module and so the fiber at $\chi$
vanishes.

This completes the proof.
\end{proof}

\subsection{Filtration on projectives}
We finish the proof of Theorem \ref{Thm:cat_O_ss} using Lemma \ref{Lem:ss_order}.
In order to do this, note that the functor
$$\nabla_{\nu_0}(\bullet)=\Hom_{\Ca_{\nu_0}(\A_{\lambda+z\chi})}(\B^{-,o}_{\lambda+z\chi}, \bullet):
\OCat_{\nu}(\Ca_{\nu_0}(\A_{\lambda+z\chi}))\rightarrow \OCat_\nu(\A_{\lambda+z\chi})$$
is exact for a Zariski generic $z\in \C$, where $\B^{-,o}_{\lambda+z\chi}=\A_{\lambda+z\chi}/
\A^{<0,\nu_0}_{\lambda+z\chi}\A_{\lambda+z\chi}$. Indeed, completely analogously to
Proposition \ref{Prop:projectivity}, $\B^{-,o}_{\lambda+z\chi}$ is a projective
$\Ca_{\nu_0}(\A_{\lambda+z\chi})$-module for a Zariski generic $z$.

The category $\OCat_{\nu}(\A_{\lambda^+})$ is highest weight with respect to the order $\leqslant_{\nu}$.
It follows that it satisfies the conditions of Lemma \ref{Lem:w_st_stratif}. Together with the exactness
of $\nabla_{\nu_0}$, this implies that $\OCat_\nu(\A_{\lambda^+})$ is standardly stratified with
respect to the pre-order $\preceq_{\nu_0}$. This finishes the proof of Theorem \ref{Thm:cat_O_ss}.

Let us record a result that follows from the claim that the right $\Ca_{\nu_0}(\A_{\lambda^+})$-module
$\B^+_{\lambda^+}$ and the left $\Ca_{\nu_0}(\A_{\lambda^+})$-module $\B^{-,o}_{\lambda^+}$
are projective.

\begin{Lem}\label{Lem:der_st_cost}
For $M\in D^b(\OCat_\nu(\Ca_{\nu_0}(\A_{\lambda^+})))$, we have
$$\Delta_{\nu_0}(M)=\B^+_{\lambda^+}\otimes^L_{\Ca_{\nu_0}(\A_{\lambda^+})}M,\quad \nabla_{\nu_0}(M)=
R\operatorname{Hom}_{\Ca_{\nu_0}(\A_{\lambda^+})}(\B^{-,o}_{\lambda^+},M).$$
\end{Lem}

\section{Cross-walling functors}\label{S_CW}
\subsection{Definition and basic properties}
Recall that we have a full embedding $\iota_\nu: D^b(\OCat_\nu(\A_\lambda^\theta))\hookrightarrow
D^b(\operatorname{Coh}(\A_\lambda^\theta))$ (Lemma \ref{Lem:local} and (3) of Proposition
\ref{Prop:hw_alg}). The image clearly lies in the full subcategory
$D^b_{hol}(\operatorname{Coh}(\A_\lambda^\theta))$ of all complexes with holonomic homology
(recall  that a coherent $\A_\lambda^\theta$-module is called holonomic
if its support is lagrangian, note that by the Gabber involutivity theorem, the support is
always coisotropic).
It was checked in \cite[Section 8.2]{BLPW} that the functor $\iota_\nu: D^b(\OCat_\nu(\A_\lambda^\theta))\hookrightarrow
D^b_{hol}(\operatorname{Coh}(\A_\lambda^\theta))$ admits both left adjoint $\iota^!_\nu$
and right adjoint $\iota^*_\nu$.

For two generic one-parameter subgroups $\nu,\nu':\C^\times\rightarrow T$  we define the {\it cross-walling}
functor (shuffling functor from \cite[Section 8.2]{BLPW}) by setting
$$\CW_{\nu\rightarrow \nu'}:= \iota^!_{\nu'}\circ \iota_\nu:D^b(\OCat_\nu(\A_\lambda^\theta))\rightarrow D^b(\OCat_{\nu'}(\A_\lambda^\theta)).$$
So the functor $\CW_{\nu\rightarrow \nu'}$ is uniquely characterized by a bi-functorial isomorphism
\begin{equation}\label{eq:CW_char}
\operatorname{Hom}_{D^b(\OCat_{\nu'}(\A_\lambda^\theta))}(\CW_{\nu\rightarrow \nu'}M,N)=\operatorname{Hom}_{D^b(\Coh(\A_\lambda^\theta))}(M,N).
\end{equation}

Note that the functor $\CW_{\nu\rightarrow \nu'}$ admits a right adjoint functor
$\CW^*_{\nu\rightarrow \nu'}=\iota_\nu\circ \iota_{\nu'}^*$. Also note that there is
a natural functor morphism
\begin{equation}\label{eq:CW_fun_morph}
\CW_{\nu\rightarrow \nu''}\rightarrow \CW_{\nu'\rightarrow \nu''}\circ \CW_{\nu\rightarrow \nu'}
\end{equation}
induced by the adjunction morphism $\operatorname{id}\rightarrow \iota_{\nu'}\circ \iota_{\nu'}^!$.

When $\Loc_\lambda,\Gamma_\lambda$ are abelian equivalences, we can carry the functor
$\CW_{\nu\rightarrow \nu'}$ to $D^b(\OCat_\nu(\A_\lambda))\rightarrow D^b(\OCat_{\nu'}(\A_\lambda))$.
Obviously, this functor can be realized as $\iota_{\nu'}^!\circ \iota_\nu$, where
$\iota_\nu$ now stands for the inclusion $D^b(\OCat_\nu(\A_\lambda))\hookrightarrow D^b_{hol}(\A_\lambda\operatorname{-mod})$.

Let us finish this section by discussing the case of products, where we have $X=X^1\times X^2$.
Recall that we have an identification $\OCat_\nu(\A_\lambda^\theta)\cong \OCat_\nu(\A^{1\theta^1}_{\lambda^1})\boxtimes
\OCat_\nu(\A^{2\theta^2}_{\lambda^2})$, Lemma \ref{Lem:product}. The following claim follows from {\it loc.cit.}
and (\ref{eq:CW_char}).

\begin{Lem}\label{Lem:CW_prod}
Under the identification of the previous paragraph, we have $\CW_{\nu\rightarrow \nu'}=\CW^1_{\nu\rightarrow \nu'}\boxtimes \CW^2_{\nu\rightarrow \nu'}$.
\end{Lem}
\subsection{Main results}\label{SS_main_CW}
Let us start with some notation.  Pick two generic one-parameter subgroups
$\nu,\nu': \C^\times \rightarrow T$. Define $\delta(\nu,\nu')\in \Z_{\geqslant 0}$ as the number of $T$-weights  in $\bigoplus_{p\in X^T}T_pX$ that are positive on $\nu$ and negative on $\nu'$.  Note that
$\delta(\nu,\nu')+\delta(\nu',\nu'')\geqslant \delta(\nu,\nu'')$ and
$\delta(\nu,\nu')=\delta(\nu',\nu)$. We will say that a sequence $\nu_1,\ldots,\nu_k$
of generic one-parameter subgroups of $T$ is {\it reduced} is $\delta(\nu_1,\nu_k)=
\sum_{i=1}^{k-1}\delta(\nu_i,\nu_{i+1})$. We say that
a reduced sequence $\nu_1,\ldots,\nu_k$ is {\it maximal reduced} if $\nu_i,\nu_{i+1}$ lie in chambers
separated by a single wall.

\begin{Lem}\label{Lem:max_reduced}
Every reduced sequence can be completed to a maximal reduced one.
\end{Lem}
\begin{proof}
We need to check that generic $\nu,\nu'$ are included into a maximal reduced sequence. Pick sufficiently
general $\nu,\nu'$ in their chambers so that the interval joining $\nu,\nu'$ does not intersect
pairwise intersections of walls. Let $\nu=\nu_0,\nu_1,\ldots,\nu'=\nu_k$ be one-parameter subgroups
from the chambers intersected by that interval taken in order. Then $\nu_0,\ldots,\nu_k$
is easily seen to be a maximal reduced sequence.
\end{proof}

Now we are ready to state the main results.

\begin{Thm}\label{Thm:braid_equiv}
The functor $\CW_{\nu\rightarrow\nu'}$ is an equivalence for any $\nu,\nu'$. Moreover,
for any reduced sequence $\nu_1,\ldots,\nu_k$, the natural functor morphism
$$\CW_{\nu_1\rightarrow \nu_k}\rightarrow \CW_{\nu_{k-1}\rightarrow \nu_k}\circ \CW_{\nu_{k-2}\rightarrow \nu_{k-1}}\circ\ldots \circ \CW_{\nu_1\rightarrow \nu_2}$$
is an isomorphism.
\end{Thm}

This theorem proves several conjectures of Braden, Licata, Proudfoot and Webster
(stated in a previous version of \cite{BLPW}).
The claim that the functors $\CW_{\nu\rightarrow \nu'}$ yield an action of the
Deligne groupoid of the hyperplane arrangement given by the non-generic one-parameter subgroups
is established as in \cite[Section 6.4]{BPW}.

Our next main result (that plays a crucial role in the proof of Theorem \ref{Thm:braid_equiv})
concerns an interplay between the equivalences $\CW_{\nu\rightarrow \nu'}$
and the standardly stratified structures on the categories involved. In order to
state the result we need some more notation.

Pick $\nu_0:\C^\times\rightarrow T$.
Recall that for a generic $\nu$ and some $\nu_0$ we will write $\nu\rightsquigarrow \nu_0$ if $\nu_0$ lies in
the closure of the chamber containing $\nu$.
We write $\underline{\CW}_{\nu\rightarrow \nu'}$ for the cross-walling functor between
categories $\mathcal{O}$ for $\Ca_{\nu_0}(\A_\lambda^\theta)$.

\begin{Prop}\label{Prop_CW_prop}
Let $\nu$ be generic with $\nu\rightsquigarrow \nu_0$.  The following is true.
\begin{enumerate}
\item If $\nu'$ is generic with $\nu'\rightsquigarrow \nu_0$, then  $\nabla_{\nu_0}\circ \underline{\CW}^*_{\nu\rightarrow \nu'}\cong \CW^*_{\nu\rightarrow \nu'}\circ \nabla_{\nu_0}$
    and $\Delta_{\nu_0}\circ \underline{\CW}_{\nu\rightarrow \nu'}\cong
    \CW_{\nu\rightarrow \nu'}\circ \Delta_{\nu_0}$.
\item The functor $\CW^*_{\nu\rightarrow -\nu}[\frac{1}{2}\dim X]$ is a Ringel duality.
\end{enumerate}
\end{Prop}

We will recall a definition of the Ringel duality in the next section.

In particular, (2) implies the shuffling part of \cite[Conjecture 8.24]{BLPW} (the twisting part
will be proved in the next section).

\subsection{Ringel duality, homological duality and long wall-crossing}
Let $\Cat_1,\Cat_2$ be two highest weight categories. Consider the full subcategories
$\Cat_1^\Delta\subset \Cat_1,\Cat_2^\nabla\subset \Cat_2$ of standardly filtered
and of costandardly filtered objects. We say that $\Cat_2$ is Ringel dual to
$\Cat_1$ if there is an equivalence $\Cat_1^\Delta\xrightarrow{\sim}\Cat_2^\nabla$
of exact categories (and we write $\Cat_1^\vee$ for $\Cat_2$).
This equivalence is known to extend to a derived
equivalence $\mathcal{R}:D^b(\Cat_1)\xrightarrow{\sim} D^b(\Cat_2)$
to be called the Ringel duality functor. The functor $\mathcal{R}$ maps standards
to costandards, projectives to tiltings, tiltings to injectives.
So we can identify highest weight posets of $\Cat_1,\Cat_2$
as sets (they will have the opposite orders).
In this section we will describe the category $\OCat_\nu(\A_\lambda^\theta)^\vee$.



Consider the functor $D:=\operatorname{RHom}_{\A_{\lambda^+}}(\bullet, \A_{\lambda^+})$
defines an equivalence $D^b(\A_{\lambda^+}\operatorname{-mod})\rightarrow D^b(\A_{\lambda^+}^{opp}\operatorname{-mod})^{opp}$.
Note that if $\Supp H^i(D M)\subset \Supp M$ for all $i$. Also note that $H^i(DM)$ is weakly $T$-equivariant
provided $M$ is. Using Lemma \ref{Lem:O_equiv},
we see that $D$ restricts to an equivalence $D^b_{\OCat_\nu}(\A_{\lambda^+}\operatorname{-mod})\xrightarrow{\sim}
D^b_{\OCat_\nu}(\A_{\lambda^+}^{opp}\operatorname{-mod})^{opp}$. So we get
a derived equivalence $D:D^b(\OCat_\nu(\A_{\lambda^+}))\xrightarrow{\sim}
D^b(\OCat_\nu(\A_{\lambda^+}^{opp}))^{opp}$.

\begin{Prop}\label{Prop:D_Ringel}
There is a labeling preserving highest weight equivalence $\OCat_\nu(\A_{\lambda^+})^\vee\xrightarrow{\sim}
\OCat_\nu(\A_{\lambda^+}^{opp})^{opp}$ that intertwines $\mathcal{R}$ and $D$.
\end{Prop}

We start with a technical lemma. Pick $p\in X^T$. It uniquely deforms to a $T$-fixed point in any $X_\chi$
that we again denote by $p$. Let $Y_{\chi,p}$ denote the $\nu$-contracting locus for $p$ in $X_\chi$.

\begin{Lem}\label{Lem:dual_hom_vanish}
The object $D(\Delta_{\lambda^+}(p))$ is concentrated in homological degree $\frac{1}{2}\dim X$
and, moreover, its characteristic cycle (an element of the vector space with basis formed by the irreducible
components of the contracting variety $Y$) coincides with the class of $\overline{\C^\times Y_{\chi,p}}\cap X$.
\end{Lem}
\begin{proof}
Let us prove the first claim. Set $\ell:=\{\lambda+z\chi| z\in \C\}$.
What we need to prove is that $\operatorname{Ext}^i(\Delta_{\lambda'},\A_{\lambda'})=0$
provided $i\neq \frac{1}{2}\dim X$ for a Zariski generic $\lambda'\in \ell$.
Our claim will follow follow if we show that the support of $\operatorname{Ext}^i(\Delta_\ell, \A_\ell)$
in $\C \chi$ is finite and that the $\C[\ell]$-module $\operatorname{Ext}^i(\Delta_\ell, \A_\ell)$ is generically free.
Here we write $\Delta_\ell=\A_\ell/\A_\ell\A_\ell^{>0}$.

The right $\A_\ell$-modules
$\operatorname{Ext}^i(\Delta_\ell, \A_\ell)$  are filtered with
  $\gr\operatorname{Ext}^i(\Delta_\ell, \A_\ell)\subset
  \operatorname{Ext}^i(\gr\Delta_\ell, \C[X_{\C\chi}])$.
The claim about generic freeness follows from Lemma \ref{Lem:gen_freeness}. Also to prove
that claim in the previous paragraph that the support is  finite
it is enough to prove that  $\operatorname{Ext}^i(\gr\Delta_\ell, \C[X_{\C\chi}])$
is supported at 0 as in Step 4 of the proof of Proposition \ref{Prop:projectivity}.

Set $\Delta_{\C\chi}:=\C[X_{\C\chi}]/\C[X_{\C\chi}]\C[X]_{\C\chi}^{>0}$.
We have $\Delta_{\C\chi}\twoheadrightarrow \gr\Delta_\ell$. Moreover, the kernel is
supported at 0, see the proof of (4) of Lemma \ref{Lem:BZ_assoc_graded}. So it is enough to
show that the support of $\operatorname{Ext}^i(\Delta_{\C\chi}, \C[X_{\C\chi}])$
is  at 0 when $i\neq \frac{1}{2}\dim X$. This follows  from the observation  that, generically over $\C\chi$, the ideal  $\C[X_{\C\chi}]\C[X_{\C\chi}]^{>0}$ is a locally complete intersection in a smooth variety.

The argument above also implies that the associated graded of $D(\Delta_{\lambda^+}(p))$ coincides with
that of $\operatorname{Ext}^{\frac{1}{2}\dim X}(\Delta_{\C\chi}(p)|_\chi,\C[X_{\chi}])$.
The latter is just the class of the contracting component
$Y_{\chi,p}$ (defined as the sum of components of $X\cap \overline{\C^\times Y_{\chi,p}}$
with obvious multiplicities).
%
\end{proof}

\begin{proof}[Proof of Proposition \ref{Prop:D_Ringel}]
We write $\Delta_{\lambda}(p)^\vee$ for $D(\Delta_{\lambda}(p))$, thanks to Lemma \ref{Lem:dual_hom_vanish}, this is an object in $\OCat(\A_{-\lambda})$ (and not just a complex in its derived category).  We have $\operatorname{End}(\Delta_{\lambda}(p)^\vee)=\C$
and $\Ext^i(\Delta_{\lambda}(p)^\vee, \Delta_{\hat{{\lambda}}}(p')^\vee)=0$ if $i>0$ or $p\leqslant^\theta p'$.
We remark that the orders $\leqslant^\theta$ and $\leqslant^{-\theta}$ can be refined
to opposite partial orders (by the values of the real moment
maps for the actions of $\mathbb{S}^1\subset \nu(\C^\times)$),
compare with \cite[5.4]{Gordon}. So it only
remains to prove that the characteristic cycle of $\Delta_{\lambda}(p)^\vee$ consists
of the contracting components $Y_{p'}$ with $p'\leqslant^{-\theta}p$.
The characteristic cycle of $\Delta_{\lambda}(p)^\vee$ coincides with $\overline{\C^\times Y_{\theta,p}}\cap X$,
by Lemma \ref{Lem:dual_hom_vanish}. But the characteristic cycle of $\Delta_{-\lambda}(p)$ is the same. Our claim follows.
\end{proof}

Now we can establish a connection between the {\it long wall-crossing} functor
from \cite[Section 4.1]{BL} and the Ringel duality.  Pick $N\gg 0$ and set $\lambda^-:=\lambda-N\theta$
so that the analogs of (1) and (2) hold for $(\lambda^-,\theta)$. We can consider
the functor $\WC_{\lambda^-\leftarrow \lambda}: D^b(\A_\lambda\operatorname{-mod})
\rightarrow D^b(\A_{\lambda^-}\operatorname{-mod})$ given by $\A_{\lambda,-N\chi}^{(-\theta)}\otimes^L_{\A_\lambda}\bullet$. This is the long
twisting functor from \cite[Section 8.1]{BLPW}, it is an equivalence, see \cite[6.4]{BPW}.
Since $\A_{\lambda,-N\chi}^{(-\theta)}$ is a HC bimodule, the functor
$\WC_{\lambda\rightarrow \lambda^-}$ restricts to an equivalence
$D^b(\OCat_\nu(\A_\lambda))\rightarrow D^b(\OCat_\nu(\A_{\lambda^-}))$.

The following proposition implies the twisting part of
\cite[Conjecture 8.24]{BLPW}.

\begin{Prop}\label{Prop:WC_Ringel}
There is a labelling preserving highest weight equivalence
$\OCat_\nu(\A_{\lambda^-})\xrightarrow{\sim} \OCat_{\nu}(\A_\lambda)^\vee$
that intertwines $\WC_{\lambda\rightarrow \lambda^-}$ and $\mathcal{R}$.
\end{Prop}
\begin{proof}
In \cite[Section 4.3]{BL} we have established a $t$-exact equivalence
$D^b_{hol}(\A_{\lambda^-}\operatorname{-mod})\xrightarrow{\sim}
D^b_{hol}(\A_\lambda^{opp}\operatorname{-mod})^{opp}$ that intertwines
$\WC_{\lambda^-\leftarrow \lambda}$ and $D$ (there we were dealing with the case
when $X$ is Nakajima quiver variety but the proof works in the general case as
well). Now we are done by Proposition \ref{Prop:D_Ringel}.
\end{proof}

\subsection{Proof of (1) of Proposition \ref{Prop_CW_prop}}
In this section we prove (1) of Proposition \ref{Prop_CW_prop}. Our proof is based on the
computation of the derived tensor product of certain bimodules.

The bimodules we are interested in are as follows:
$$\mathcal{B}^{-,o}_{\lambda^+}:=\A_{\lambda^+}/\A_{\lambda^+}^{<0,\nu_0}\A_{\lambda^+},
\mathcal{B}^+_{\lambda^+}=\A_{\lambda^+}/\A_{\lambda^+}\A_{\lambda^+}^{>0,\nu_0}.$$
The former is a $\Ca_{\nu_0}(\A_{\lambda^+})$-$\A_{\lambda^+}$-bimodule, and the latter is
a $\A_{\lambda^+}$-$\Ca_{\nu_0}(\A_{\lambda^+})$-bimodule.

\begin{Lem}\label{Lem:der_tens_prod}
We have $\mathcal{B}^{-,o}_{\lambda^+}\otimes^L_{\A_{\lambda^+}}\mathcal{B}^+_{\lambda^+}=\Ca_{\nu_0}(\A_{\lambda^+})$.
\end{Lem}
\begin{proof}
Note that
$$\mathcal{B}^{-,o}_{\lambda^+}\otimes_{\A_{\lambda^+}}\mathcal{B}^+_{\lambda^+}=\A_{\lambda^+}/
(\A_{\lambda^+}^{<0,\nu_0}\A_{\lambda^+}+\A_{\lambda^+}\A_{\lambda^+}^{>0,\nu_0})=
\A_{\lambda^+}/\bigoplus_{i>0}\A^{-i}_{\lambda^+}\A^i_{\lambda^+}=
\Ca_{\nu_0}(\A_{\lambda^+}).$$
So we just need to check that the higher Tor's vanish.

Set $\ell:=\{\lambda+z\chi, z\in \C\}$. We can consider the bimodules $\B^{-,o}_{\ell},\B^+_{\ell}$ and form
the derived tensor product $\B^{-,o}_\ell\otimes^{\A_\ell}\B^+_{\ell}$.  Similarly
to Step 2 of  the proof of Proposition \ref{Prop:projectivity},  we just need to check that each $\operatorname{Tor}_i^{\A_{\ell}}(\B^{-,o}_{\ell}, \B^+_{\ell})$ is torsion over $\C[\ell]$.

The bimodules $\operatorname{Tor}_i^{\gr\A_{\ell}}(\gr\B^{-,o}_{\ell}, \gr\B^+_{\ell})$
are all supported at $0\in \C\chi$. Indeed, the fiber of $\gr\B^{+}_{\ell}$ over $\chi$
is $\C[Y_{\chi}]$, while the fiber of
of $\gr\B^{-,o}_{\ell}$ equals $\C[Y^-_{\chi}]$,
compare to (4) of Lemma \ref{Lem:BZ_assoc_graded}. Here $Y^-_{\chi}:=\{(x,z)\in X_\chi\times X^{\nu_0(\C^\times)}_\chi| \lim_{t\rightarrow\infty}\nu_0(t)x=z\}$ and $Y_\chi$ is defined similarly.
The varieties $Y_{\chi}$ and $Y^-_{\chi}$ intersect transversally,
hence our claim. Note that
\begin{equation}\label{eq:Tor_incl}
\gr \operatorname{Tor}_i^{\A_{\ell}}(\B^{-,o}_{\ell}, \B^+_{\ell})
\subset \operatorname{Tor}_i^{\C[X_{\C\chi}]}(\gr\B^{-,o}_{\ell}, \gr\B^+_{\ell})
\end{equation}
(an inclusion of $\gr \Ca_{\nu_0}(\A_\lambda)$-bimodules).
We conclude that $\gr \operatorname{Tor}_i^{\A_{\ell}}(\B^{-,o}_{\ell}, \B^+_{\ell})$
is supported at $0$. We want to deduce from here that $\operatorname{Tor}_i^{\A_{\ell}}(\B^{-,o}_{\ell},
\B^+_{\ell})$ is torsion over $\C[\ell]$.

 Thanks to Proposition \ref{Prop:supp_gr}, the previous claim will follow if we check that   $\operatorname{Tor}_i^{\A_{\ell}}(\B^{-,o}_{\ell}, \B^+_{\ell})$ is a Harish-Chandra $\Ca_{\nu_0}(\A_\ell)$-bimodule. We have an epimorphism $\C[X_{\C\chi}]^{\nu_0(\C^\times)}
\twoheadrightarrow \gr\Ca_{\nu_0}(\A_{\ell})$. So it is enough to prove that the left and the right
actions of $\C[X_{\C\chi}]^{\C^\times}$ on  $\operatorname{Tor}_i^{\C[X_{\C\chi}]}(\gr\B^{-,o}_{\ell}, \gr\B^+_{\ell})$ coincide. We have $\C[X_{\C\chi}]/\C[X_{\C\chi}]\C[X_{\C\chi}]^{>0,\nu_0}\twoheadrightarrow \gr \B^{-,o}_{\ell}$. So the action of $\C[X_{\C\chi}]^{\C^\times}$ on $\gr \B^{-,o}_{\ell}$ (from the right) is restricted from the $\C[X_{\C\chi}]$-action (from the left). It follows that the left $\C[X_{\C\chi}]^{\C^\times}$-action on $\operatorname{Tor}_i^{\C[X_{\C\chi}]}(\gr\B^{-,o}_{\ell}, \gr\B^+_{\ell})$ is restricted from
the $\C[X_{\C\chi}]$-action (we take the Tor of modules over the commutative ring
$\C[X_{\C\chi}]$). For similar reasons, the same is true for the right action.
So $\gr \operatorname{Tor}_i^{\A_{\ell}}(\B^{-,o}_{\ell}, \B^+_{\ell})$ is a $\C[X_{\C\chi}]^{\nu_0(\C^\times)}$-module.
Hence
$\operatorname{Tor}_i^{\A_{\ell}}(\B^{+,o}_{\ell}, \B^-_{\ell})$ is indeed a HC $\Ca_{\nu_0}(\A_\ell)$-bimodule.
\end{proof}


\begin{proof}[Proof Proposition \ref{Prop_CW_prop}, (1)]
Let us show that $\CW_{\nu\rightarrow \nu'}^*\circ \nabla_{\nu_0}\cong \nabla_{\nu_0}\circ
\underline{\CW}_{\nu\rightarrow \nu'}^*$, the other isomorphism is similar.

We need to establish a bi-functorial isomorphism
\begin{equation}\label{eq:bifun_iso1}\Hom_{D^b(\A_{\lambda^+}\operatorname{-mod})}(M, \CW^*_{\nu\rightarrow \nu'}\nabla_{\nu_0}(N))\xrightarrow{\sim}
\Hom_{D^b(\A_{\lambda^+}\operatorname{-mod})}(M, \nabla_{\nu_0}(\underline{\CW}^*_{\nu\rightarrow \nu'}N)),\end{equation}
for $M\in D^b_{\OCat_\nu}(\A_{\lambda^+}\operatorname{-mod}),
N\in D^b_{\OCat_{\nu'}}(\Ca_{\nu_0}(\A_{\lambda^+})\operatorname{-mod})$.

By the definition of $\CW^*_{\nu\rightarrow \nu'}$, we have
\begin{equation}\label{eq:bifun_iso2}
\Hom_{D^b(\A_{\lambda^+}\operatorname{-mod})}(M, \CW^*_{\nu\rightarrow \nu'}\nabla_{\nu_0}(N))\xrightarrow{\sim}
\Hom_{D^b(\A_{\lambda^+}\operatorname{-mod})}(M,\nabla_{\nu_0}(N)).
\end{equation}
On the other hand,
\begin{equation}\label{eq:bifun_iso3}
\begin{split}
&\Hom_{D^b(\A_{\lambda^+}\operatorname{-mod})}(M,\nabla_{\nu_0}(N))=
\Hom_{D^b(\A_{\lambda^+}\operatorname{-mod})}(M,R\Hom_{\Ca_{\nu_0}(\A_{\lambda^+})}(\B^{-,o}_{\lambda^+},N))\\
&\xrightarrow{\sim}\Hom_{D^b(\Ca_{\nu_0}(\A_{\lambda^+})\operatorname{-mod})}
(\B^{-,o}_{\lambda^+}\otimes^L_{\A_{\lambda^+}}M,N).
\end{split}
\end{equation}
The first equality holds thanks to Lemma \ref{Lem:der_st_cost}.

Note that $\B^{-,o}_{\lambda^+}\otimes^L_{\A_{\lambda^+}}M\in
D^b_{\OCat_{\nu}}(\Ca_{\nu_0}(\A_{\lambda^+})\operatorname{-mod})$. Indeed, this is true
for $M=\Delta_{\nu_0}(M')$, where $M'\in \OCat_\nu(\Ca_{\nu_0}(\A_{\lambda^+})$,
because, by Lemma \ref{Lem:der_tens_prod}, we have
$$\B^{-,o}_{\lambda^+}\otimes^L_{\A_{\lambda^+}}M=
\B^{-,o}_{\lambda^+}\otimes^L_{\A_{\lambda^+}}\B^+_{\lambda^+}\otimes^L_{\Ca_{\nu_0}(\A_{\lambda^+})}M'=M'.$$
Since the objects of the form $\Delta_{\nu_0}(M')$ span $D^b_{\OCat_\nu}(\A_{\lambda^+}\operatorname{-mod})$,
we deduce that $\B^{-,o}_{\lambda^+}\otimes^L_{\A_{\lambda^+}}M\in
D^b_{\OCat_{\nu}}(\Ca_{\nu_0}(\A_{\lambda^+})\operatorname{-mod})$ for all $M\in
D^b_{\OCat_\nu}(\A_{\lambda^+}\operatorname{-mod})$. It follows that
\begin{equation}\label{eq:bifun_iso4}
\begin{split}
&\Hom_{D^b(\Ca_{\nu_0}(\A_{\lambda^+})\operatorname{-mod})}
(\B^{-,o}_{\lambda^+}\otimes^L_{\A_{\lambda^+}}M,N)\\
&\xrightarrow{\sim}
\Hom_{D^b(\Ca_{\nu_0}(\A_{\lambda^+})\operatorname{-mod})}
(\B^{-,o}_{\lambda^+}\otimes^L_{\A_{\lambda^+}}M,\underline{\CW}^*_{\nu\rightarrow \nu'}N)\\
&\xrightarrow{\sim} \Hom_{D^b(\A_{\lambda^+}\operatorname{-mod})}
(M,\nabla_{\nu_0}(\underline{\CW}^*_{\nu\rightarrow \nu'}N)).
\end{split}
\end{equation}
Combining (\ref{eq:bifun_iso2}),(\ref{eq:bifun_iso3}) and (\ref{eq:bifun_iso4}),
we get (\ref{eq:bifun_iso1}).
\end{proof}

%

\subsection{Proof of (2) of Proposition \ref{Prop_CW_prop}}
In this section we prove (2) of Proposition \ref{Prop_CW_prop}.

Pick some parameter $\lambda$.

\begin{Lem}\label{CW_vs_D}
We have a bi-functorial isomorphism $$R\operatorname{Hom}_{D^b(\OCat_{\nu'}(\A_{\lambda^+}))}(\CW_{\nu\rightarrow \nu'}M, N)=D(M)\otimes^L_{\A_{\lambda^+}}N.$$
\end{Lem}
\begin{proof}
(\ref{eq:CW_char}) implies that  the left hand side is $R\operatorname{Hom}_{D^b(\A_{\lambda^+})}(M,N)$.
The claim is now standard.
\end{proof}

One can compute $D(\Delta_{\nu}(p))$. The following claim was obtained
in the proof of Proposition \ref{Prop:D_Ringel} (and is a formal corollary of that
proposition).

\begin{Lem}\label{Lem:dual_stand}
We have $D(\Delta_\nu(p))=\Delta^{opp}_{\nu}(p)[-\frac{1}{2}\dim X]$.
\end{Lem}

\begin{Cor}\label{Cor:CW_costand}
We have $\CW^*_{\nu\rightarrow -\nu}\Delta_{-\nu}(p)=\nabla_{\nu}(p)[\frac{1}{2}\dim X]$.
\end{Cor}
\begin{proof}
We have \begin{align*}
&R\Hom_{\OCat_{-\nu}(\A_{\lambda^+})}(\Delta_{\nu}(p), \CW^*_{\nu\rightarrow -\nu}\Delta_{-\nu}(p')[k])
=D(\Delta_{\nu}(p))\otimes^L_{\A_{\lambda^+}}\Delta_{-\nu}(p')[k]\\
&=\Delta_{\nu}^{opp}(p)\otimes^L_{\A_{\lambda^+}}\Delta_{-\nu}(p')[k-\frac{1}{2}\dim X].
\end{align*}
It follows from Lemma \ref{Lem:der_tens_prod} that the last expression is zero unless
$p=p'$ and $k=\frac{1}{2}\dim X$, in which case it is one dimensional. Therefore  $\CW^*_{\nu\rightarrow -\nu}\Delta_{-\nu}(p')=\nabla_\nu(p')[\frac{1}{2}\dim X]$.
\end{proof}

Let us proceed to the proof of (2).

Recall, see Section \ref{SSS_cat_O_gen}, that $\OCat_{-\nu}(\A_{\lambda^+})\cong \OCat_\nu(\A_{\lambda^+}^{opp})^{opp}$
via $N\mapsto N^{(*)}$ (the restricted dual).

\begin{Lem}\label{Lem:tens_prod_vs_hom}
For $M\in \OCat_{\nu}(\A_{\lambda^+}^{opp}), N\in \OCat_{-\nu}(\A_{\lambda^+})$ we have a
bifunctorial isomorphism
\begin{equation}\label{eq:hom_dual}\Hom_{\A_{\lambda^+}^{opp}}(M,N^{(*)})=(M\otimes_{\A_{\lambda^+}}N)^*.\end{equation}
\end{Lem}
\begin{proof}
The space on the right is that of $\K$-bilinear maps $M\times N\rightarrow \K$ satisfying
$\langle ma,n\rangle=\langle m,an\rangle$. So we have a homomorphism from the left hand side
to the right hand side that to an $\A_{\lambda^+}^{opp}$-linear map $\varphi: M\rightarrow N^{(*)}$ assigns
the linear map $\langle m,n\rangle_{\varphi}:=\langle\varphi(m),n\rangle$. This homomorphism
is clearly bifunctorial and injective. So it remains to prove that it is surjective.
By the 5 lemma it is enough to prove the surjectivity when both $M,N$ are projective. We will do this
in the case when $M,N$ are standardly filtered. When $N$ is standardly filtered, the
object $N^{(*)}$ is costandardly filtered. So the left hand side has dimension
$\sum_{p\in X^T}[M:\Delta^{opp}_\nu(p)][N:\Delta_{-\nu}(p)]$. By a direct analog of Lemma
\ref{Lem:der_tens_prod}, the functor
$$\bullet\otimes^L_{\A_{\lambda^+}}?: D^b(\OCat_{\nu}(\A^{opp}_{\lambda^+}))\times
D^b(\OCat_{-\nu}(\A_{\lambda^+}))\rightarrow D^b(\operatorname{Vect})$$
is acyclic on standardly filtered objects and also
$\dim\Delta^{opp}_{\nu}(p)\otimes_{\A_{\lambda^+}}\Delta_{-\nu}(p')=\delta_{pp'}$.
It follows that the dimension of the right hand side of (\ref{eq:hom_dual}) is also
$\sum_{p\in X^T}[M:\Delta^{opp}_\nu(p)][N:\Delta_{-\nu}(p)]$.
\end{proof}

Lemma \ref{Lem:tens_prod_vs_hom} implies that
$$ \Hom_{\A_{\lambda^+}^{opp}}(M,N^{(*)})=(M\otimes_{\A_{\lambda^+}}N)^*,\quad
M\in D^b(\OCat_\nu(\A^{opp}_{\lambda^+})), N\in D^b(\OCat_{-\nu}(\A_{\lambda^+})).$$
So we get the following
bi-functorial isomorphism:
\begin{align*}
&\Hom_{D^b(\OCat_{-\nu}(\A_{\lambda^+}))}(\CW_{\nu\rightarrow -\nu}(M),N)= D(M)\otimes_{\A_{\lambda^+}}N=\\
&=\Hom_{D^b(\OCat_{\nu}(\A^{opp}_{\lambda^+}))}(D(M), N^{(*)})^*=\Hom_{D^b(\OCat_{-\nu}(\A_{\lambda^+}))}(N, D^{(*)}(M))^*.
\end{align*}
Here we write $D^{(*)}(M)$ for $(D(M))^{(*)}$.  Note that $D^{(*)}$ is an equivalence. So the functor $\mathcal{F}:=\CW_{\nu\rightarrow -\nu}\circ (D^{(*)})^{-1}$
satisfies
\begin{equation}\label{eq:Serre}\Hom_{D^b(\OCat_{-\nu}(\A_{\lambda^+}))}(\mathcal{F}N', N)\cong \Hom_{D^b(\OCat_{-\nu}(\A_{\lambda^+}))}(N,N')^*.\end{equation}
From here it is easy to see that $\mathcal{F}$ is fully faithful. Also $\mathcal{F}$
has right adjoint: $\mathcal{F}^*=D^{(*)}\circ \CW_{\nu\rightarrow -\nu}^*$.
Since $\mathcal{F}$ is fully faithful, the adjunction morphism $\operatorname{id}\rightarrow\mathcal{F}^*\circ \mathcal{F}$ is an isomorphism. It follows from (\ref{eq:Serre})
 that the adjoint $\mathcal{F}^*$ is also fully faithful. So we conclude that
$\mathcal{F},\mathcal{F}^*$ are mutually inverse equivalences. Therefore $\CW_{\nu\rightarrow -\nu}$
is an equivalence of triangulated categories.

Note that $\CW^*_{\nu\rightarrow -\nu}[\frac{1}{2}\dim X]$ sends $\Delta_{-\nu}(p)$
to $\nabla_\nu(p)$. Since $\CW^*_{\nu\rightarrow -\nu}[\frac{1}{2}\dim X]$
is an equivalence, it follows from Corollary \ref{Cor:CW_costand} that
$\CW^*_{\nu\rightarrow -\nu}[\frac{1}{2}\dim X]$ is a Ringel duality functor.
This finishes the proof of (2).

\subsection{Short cross-walling functors are equivalences}\label{SS_short_CW_equi}
Here we consider the case when $\nu,\nu'$ lie in chambers sharing a wall (we say that $\nu,\nu'$
are neighbors). We will show that  $\CW_{\nu\rightarrow \nu'}$ (a short wall-crossing functor) is an equivalence.

We will prove a stronger result. Let $\nu,\nu'\rightsquigarrow \nu_0$. Let $\underline{\CW}_{\nu\rightarrow\nu'}$
denote the cross-walling functor $D^b(\OCat_\nu(\Ca_{\nu_0}(\A_\lambda^\theta)))\rightarrow
D^b(\OCat_{\nu'}(\Ca_{\nu_0}(\A_\lambda^\theta)))$.
\begin{Lem}\label{Lem:equiv_inherit}
If $\underline{\CW}_{\nu\rightarrow \nu'}$ is an equivalence, then so is $\CW_{\nu\rightarrow \nu'}$.
\end{Lem}
The claim that $\CW_{\nu\rightarrow \nu'}$ is an equivalence will follow: thanks
to (2) of Proposition \ref{Prop_CW_prop}, when $\nu,\nu'$ are neighbors, the functor
$\underline{\CW}_{\nu\rightarrow\nu'}$ is the sum of Ringel dualities with various shifts.

\begin{proof}[Proof of Lemma \ref{Lem:equiv_inherit}]
Recall that, according to (1) of Proposition \ref{Prop_CW_prop},
\begin{equation}\label{eq:func_equal}
\begin{split}
&\CW_{\nu\rightarrow \nu'}\circ \Delta_{\nu_0}\cong \Delta_{\nu_0}\circ \underline{\CW}_{\nu\rightarrow \nu'},\\
&\CW^*_{\nu\rightarrow \nu'}\circ \nabla_{\nu_0}\cong \nabla_{\nu_0}\circ \underline{\CW}^*_{\nu\rightarrow \nu'}.
\end{split}
\end{equation}
It follows that both $\CW_{\nu\rightarrow \nu'},\CW^*_{\nu\rightarrow \nu'}$ respect the
filtrations $$D^b_{\preceq_{\nu_0}Z}(\OCat_\nu(\A_{\lambda^+}))\subset D^b(\OCat_\nu(\A_{\lambda^+})),
D^b_{\preceq_{\nu_0}Z}(\OCat_{\nu'}(\A_{\lambda^+}))\subset D^b(\OCat_{\nu'}(\A_{\lambda^+})),$$
where $Z$ runs over the poset of the connected components of $X^{\nu_0(\C^\times)}$.
Recall, Lemma \ref{Lem:derived_inclusion}, that  $D^b(\OCat_\nu(\A_{\lambda^+})_{\preceq_{\nu_0}Z})$
is a full subcategory in $D^b(\OCat_\nu(\A_{\lambda^+})$.
Moreover,  the associated graded categories are just $D^b(\OCat_\nu(\Ca_{\nu_0}(\A_{\lambda^+}))$, $
D^b(\OCat_{\nu'}(\Ca_{\nu_0}(\A_{\lambda^+}))$ by Proposition \ref{Prop:assoc_graded_O}. It follows
from (\ref{eq:func_equal}) that  the  functors $D^b(\OCat_\nu(\Ca_{\nu_0}(\A_{\lambda^+}))\rightleftarrows D^b(\OCat_{\nu'}(\Ca_{\nu_0}(\A_{\lambda^+}))$ induced by  $\CW_{\nu\rightarrow \nu'},
\CW^*_{\nu\rightarrow \nu'}$ are $\underline{\CW}_{\nu\rightarrow \nu'}, \underline{\CW}^*_{\nu\rightarrow \nu'}$.
The latter are mutually quasi-inverse equivalences. So if $\CW_{\nu\rightarrow \nu'}(K)=0$
for some $K\in D^b(\OCat_\nu(\A_{\lambda^+}))$, then $K=0$. Applying this
for $K$ being the cone of the adjunction morphism $M\rightarrow \CW^*_{\nu\rightarrow \nu'}\circ \CW_{\nu\rightarrow \nu'}M$, we see that $M\xrightarrow{\sim} \CW^*_{\nu\rightarrow \nu'}\circ \CW_{\nu\rightarrow \nu'}M$.
Similarly, we get that $\CW_{\nu\rightarrow \nu'}\circ \CW^*_{\nu\rightarrow \nu'}N\xrightarrow{\sim}N$.
We conclude that $\CW_{\nu\rightarrow \nu'}, \CW^*_{\nu\rightarrow \nu'}$ are mutually inverse equivalences.
\end{proof}

\subsection{Proof of Theorem \ref{Thm:braid_equiv}}
The following theorem is a crucial step in the proof of (and also a special case of) Theorem \ref{Thm:braid_equiv}.

\begin{Thm}\label{Thm:short_CW_comp}
Let $\nu,\nu',\nu''$ be three generic one-parameter subgroups forming a reduced sequence (in this order)
such that $\nu',\nu''$ are neighbors. Then $\CW_{\nu\rightarrow \nu''}\xrightarrow{\sim}\CW_{\nu'\rightarrow \nu''}\circ \CW_{\nu\rightarrow \nu'}$.
\end{Thm}

A key step in the proof is the next proposition. Let $\bar{\nu}$ be a generic one-parameter
subgroup that is a neighbor of $\nu$ and such that
the sequence $\nu,\bar{\nu}, \nu',\nu''$ is reduced (note that here we can have $\bar{\nu}=\nu'$).
Let $\nu_1,\nu_2$ be such that $\nu,\bar{\nu}\rightsquigarrow \nu_1, \nu',\nu''\rightsquigarrow \nu_2$
and $\nu_1$ lies on the wall between $\nu,\bar{\nu}$, while $\nu_2$ lies on the wall between
$\nu',\nu''$. Below we will write $\A$ instead of $\A_{\lambda^+}$.

\begin{Prop}\label{Lem:Hom_iso}
For $M\in D^b(\OCat_\nu(\Ca_{\nu_1}(\A))), N\in D^b(\OCat_{\nu''}(\Ca_{\nu_2}(\A)))$,
the natural homomorphism
\begin{equation}\label{eq:nat_hom} \Hom_{\A}(\Delta_{\nu_1}(M),\CW^*_{\nu\rightarrow \nu'}\CW^*_{\nu'\rightarrow \nu''}\nabla_{\nu_2}(N))
\rightarrow \Hom_{\A}(\Delta_{\nu_1}(M),\CW_{\nu\rightarrow \nu''}^*\nabla_{\nu_2}(N))\end{equation}
is an isomorphism.
\end{Prop}

We will give a proof of Proposition \ref{Lem:Hom_iso} after a series of auxiliary results.

We start by giving an alternative interpretation of the homomorphism (\ref{eq:nat_hom}). Below we identify
all categories $D^b(\OCat_{\tilde{\nu}}(\A))$ with $D^b_{\OCat_{\tilde{\nu}}}(\A\operatorname{-mod})$ so that the functors
$\iota_{\tilde{\nu}}$ become the inclusions. So we will omit the functors $\iota_{\tilde{\nu}}$
and get functor morphisms $\iota^*_{\tilde{\nu}}\rightarrow\operatorname{id}$.

\begin{Lem}\label{Lem:fun_commut_diagr1}
For $\tilde{M}\in D^b_{\OCat_\nu}(\A\operatorname{-mod}), \tilde{N}\in D^b_{\OCat_{\nu''}}
(\A\operatorname{-mod})$, there are bifunctorial isomorphisms
\begin{align*}&\Hom_{D^b(\A\operatorname{-mod})}(\tilde{M},\tilde{N})\xrightarrow{\sim}
\Hom_{D^b(\A\operatorname{-mod})}(\tilde{M},\CW_{\nu\rightarrow \nu''}^*\tilde{N}),\\
&\Hom_{D^b(\A\operatorname{-mod})}(\tilde{M},\iota_{\nu'}^*\tilde{N})
\xrightarrow{\sim}
\Hom_{D^b(\A\operatorname{-mod})}(\tilde{M},\CW^*_{\nu\rightarrow \nu'}\CW_{\nu'\rightarrow \nu''}^*\tilde{N})
\end{align*}
making the following diagram commutative

\begin{picture}(150,30)
\put(2,2){$\Hom_{D^b(\A\operatorname{-mod})}(\tilde{M},\CW^*_{\nu\rightarrow \nu'}\CW_{\nu'\rightarrow \nu''}^*\tilde{N})$}
\put(92,2){$\Hom_{D^b(\A\operatorname{-mod})}(\tilde{M},\CW_{\nu\rightarrow \nu''}^*\tilde{N})$}
\put(14,22){$\Hom_{D^b(\A\operatorname{-mod})}(\tilde{M},\iota_{\nu'}^*\tilde{N})$}
\put(100,22){$\Hom_{D^b(\A\operatorname{-mod})}(\tilde{M},\tilde{N})$}
\put(30,20){\vector(0,-1){13}}
\put(115,20){\vector(0,-1){13}}
\put(75,4){\vector(1,0){15}}
\put(60,24){\vector(1,0){37}}
\end{picture}
\end{Lem}

The proof is straightforward.

Consider the $\A$-$\Ca_{\nu_1}(\A)$-bimodule $\B^1:=\A/\A\A^{>0,\nu_1}$ and the
$\Ca_{\nu_2}(\A)$-$\A$-bimodule $\B^2:=\A/\A^{<0,\nu_2}\A$
so that $\Delta_{\nu_1}(\bullet)=\B^1\otimes^L_{\Ca_{\nu_1}(\A)}\bullet$
and $\nabla_{\nu_2}(\bullet)=R\Hom_{\Ca_{\nu_2}(\A)}(\B^2,\bullet)$, by Lemma \ref{Lem:der_st_cost}.
Set $\mathcal{Q}:=\B^2\otimes^L_{\A}\B^1$. This is an object
in $D^b(\Ca_{\nu_2}(\A)\otimes \Ca_{\nu_1}(\A)^{opp}\operatorname{-Mod})$
(the notation ``Mod'' stand for the category of all modules).

Let $\underline{\iota}_{\nu'}$ denote the inclusion $D^b(\OCat_{\nu'}(\Ca_{\nu_2}(\A)))=
D^b_{\OCat_{\nu'}}(\Ca_{\nu_2}(\A)\operatorname{-mod})
\hookrightarrow D^b(\Ca_{\nu_2}(\A)\operatorname{-mod})$. So we get the right adjoint functor
$\underline{\iota}^*_{\nu'}: D^b(\Ca_{\nu_2}(\A)\operatorname{-mod})\rightarrow
D^b_{\OCat_{\nu'}}(\Ca_{\nu_2}(\A)\operatorname{-mod})$. Also we consider the homological duality functor
$\underline{D}:D^b(\Ca_{\nu_1}(\A)\operatorname{-mod})
\xrightarrow{\sim}D^b(\Ca_{\nu_1}(\A)^{opp}\operatorname{-mod})$.

\begin{Lem}\label{Lem:fun_commut_diagr2}
For $M\in D^b_{\OCat_\nu}(\Ca_{\nu_1}(\A)\operatorname{-mod}),
N\in D^b_{\OCat_{\nu''}}
(\Ca_{\nu_2}(\A)\operatorname{-mod})$, there are bifunctorial isomorphisms
\begin{align*}&\Hom_{D^b(\A\operatorname{-mod})}(\Delta_{\nu_1}(M),\nabla_{\nu_2}(N))\xrightarrow{\sim}
\Hom_{D^b(\Ca_{\nu_2}(\A)\otimes \Ca_{\nu_1}(\A)^{opp}\operatorname{-Mod})}(\mathcal{Q},N\boxtimes \underline{D}M),\\
&\Hom_{D^b(\A\operatorname{-Mod})}(\Delta_{\nu_1}(M),\iota_{\nu'}^*\nabla_{\nu_2}(N))
\xrightarrow{\sim}
\Hom_{D^b(\Ca_{\nu_2}(\A)\otimes \Ca_{\nu_1}(\A)^{opp}\operatorname{-Mod})}(\mathcal{Q},\underline{\iota}^*_{\nu'}(N)\boxtimes \underline{D}M)
\end{align*}
making the following diagram commutative

\begin{picture}(150,30)
\put(8,2){$\Hom_{\ldots}(\mathcal{Q},\underline{\iota}^*_{\nu'}(N)\boxtimes \underline{D}M)$}
\put(94,2){$\Hom_{\ldots}(\mathcal{Q},N\boxtimes \underline{D}M)$}
\put(2,22){$\Hom_{D^b(\A\operatorname{-mod})}(\Delta_{\nu_1}(M),\iota_{\nu'}^*\nabla_{\nu_2}(N))$}
\put(86,22){$\Hom_{D^b(\A\operatorname{-mod})}(\Delta_{\nu_1}(M),\nabla_{\nu_2}(N))$}
\put(30,20){\vector(0,-1){13}}
\put(107,20){\vector(0,-1){13}}
\put(60,4){\vector(1,0){32}}
\put(66,24){\vector(1,0){18}}
\end{picture}
\end{Lem}
\begin{proof}
First of all, there is a functor isomorphism
$\nabla_{\nu_2}\circ \underline{\iota}^*_{\nu'}\xrightarrow{\sim} \iota^*_{\nu'}\circ \nabla_{\nu_2}$
intertwining the morphisms $\nabla_{\nu_2}\circ \underline{\iota}^*_{\nu'},
\iota^*_{\nu'}\circ \nabla_{\nu_2}\rightarrow \nabla_{\nu_2}$. This was
established in the proof of (1) of Proposition
\ref{Prop_CW_prop} (recall that on $D^b_{\OCat_\nu}(\A\operatorname{-mod})$ the functor
$\iota^*_{\nu'}$ coincides with $\CW^*_{\nu\rightarrow \nu'}$ and a similar claim holds for
$\underline{\iota}^*_{\nu'}$).

The claim of the lemma is now a consequence of the following
bifunctorial isomorphisms (and similar isomorphisms with $N$ replaced with $\underline{\iota}^*_{\nu'}N$)
\begin{align*}
&\Hom_{D^b(\A\operatorname{-mod})}(\Delta_{\nu_1}(M),\nabla_{\nu_2}(N))=\\
&\Hom_{D^b(\A\operatorname{-mod})}(\B^1\otimes_{\Ca_{\nu_1}(\A)}^L M ,R\Hom_{\Ca_{\nu_2}(\A)}(\B^2,N))=\\
&\Hom_{D^b(\Ca_{\nu_2}(\A)\operatorname{-Mod})}(\mathcal{Q}\otimes^L_{\Ca_{\nu_1}(\A)}M, N)=\\
&\Hom_{D^b(\Ca_{\nu_2}(\A)\otimes \Ca_{\nu_1}(\A)^{opp}\operatorname{-Mod})}(\mathcal{Q}, N\boxtimes \underline{D}M).
\end{align*}
\end{proof}

Now we will need to get some information about $\mathcal{Q}\in D^b(\Ca_{\nu_2}(\A)\otimes \Ca_{\nu_1}(\A)^{opp}\operatorname{-Mod})$.

\begin{Lem}\label{Lem:Q_struct}
We have $H_\bullet(\mathcal{Q})\in\OCat_{\nu'}(\Ca_{\nu_2}(\A))\boxtimes
\OCat_{\nu}(\Ca_{\nu_1}(\A)^{opp})$.
\end{Lem}
\begin{proof}
The proof is in several steps.

{\it Step 1}. Note that both bimodules $\mathcal{B}^1$ and $\mathcal{B}^2$ are weakly $T$-equivariant.
So $H_\bullet(\mathcal{Q})$ is weakly $T$-equivariant as well. Let us check that the numbers
$\langle \nu_1, \alpha\rangle, -\langle \nu_2,\alpha\rangle$ are bounded from above for all
$T$-weights $\alpha$ appearing in  $H_\bullet(\mathcal{Q})$. Let us prove this claim
for $\langle \nu_1,\alpha\rangle$, the proof for $-\langle \nu_2,\alpha\rangle$ is similar.
Recall that $\mathcal{B}^1=\A/\A\A^{>0,\nu_1}$. Choose a projective weakly $T$-equivariant
$\A$-module resolution for $\B^2$. We see that  $H_\bullet(\mathcal{Q})$ is also the
homology of a complex whose terms are direct sums of several copies of $\mathcal{B}^1$ with
twisted $T$-actions. Our claim follows.

{\it Step 2}. Now let us check that  $H_\bullet(\mathcal{Q})$ is the union of objects in
\begin{equation}\label{eq:prod_cat}\OCat_{\nu'}(\Ca_{\nu_2}(\A))\boxtimes
\OCat_{\nu}(\Ca_{\nu_1}(\A)^{opp}).\end{equation} Consider the action of
$(\C^\times)^2$ on $\Ca_{\nu_2}(\A)\otimes\Ca_{\nu_1}(\A)^{opp}$
by $\nu_1\times \nu_2$ so that the first copy of $\C^\times$ acts trivially on the second factor
and vice versa. Since the numbers $\langle\nu_1,\alpha\rangle, -\langle \nu_2,\alpha\rangle$
are bounded by above, we see that $H_\bullet(\mathcal{Q})$ is the union of objects  in
$\OCat_{\nu_1}(\Ca_{\nu_2}(\A))\boxtimes
\OCat_{-\nu_2}(\Ca_{\nu_1}(\A)^{opp})$.
The latter category coincides with (\ref{eq:prod_cat}). This is because $\nu$ and $\nu_1$ lie in
one half-space with respect to the wall containing $\nu_2$, and $\nu',\nu_2$ also lie in one half-space with respect to
the wall containing $\nu_1$ (it is here that we use our choice of the one-parameter subgroups involved).

{\it Step 3}. Now we claim that each $H_i(\mathcal{Q})$ is actually the direct sum of objects
in (\ref{eq:prod_cat}). Set $\tilde{\nu}:=\nu_1-\nu_2$. Note that all weight spaces
for $\tilde{\nu}$ in $H_i(Q)$ are finite dimensional and the weights are bounded from above. Indeed, it is enough to check
this for   $\operatorname{Tor}_i^{\C[X_0]}(\gr\mathcal{B}^2, \gr\mathcal{B}^1)$ because
$$\gr \operatorname{Tor}_i^{\A}(\mathcal{B}^2,\mathcal{B}^1)\subset
\operatorname{Tor}_i^{\C[X_0]}(\gr\mathcal{B}^2,\gr\mathcal{B}^1).$$
Note that $\operatorname{Tor}_i^{\C[X_0]}(\gr\mathcal{B}^2, \gr\mathcal{B}^1)$
is a finitely generated module over $\C[X_0]/\C[X_0](\C[X_0]^{>0,\nu_1}+\C[X]^{<0,\nu_2})$.
The eigencharacters of $\tilde{\nu}$ on the latter  are all non-negative. The zero eigenspace
is easily seen to be finite dimensional. Since the algebra $\C[X_0]/\C[X_0](\C[X_0]^{>0,\nu_1}+\C[X]^{<0,\nu_2})$
is finitely generated,   every eigenspace is finite dimensional. This shows that the weights of
the $\tilde{\nu}$-action on $H_i(\mathcal{Q})$ are bounded from above and all eigen-spaces are finite
dimensional.

Let $h$ denote the image of $1$ under the quantum comoment map for the actions of $\nu_1-\nu_2$ on
$\Ca_{\nu_2}(\A)\otimes\Ca_{\nu_1}(\A)^{opp}$. The element $h$ preserves the grading hence acts locally finitely on
$H_i(\mathcal{Q})$.
For $z\in \C, j\in \Z$, let $H_i(\mathcal{Q})[z,j]$ denote the generalized $z$-eigenspace for $h$
in the component of degree $j$. So $M_{z}:=\bigoplus_j H_i(\mathcal{Q})[z+j,j]$ is a submodule in
$H_i(\mathcal{Q})$ and $H_i(\mathcal{Q})=\bigoplus_{z\in \C}M_{z}$. By (4) of Lemma \ref{Lem:fin_length}, $M_z$
lies in (\ref{eq:prod_cat}).

{\it Step 4}. To check that $H_i(\mathcal{Q})$ lies in (\ref{eq:prod_cat}) it is enough to show
that $$\dim\operatorname{Hom}_{D^b(\Ca_{\nu_2}(\A))\otimes\Ca_{\nu_1}(\A)^{opp}\operatorname{-Mod})}(\mathcal{Q},
I[k])<\infty$$ for any indecomposable injective object $I$ in (\ref{eq:prod_cat}) and any integer
$k$. Recall, Lemma \ref{Lem:fun_commut_diagr2}, that $$\operatorname{Hom}_{D^b(\Ca_{\nu_2}(\A))\otimes\Ca_{\nu_1}(\A)^{opp}\operatorname{-mod})}
(\mathcal{Q},\iota_{\nu'}^*N\boxtimes \underline{D}M)=\operatorname{Hom}_{D^b(\A)}(\Delta_{\nu_1}(M),\nabla_{\nu_2}(M)).$$
The right hand side is finite dimensional. Now we remark that any indecomposable injective
in (\ref{eq:prod_cat}) is a product of the indecomposable injectives in the factors, this follows
from Lemma \ref{Lem:product}. Both
$\underline{\iota}^*_{\nu'}: D^b(\OCat_{-\nu'}(\Ca_{\nu_2}(\A)))\rightarrow
D^b(\OCat_{\nu'}(\Ca_{\nu_2}(\A)))$ and $\underline{D}:
D^b(\OCat_{\nu}(\Ca_{\nu_1}(\A)))\rightarrow D^b(\OCat_{\nu}(\Ca_{\nu_1}(\A)^{opp}))$
are equivalences. So any injective in (\ref{eq:prod_cat}) has the form $N\boxtimes \underline{D}(M)$.
\end{proof}

\begin{proof}[Proof of Proposition \ref{Lem:Hom_iso}]
We claim that the bottom row of the commutative diagram in Lemma \ref{Lem:fun_commut_diagr2}
is an isomorphism. For this we note that both $\mathcal{Q}$ and $\underline{\iota}^*_{\nu'}N\boxtimes
\underline{D}M$ lie in $D^b(\OCat_{\nu'}(\Ca_{\nu_2}(\A)\boxtimes
\OCat_{\nu}(\Ca_{\nu_1}(\A)^{opp}))$. Furthermore, on $N\boxtimes \underline{D}M
\in \OCat_{\nu''}(\Ca_{\nu_2}(\A))\boxtimes
\OCat_{\nu}(\Ca_{\nu_1}(\A)^{opp}$, the functor $\underline{\iota}_{\nu'}^*\boxtimes
\operatorname{id}$ coincides $\underline{\CW}^*_{\nu'\rightarrow \nu''}\boxtimes \operatorname{id}=\underline{\CW}^*_{(\nu',\nu)\rightarrow (\nu'',\nu)}$, the last equality
holds thanks to Lemma \ref{Lem:CW_prod}. So the claim in the beginning of the proof follows.
\end{proof}

\begin{proof}[Proof of Theorem \ref{Thm:short_CW_comp}]
First of all, we claim that, in the notation of Proposition \ref{Lem:Hom_iso},
the natural morphism $\CW^*_{\nu\rightarrow \nu'}\circ \CW^*_{\nu'\rightarrow \nu''}\nabla_{\nu_2}(N)
\rightarrow \CW^*_{\nu\rightarrow \nu''}\nabla_{\nu_2}(N)$ is an isomorphism. Indeed, let $C$ be the cone.
Then, by Proposition \ref{Lem:Hom_iso}, we have $R\Hom_{\OCat_{\nu}(\A)}(\Delta_{\nu_1}(M), C)=0$.
Since all standard objects in $\OCat_{\nu}(\A)$ are of the form $\Delta_{\nu_1}(M)$
we see that $C=0$.

Now the functor $\CW^*_{\nu'\rightarrow \nu''}\circ \CW^*_{\nu\rightarrow \nu'}\rightarrow
\CW^*_{\nu\rightarrow \nu''}$  is an isomorphism on all costandard objects in
$\OCat_{\nu''}(\A)$. So it is an isomorphism.
\end{proof}

\begin{proof}[Proof of Theorem \ref{Thm:braid_equiv}]
The part about the isomorphism of functors follows from Theorem \ref{Thm:short_CW_comp}
and easy induction. The claim that $\CW^*_{\nu\rightarrow \nu'}$ is an equivalence
now follows from the existence of a maximal reduced decomposition and
Lemma \ref{Lem:equiv_inherit}.
\end{proof}

\section{Appendix: comparison to \cite{MO}}
In \cite{MO}, Maulik and Okounkov gave a geometric construction of $R$-matrices.
In this section we will argue that the cross-walling functors should be viewed
as a categorical version of their construction.

\subsection{Stable envelopes and Verma modules}
Namely, let $X$ be a conical symplectic resolution equipped with a Hamiltonian
action of a torus $T$. We assume that $T$ has finitely many fixed points in
$X$ (although this is not required in \cite{MO}). Given a cocharacter
$\nu$, Maulik and Okounkov, \cite[Section 3.5]{MO}, constructed a map
$\mathsf{Stab}_\nu: H^*_{T\times \C^\times}(X^{\nu(\C^\times)})\rightarrow H^*_{T\times \C^\times}(X)$
that becomes an isomorphism on the localized equivariant cohomology (here
$\C^\times$ stands for the contracting torus). The map enjoys a certain transitivity
property. Namely, pick $\nu_0$ with $\nu\rightsquigarrow \nu_0$.
Then we get maps $\mathsf{Stab}_{\nu_0}:H^*_{T\times \C^\times}(X^{\nu_0(\C^\times)})
\rightarrow H^*_{T\times \C^\times}(X)$
and $\underline{\mathsf{Stab}}_{\nu}: H^*_{T\times \C^\times}(X^{\nu(\C^\times)})
\rightarrow H^*_{T\times \C^\times}(X^{\nu_0(\C^\times)})$. We have the following
equality \cite[Lemma 3.6.1]{MO}: $\mathsf{Stab}_\nu=\mathsf{Stab}_{\nu_0}\circ \underline{\mathsf{Stab}}_{\nu}$.

Following the original intuition of Okounkov, we would like to claim that the functor $\Delta_{\nu}$
should be viewed as a categorical version of $\mathsf{Stab}_\nu$. Namely, pick a generic
co-character $\tilde{\nu}$ such that $\tilde{\nu}\rightsquigarrow\nu$. We replace
$H^*_{T\times \C^\times}(X),H^*_{T\times \C^\times}(X^{\nu(\C^\times)})$
with $\OCat_{\tilde{\nu}}(\A_{\lambda^+}), \OCat_{\tilde{\nu}}(\Ca_\nu(\A_{\lambda^+}))$
so that we have an exact functor $\Delta_\nu:\OCat_{\tilde{\nu}}(\Ca_\nu(\A_{\lambda^+}))
\rightarrow \OCat_{\tilde{\nu}}(\A_{\lambda^+})$. An analog of the claim that $\mathsf{Stab}_\nu$
is an isomorphism after localization is that $\Delta_\nu$ lifts the equivalence
$\OCat_{\tilde{\nu}}(\Ca_\nu(\A_{\lambda^+}))
\xrightarrow{\sim} \gr\OCat_{\tilde{\nu}}(\A_{\lambda^+})$ (where the filtration is taken
with respect to $\nu$), see  Section \ref{SS_stand_exact}. Lemma \ref{Lem:parab_ind}
provides a transitivity result.

\subsection{R-matrices}
Pick  generic co-characters $\nu',\nu$.
Maulik and Okounkov, \cite[Section 4]{MO}, defined an automorphism $R_{\nu\rightarrow \nu'}$ of
$H^*_{T\times \C^\times}(X^{T})_{loc}$
as $\mathsf{Stab}_{\nu'}^{-1}\circ \mathsf{Stab}_{\nu}$ (here and below $\bullet_{loc}:=
\bullet\otimes_{\C[\mathfrak{t}\times \C]}\C(\mathfrak{t}\times \C)$).

We could also consider the automorphism $\tilde{R}_{\nu\rightarrow \nu'}:=\mathsf{Stab}_{\nu'}\circ
\mathsf{Stab}_\nu^{-1}$ of $H^*_{T\times \C^\times}(X)_{loc}$.
The transitivity result quoted in the last section implies that
if $\nu,\nu'\rightsquigarrow \nu_0$, then
\begin{equation}\label{R:transit}
\tilde{R}_{\nu\rightarrow \nu'}=\mathsf{Stab}_{\nu_0}\circ
\underline{\tilde{R}}_{\nu\rightarrow \nu'}\circ \mathsf{Stab}^{-1}_{\nu_0}.
\end{equation}

We would like to argue that the functor $\CW_{\nu\rightarrow \nu'}:D^b(\OCat_{\nu}(\A_{\lambda^+}))
\xrightarrow{\sim} D^b(\OCat_{\nu'}(\A_{\lambda^+}))$ should be viewed as a categorical
version of $\tilde{R}_{\nu\rightarrow \nu'}$. A part of our reasons for this is Proposition
\ref{Prop_CW_prop}. Part (1) of that proposition should be viewed as a categorical
analog of (\ref{R:transit}).
To interpret part (2) of that proposition we note that  using $\mathsf{Stab}_\nu$ we can define a basis
in $H^*_{T\times \C^\times}(X)_{loc}$ called the stable basis, it consists of the
images of the classes of the  fixed points. The map $\tilde{R}_{\nu\rightarrow \nu'}$
intertwines the stable bases for $\nu$ and $\nu'$. In particular, the stable basis
for $-\nu$ should be viewed as the opposite to $\nu$, and the map $\tilde{R}_{\nu\rightarrow -\nu}$
intertwines these two bases. Part (2) of Proposition \ref{Prop_CW_prop} should be
viewed as a categorical analog of this.

By its very definition, the map $\tilde{R}_{\nu\rightarrow \nu'}$ satisfies
the braid relations, compare to \cite[4.1.8]{MO}. A categorical analog of
this is Theorem \ref{Thm:braid_equiv}.

The map $\tilde{R}_{\nu\rightarrow \nu'}$ and the functor $\CW_{\nu\rightarrow \nu'}$
share some further similarities. For example, using \cite[Theorem 4.6.1]{MO}, one
can prove that $\tilde{R}_{\nu\rightarrow \nu'}$ commutes with the action of the
convolution algebra of the Steinberg variety $X\times_{X_0}X$. On the other hand,
it is not difficult to show that the functors $\CW_{\nu\rightarrow \nu'}$ commute
with the functors of taking derived tensor products with Harish-Chandra bimodules, compare to
\cite[Section 8.3]{BLPW}.

We do not know if the map $R_{\nu\rightarrow \nu'}$ itself has a categorical analog.

\end{document}